\documentclass[11pt]{amsart}
\usepackage{cmap}
\usepackage[utf8]{inputenc}
\usepackage[margin=1in]{geometry}
\usepackage{amsmath,amssymb,amsthm,bm,enumitem,graphicx,xcolor}
\usepackage[colorlinks=true,linkcolor=blue!50!black,citecolor=blue!50!black,urlcolor=blue!50!black,
  pdftitle={New explicit examples of Mannheim pair in L3},
  pdfauthor={Jose Alan, Renato Teixeira, Thiago Mendonca}]{hyperref}

\newtheorem{theorem}{Theorem}[section]
\newtheorem*{theorem A}{Theorem A}
\newtheorem*{theorem B}{N\"olker's Theorem}

\newtheorem{proposition}{Proposition}[section]
\newtheorem{corollary}{Corollary}[section]

\theoremstyle{remark}
\newtheorem{remark}{Remark}[section]
\theoremstyle{remark}

\theoremstyle{definition}
\newtheorem{definition}{Definition}[section]
\newtheorem{example}{Example}[section]
\numberwithin{equation}{section}

\newcommand{\kg}{{\kappa^{\smash[b]{T_{\beta}}}_{g}}}

\title{New explicit examples of Mannheim pair in $\mathbb{L}^3$}

\author{Jos\'e Alan}
\address{Universidade Federal Rural de Pernambuco (UFRPE), Departamento de Matem\'atica, Rua Manoel de Medeiros, s/n, Dois Irm\~aos, Recife, PE, 52171-900, Brazil}
\email{josealan.santos@ufrpe.br}

\author{Renato Teixeira}
\address{Universidade Federal Rural de Pernambuco (UFRPE), Departamento de Matem\'atica, Rua Manoel de Medeiros, s/n, Dois Irm\~aos, Recife, PE, 52171-900, Brazil}
\email{renato.teixeira@ufrpe.br}

\author{Thiago Mendon\c{c}a}
\address{Universidade Federal Rural de Pernambuco (UFRPE), Departamento de Matem\'atica, Rua Manoel de Medeiros, s/n, Dois Irm\~aos, Recife, PE, 52171-900, Brazil}
\email{thiago.mendonca@ufrpe.br}

\date{}

\begin{document}

\begin{abstract}
In this paper, we construct explicit examples of Mannheim pairs, in the classical sense, in Minkowski 3-space $\mathbb{L}^3$, one for each of the five types of such pairs. In our main examples the Mannheim curve is neither a helix nor a slant helix. For each example we give the complete Frenet apparatus of both curves of the pair and their Cartesian coordinates in terms of elementary functions. The construction writes the tangent indicatrix in angular coordinates and uses a suitable auxiliary angle function, which reduces all the relevant integrals to elementary form; it is carried out both starting from the Mannheim curve and starting from its partner. In these coordinates, we also characterize the helices and the slant helices with timelike axis. We also prove, for the five types, that a Mannheim curve is a slant helix if and only if its partner is a helix, and we give an explicit example of such a pair.
\end{abstract}

\maketitle

\section{Introduction}

The study of curves with special geometric properties is a central topic in classical differential geometry. Among these, Mannheim curves---introduced by A. Mannheim in 1878 \cite{Mannheim}---are defined by the existence of a partner curve such that, at corresponding points, the principal normal line of the first curve coincides with the binormal line of the second. If $\alpha$ and $\alpha^*$ are such curves, the pair $(\alpha,\alpha^*)$ is called a Mannheim pair.

In Euclidean space $\mathbb{E}^3$ it is well known \cite{Liu,OrbayKasap} that a curve $\alpha$ with curvature $\kappa>0$ and torsion $\tau$ is a Mannheim curve if and only if there exists a nonzero constant $\lambda$ such that
\[
\kappa=\lambda(\kappa^2+\tau^2).
\]
In $\mathbb{L}^3$ the relation depends on the causal character of the Frenet frame of $\alpha$: writing $\alpha^*=\alpha+\lambda N_\alpha$, a Mannheim curve $\alpha$ with $\kappa>0$ satisfies, for a nonzero constant $\lambda$, one of
\begin{align*}
\kappa&=\lambda(\kappa^2-\tau^2) && (B_\alpha\ \text{timelike}),\\
\kappa&=\lambda(\tau^2-\kappa^2) && (T_\alpha\ \text{timelike}),\\
\kappa&=-\lambda(\kappa^2+\tau^2) && (N_\alpha\ \text{timelike}),
\end{align*}
and a converse holds (Theorem~\ref{theorem caracterization mannheim pair}; see also \cite{Liu,Kahraman2011,ElsharkawyElshenhab}).

The literature contains several characterizations and examples of Mannheim curves. Liu and Wang \cite{Liu} established necessary and sufficient conditions for a curve to be a Mannheim partner in Euclidean and Minkowski 3-spaces. In $\mathbb{L}^3$, Kahraman, \"Onder, Kazaz and U\u{g}urlu \cite{Kahraman2011} distinguished five types of Mannheim pairs, according to the causal characters of the Frenet vectors of the two curves, and Elsharkawy and Elshenhab \cite{ElsharkawyElshenhab} studied the relations between the Frenet apparatus of the curves of each type. Orbay and Kasap \cite{OrbayKasap} derived relationships between the curvatures and torsions of the two curves of a Mannheim pair in $\mathbb{E}^3$. More recently, \c{C}ay and Yaylı \cite{Alper} studied slant helices that are Mannheim curves in $\mathbb{E}^3$ and gave an explicit example. In \cite{Thiago}, the correspondence between Bishop frames of Mannheim pairs was analyzed, and in \cite{TS,Thiago1} Mannheim curves were investigated in Heisenberg groups and in the special linear group, respectively. U\c{c}um, Camcı and \.{I}larslan \cite{Ucum} provided a new approach to Mannheim curves in $\mathbb{E}^3$, in which the vector joining corresponding points need not be parallel to the principal normal, including examples within the classes of Salkowski and anti-Salkowski curves. This generalized approach was developed in $\mathbb{L}^3$ for curves with $T_\alpha$ timelike by Ar\i c\i{} and \.{I}larslan \cite{AriciIlarslan}, who constructed explicit examples in which the Mannheim curve is a slant helix and its mate is a helix, while curves with $T_\alpha$ spacelike are treated with a new approach in \cite{Kecilioglu2026}. In this paper we always consider the classical definition, in which $\alpha^*=\alpha+\lambda N_\alpha$.

To the best of our knowledge, there is no explicit example in the literature of a classical Mannheim pair in $\mathbb{L}^3$ whose Mannheim curve is neither a helix nor a slant helix. This raises the natural question: \textit{Are there Mannheim curves in $\mathbb{L}^3$ that are neither helices nor slant helices, yet admit explicit parametrizations by elementary functions?}

In this paper, we answer this question affirmatively. We describe the Mannheim curves of each of the five types by means of the angular coordinates of the tangent indicatrix and an auxiliary angle function $\delta(t)$ (Theorem~\ref{thm:general_mannheim}), and we do the same starting from the partner curve, with an auxiliary function $\delta^*(x)$ (Theorem~\ref{thm:mannheim_partner}). By choosing these functions suitably, we obtain, in five examples (one for each type) and in one more example constructed from the partner curve, closed-form expressions for the coordinates of the curves, their Frenet vectors, curvatures and torsions. In all of these examples the Mannheim curve is neither a helix nor a slant helix. The main technical difficulty is to choose the angular functions so that all integrals defining the coordinates become elementary simultaneously; the choices made in the examples are designed to overcome this difficulty. In addition, we prove, for the five types, that a Mannheim curve is a slant helix if and only if its partner is a helix (Proposition~\ref{prop:slant_helix_L3}), a fact known in $\mathbb{E}^3$ \cite{Ucum} and, for curves with $T_\alpha$ timelike in $\mathbb{L}^3$ and in the generalized setting, in \cite{AriciIlarslan}, and we give an explicit example of such a pair (Example~\ref{Example Mannheim pair slanthelixXhelix}). We also characterize, in the angular coordinates, the helices and the slant helices with timelike axis (Corollary~\ref{cor:helix_slantL3}).

The paper is organized as follows. In Section~\ref{sec:prelim}, we recall the basic notions (Minkowski space, Frenet apparatus, helices and slant helices), obtain the Frenet apparatus in terms of the angular coordinates of the tangent indicatrix, characterize in these coordinates the helices and the slant helices with timelike axis, and collect the relations between the curves of a Mannheim pair for the five types, including the equivalence above and the example of a (slant helix, helix) pair. In Section~\ref{sec:mannheim_x}, we write the Mannheim pair equations with respect to a generic parameter. In Section~\ref{sec:choice_phi mannheim condition}, we construct the Mannheim curves from the auxiliary angle function and present the five examples, with illustrations. In Section~\ref{sec:mannheim_partner}, we construct the Mannheim partner curves in the same way and present one more example.

\section{Preliminaries}\label{sec:prelim}

\subsection{Elementary functions}
A function is called \emph{elementary} if it can be obtained by elementary operations (addition, multiplication, division, composition, inversion, etc.) starting from a basic set of functions (polynomial, trigonometric, exponential, logarithmic, hyperbolic, etc.), excluding, for instance, elliptic and hypergeometric functions. Risch \cite{Risch1969} developed a method to decide whether a function admits an elementary integral, later detailed by Geddes, Czapor and Labahn \cite{Geddes1992}. A simple strategy to circumvent a negative result is to seek a new variable with respect to which the function becomes elementarily integrable.

\begin{example}
$\displaystyle \int \sqrt{1+\cosh^2(s)}\,ds$ is an elliptic integral (non-elementary).
Consider $f(t)$ such that $\sqrt{1+\cosh^2(f(t))}\cdot f'(t)=1$. Then,
\[
\int \sqrt{1+\cosh^2(s)}\,ds = \int 1\,dt = t + \text{const}.
\]
\end{example}

\begin{example}
Consider the curve $\alpha:\mathbb{R}\longrightarrow \mathbb{R}^{3}$
$$\dfrac{d\alpha}{ds}(s)=\scriptstyle \left(\dfrac{\scriptstyle a\sinh\left(t(s)\right)}{\scriptstyle \sqrt{a^2\cosh(2t(s))+c^2}},\ \dfrac{a\cosh\left(t(s)\right)}{ \sqrt{a^2\cosh(2t(s))+c^2}},\ \dfrac{c}{\sqrt{a^2\cosh(2t(s))+c^2}}\right),$$
with $\alpha(0)=(a,0,0)$, and where $t(s)$ is the inverse of
$$s(t)=\displaystyle \int_{0}^{t}\sqrt{a^2\cosh(2*)+c^2}\,d*.$$
The coordinates of $\displaystyle \int \dfrac{d\alpha}{ds}(s)\,ds$ are not elementary. And we have
$$\dfrac{d(\alpha\circ s)}{dt}(t)=\left(a\sinh(t),\,a\cosh(t),\,c\right),$$
hence
$$(\alpha\circ s)(t)=\left(a\cosh(t),\,a\sinh(t),\,c\,t\right).$$
\end{example}

\subsection{Minkowski space}
$\mathbb{L}^3=(\mathbb{R}^3, \langle \ ,\ \rangle )$ denotes the Minkowski 3-space, that is, $\mathbb{R}^3$ with the inner product
$$\left\langle \cdot , \cdot\right\rangle = dx^2 + dy^2 - dz^2 \ .$$

The basic facts on Lorentzian vector spaces used below can be found, for instance, in \cite{ONeill1983,Lopez2014}.

\begin{definition}[Causal character of a vector]Let $v\in T_{p}\mathbb{L}^3$. We say

\begin{enumerate}
    \item v is \textbf{Spacelike} if $\langle v,v\rangle >0$ or $v=0$.
    \item v is \textbf{Timelike} if $\langle v,v\rangle <0$.
    \item v is \textbf{Lightlike} if $\langle v,v\rangle =0$  and $v\neq 0$.
\end{enumerate}

\end{definition}

\begin{remark}
In our approach, we exclude lightlike vectors.
\end{remark}

\begin{proposition}\label{ortogonal do timelike}
Let $v \in T_p\mathbb{L}^3$ be a timelike vector and let $u \in T_p\mathbb{L}^3$ be nonzero with $\langle u,v\rangle = 0$. Then $u$ is spacelike.
\end{proposition}

\begin{definition}\label{def:norm}
The \textit{norm} of a vector $v \in T_{p}\mathbb{L}^3$ is defined by
\begin{equation}
    \|v\| = \sqrt{|\langle v, v \rangle|}.
\end{equation}
\end{definition}

\begin{definition}
The \emph{de Sitter space} is the one-sheeted hyperboloid
\[
\mathbb{S}^2_1=\{x\in\mathbb{L}^3:\langle x,x\rangle=1\}.
\]
The set $\{x\in\mathbb{L}^3:\langle x,x\rangle=-1\}$ is a two-sheeted hyperboloid, namely $\mathbb{H}^2\cup(-\mathbb{H}^2)$, where
\[
\mathbb{H}^2=\{x\in\mathbb{L}^3:\langle x,x\rangle=-1,\ x_3>0\}
\]
is its upper sheet, the hyperboloid model of the \emph{hyperbolic plane}.
\end{definition}

\begin{definition}\label{def:orthonormal-basis}
A basis $\{E_1,E_2,E_3\}$ of $T_p\mathbb{L}^3$ is \textit{orthonormal} if
\begin{equation}
    \langle E_i,E_j\rangle=\varepsilon_i \cdot \delta_{ij}\quad ,\quad \forall \ i,j=1,2,3,
\end{equation}
where $\varepsilon_{i}\in\{-1,1\}$, $\delta_{ii}=1$, $\delta_{ij}=0$ if $i\neq j$.
\end{definition}

\begin{definition}\label{def:cross-product}
The \textit{cross product} $u \wedge v$ on $T_{p}\mathbb{L}^3$ is defined by
\begin{equation}
    \langle u \wedge v, w \rangle := \begin{vmatrix}
        u_{1} & v_{1} & w_{1} \\
        u_{2} & v_{2} & w_{2} \\
        u_{3} & v_{3} & w_{3}
    \end{vmatrix},
\end{equation}
where $v=v_{1}\cdot E_{1} + v_{2}\cdot E_{2}+v_{3}\cdot E_{3}= (v_{1},v_{2},v_{3})$ denotes $v$ written in an orthonormal basis $\{E_{1},E_{2},E_{3}\}$ of $T_{p}\mathbb{L}^3$.
\end{definition}

\begin{proposition} Let $u,v \in T_{p}\mathbb{L}^3$, then
\begin{enumerate}
\item $u\wedge v$ is antisymmetric.
\item \begin{equation}
u \wedge v=\left(u_{2}v_{3}-u_{3}v_{2},\ u_{3}v_{1}-u_{1}v_{3},\ u_{2}v_{1}-u_{1}v_{2} \right)\hat{=}\begin{vmatrix}
    \textbf{i}&\textbf{j}&-\textbf{k}
    \\
    u_{1}&u_{2}&u_{3}
     \\
    v_{1}&v_{2}&v_{3}
\end{vmatrix}
\end{equation}
    \item \begin{equation}
    \langle u \wedge v, u \wedge v\rangle =\langle u , v\rangle.\langle u , v\rangle-\langle u , u\rangle.\langle v , v\rangle
    \end{equation}
    \item $\{E_{1},E_{2}\}$ orthonormal: $\|E_{1}\|=\|E_{2}\|=1,\ \langle E_{1},E_{2}\rangle =0$. Then, $\{E_{1},E_{2},E_{3}\}$, with $E_{3}=E_{1}\wedge E_{2}$, is orthonormal, and
    \begin{equation}
\varepsilon_{1}.\varepsilon_{2}.\varepsilon_{3}=-1.\end{equation}

\end{enumerate}
\end{proposition}
\begin{remark} Note that
$\varepsilon_{1}=-\varepsilon_{2}.\varepsilon_{3},\ \varepsilon_{2}=-\varepsilon_{1}.\varepsilon_{3},\ \varepsilon_{3}=-\varepsilon_{1}.\varepsilon_{2}$.
\end{remark}

\begin{corollary}\label{condicao dos vaerpsilons}
Let $\{E_1, E_2, E_3\}$, with $E_3 = E_1 \wedge E_2$, be an orthonormal basis of $T_p\mathbb{L}^3$.
\begin{enumerate}
\item At least one of the $\varepsilon_i$ must be -1.
    \item  If $\varepsilon_i = -1$, then  $\varepsilon_j = \varepsilon_k = 1$, with $i \neq j$, $j \neq k$, $k \neq i$.
    \item $E_{2}\wedge E_{3}=-\varepsilon_{2}\cdot E_{1}$ and $E_{3}\wedge E_{1}=-\varepsilon_{1}\cdot E_{2}$.
\end{enumerate}
\end{corollary}

\begin{definition}
Let $\pi$ be a plane in $T_p\mathbb{L}^3$, such that $\{E_1, E_2\}$ be an orthonormal basis of $\pi$. The plane $\pi$ is called:
\begin{enumerate}
    \item \textbf{spacelike} if $\varepsilon_1 = \varepsilon_2 = +1$,
    \item \textbf{timelike} if $\varepsilon_1=+1 ,\varepsilon_2 = -1$ or $\varepsilon_1=-1 ,\varepsilon_2 = +1$.
\end{enumerate}
\end{definition}
\begin{remark}
From Proposition~\ref{ortogonal do timelike}, the case $\varepsilon_1 = \varepsilon_2 = -1$ cannot occur.
\end{remark}

\begin{definition}
Let $u, v \in T_p\mathbb{L}^3$ be non-null vectors. The angle between $u$ and $v$ is defined according to their causal character:
\begin{enumerate}
    \item[(a)] If $u$ and $v$ are timelike, the hyperbolic angle $\varphi_{1} \ge 0$ is given by
    \begin{equation}\label{angle type 1}
     \cosh \varphi_{1} = \frac{|\langle u, v \rangle|}{\sqrt{|\langle u, u \rangle \langle v, v \rangle|}}.
    \end{equation}
    \item[(b)] If one vector is spacelike and the other is timelike, the hyperbolic angle $\varphi_{2} \ge 0$ is given by
    \begin{equation}
    \label{angle type 2}
    \sinh \varphi_{2} = \frac{|\langle u, v \rangle|}{\sqrt{\langle u, u \rangle \, |\langle v, v \rangle|}}.
    \end{equation}
    \item[] If $u$ and $v$ are spacelike, we have two subcases:
    \begin{enumerate}[label=(\alph*),start=3]
        \item If they span a spacelike plane (i.e., $\langle u \wedge v, u \wedge v \rangle <0$), the real angle $\varphi_{3} \in [0,\pi]$ is given by
        \begin{equation}\label{angle type 3}
            \cos \varphi_{3} = \frac{\langle u, v \rangle}{\sqrt{\langle u, u \rangle \langle v, v \rangle}}.
        \end{equation}
        \item If they span a timelike plane (i.e., $\langle u \wedge v, u \wedge v \rangle >0$), the hyperbolic angle $\varphi_{4} \ge 0$ is given by
        \begin{equation}\label{angle type 5}
        \cosh \varphi_{4} = \frac{|\langle u, v \rangle|}{\sqrt{\langle u, u \rangle \langle v, v \rangle}}.
        \end{equation}
    \end{enumerate}
\end{enumerate}
\end{definition}

\subsection{Curves in Minkowski Space \texorpdfstring{$\mathbb{L}^3$}{L3}}

The \textit{prime} ``\ $\prime{}$\ ''  denotes derivative with respect to generic parameter $t$. And the \textit{dot} ``\ $\dot{}$\ ''  denotes derivative with respect to arc-length $s$.

\begin{definition}
Let \(\alpha: I \subset \mathbb{R} \to \mathbb{L}^3\) be a regular curve. If \(\|\dot{\alpha}(s)\| = 1\) for all \(s \in I\), then \(s\) is called the \textit{arc length parameter} of \(\alpha\) and \(\alpha\) is said to be \textit{parameterized by arc length}.
\end{definition}

\begin{definition}\label{definicao inicial aparato de frenet}
Let $\beta: J \subset \mathbb{R} \to \mathbb{L}^3$, $\beta(t)$, be a regular curve parametrized by an arbitrary parameter $t$, with $\|\beta'(t)\| \neq 0$ for all $t \in J$. The Frenet–Serret apparatus of $\beta$ consists of the following vector fields and scalar functions:
\begin{enumerate}
    \item \textbf{Tangent:}
    \begin{equation}\label{T x geral}
    T_{\beta}(t) = \dfrac{\beta'(t)}{\|\beta'(t)\|},
    \end{equation}
    \item \textbf{Curvature:}
    \begin{equation}\label{kappa x geral}
    \kappa_{\beta}(t) = \dfrac{\|\beta'(t) \wedge \beta''(t)\|}{\|\beta'(t)\|^3},
    \end{equation}
    If $\kappa_\beta(t)\neq 0$,

     \item \textbf{Principal normal:}
     \begin{equation}\label{N x geral}
     N_{\beta}(t) = \dfrac{T'_{\beta}(t)}{\|T'_{\beta}(t)\|}.  \end{equation}

    \item \textbf{Binormal:}
    \begin{equation}\label{binormal x geral}
     B_{\beta}(t) = T_{\beta}(t) \wedge N_{\beta}(t),
    \end{equation}

    \item \textbf{Torsion:}
    \begin{equation}\label{tau x geral}
    \tau_{\beta}(t) = \dfrac{\det\bigl(\beta'(t),\,\beta''(t),\,\beta'''(t)\bigr)}{\|\beta'(t) \wedge  \beta''(t)\|^2}.
    \end{equation}
\end{enumerate}
\end{definition}

\begin{definition}
Let $\beta: J \longrightarrow \mathbb{L}^3$, $\beta(t)$, be a regular curve and $t: I \longrightarrow J$, $t(v)$, a diffeomorphism. The curve $ \alpha= \beta \circ t: I \longrightarrow \mathbb{L}^3$, $ \alpha(v)= \beta (t(v))$, is called a \textit{reparametrization} of $\beta$ and $t(v)$ a \textit{change of parameter} of $\beta$. When $v=s$ is the arc length parameter of $\alpha$, $\alpha$ is called an \textit{arc length reparametrization} of $\beta$ and $t(s)$ is called a \textit{change of parameter by arc length} of $\beta$.
\end{definition}

\begin{proposition}Let $\alpha:I\longrightarrow \mathbb{L}^3$, $\alpha(v)=\beta(t(v))$, be a reparametrization of $\beta:J\longrightarrow \mathbb{L}^3$, $\beta(t)=\alpha(v(t))$, with $v=v(t)$ inverse of $t(v)$ and $\dfrac{dv}{dt}>0$. Then the Frenet apparatus of $\beta$ and $\alpha$ are related by:
\begin{enumerate}
    \item \textbf{Tangent at t:}
    \begin{equation}\label{tangente at t}
    T_{\beta}(t)=T_{\alpha}(v(t))
    \end{equation}
    \item \textbf{Curvature at t:}
    \begin{equation}\label{kappa at t}
    \kappa_{\beta}(t)=\kappa_{\alpha}(v(t))
    \end{equation}
    If $\kappa_\beta(t)\neq 0$,
    \item \textbf{Principal normal at t:}
    \begin{equation}\label{normal at t}
    N_{\beta}(t)=N_{\alpha}(v(t))
    \end{equation}
    \item \textbf{Binormal at t:}
    \begin{equation}\label{binormal at t}
    B_{\beta}(t)=B_{\alpha}(v(t))
    \end{equation}
    \item \textbf{Torsion at t:}
    \begin{equation}\label{torsion at t}
    \tau_{\beta}(t)=\tau_{\alpha}(v(t))
    \end{equation}
\end{enumerate}
\end{proposition}

\begin{definition}
Let $S$ be the de Sitter space $\mathbb S^{2}_{1}$ or the hyperbolic plane $\mathbb H^{2}$, and let $\gamma:I\to S$, $\gamma(t)$, be a regular curve with $\langle\gamma',\gamma'\rangle\neq0$. Its \emph{geodesic curvature} (with respect to the orientation $N(p)=p$) in $S$ is
\begin{equation}\label{geodesic curvature in S21}
\kappa^{\smash[b]{\gamma}}_{g}(t)=\frac{\det(\gamma(t),\gamma'(t),\gamma''(t))}{\|\gamma'(t)\|^{3}} .
\end{equation}
\end{definition}

\begin{remark}
The definition above is a particular case of the following one: for a regular curve $\gamma$ with $\langle\gamma',\gamma'\rangle\neq0$ on an oriented surface $S\subset\mathbb L^{3}$ with unit normal $N$,
\begin{equation}\label{geodesic curvature general L3}
\kappa^{\smash[b]{\gamma,S}}_{g}(t)=\frac{\langle N(\gamma(t))\wedge\gamma'(t),\,\gamma''(t)\rangle}{\|\gamma'(t)\|^{3}}.
\end{equation}

Taking $S=\mathbb S^{2}_{1}$ or $S=\mathbb H^{2}$ and $N(p)=p$ one obtains the definition above.
\end{remark}

\begin{proposition}
Let $\beta:I\to\mathbb L^{3}$ be a regular curve, $\beta(t)$, with $T_\beta$ non-lightlike and $\langle T_\beta',T_\beta'\rangle\neq0$. The geodesic curvature, in $\mathbb S^{2}_{1}$ (if $T_\beta$ is spacelike) or in $\mathbb H^{2}$ (if $T_\beta$ is timelike, with $T_\beta\in\mathbb H^{2}$), of its tangent indicatrix $T_{\beta}(t)$ is

\begin{equation}\label{geodesic curvature and quotient L3}
\kappa^{\smash[b]{T_{\beta}}}_{g}(t)=\frac{\tau_\beta(t)}{\kappa_\beta(t)}.
\end{equation}
\end{proposition}

\subsection{Helix and Slant helix}

The characterization of helices by the constancy of $\tau/\kappa$ was stated by Lancret \cite{Lancret1802} and first proved by Saint-Venant \cite{SaintVenant1845}; slant helices were introduced by Izumiya and Takeuchi \cite{IzumiyaTakeuchi2004} in $\mathbb{E}^3$ and studied in $\mathbb{L}^3$ in \cite{AliLopez}.

\begin{definition}[\cite{CoutoLymberopoulos2020}]
A curve $\beta(t)$ in $\mathbb{L}^3$ with non-zero curvature is called a \emph{generalized helix} if its tangent $T_{\beta}$ makes a constant angle with a fixed direction in space.
\end{definition}

\begin{definition}[\cite{AliLopez}]
A curve $\beta(t)$ in $\mathbb{L}^3$ is called a \textit{slant helix} if its normal $N_{\beta}$ makes a constant angle with a fixed direction in space.
\end{definition}

\begin{theorem}[{\cite{CoutoLymberopoulos2020}; \cite[Theorem~2.12]{Lopez2014}}]\label{helix characterization}
A curve $\beta(t)$ in $\mathbb{L}^3$ with curvature $\kappa_{\beta} \neq 0$ is a generalized \textit{helix} if and only if
\begin{equation}\label{generalhelix_equation}
\sigma_{0}(t) =\frac{\tau_{\beta}(t)}{\kappa_{\beta}(t)}= \text{constant}.
\end{equation}
\end{theorem}

\begin{theorem}[\cite{AliLopez}]\label{slant helix characterization}
Let $\alpha(s)$ be a curve in $\mathbb{L}^3$ parametrized by arc length $s$, with $\kappa_{\alpha}^2-\varepsilon_{2}\tau_{\alpha}^2\neq0$. Then $\alpha$ is a \textit{slant helix} if and only if
\begin{equation}\label{slanthelix_equation}
\dfrac{\kappa_{\alpha}(s)^2}{|\kappa_{\alpha}(s)^2-\varepsilon_{2}.\tau_{\alpha}(s)^2|^{3/2}}\cdot \dfrac{d}{ds}\left(\dfrac{\tau_{\alpha}(s)}{\kappa_{\alpha}(s)}\right)\ \text{is constant}.
\end{equation}

\end{theorem}

\subsection{Frenet equations in \texorpdfstring{$\mathbb{L}^3$}{L3}}

In this subsection we recall the Frenet equations of a curve in $\mathbb{L}^3$ parametrized by arc length.

\begin{theorem}[\cite{CoutoLymberopoulos2020,Lopez2014}]Let $\alpha: I \to \mathbb{L}^3$ be a curve parametrized by arc-length $s$, with curvature $\kappa_\alpha>0$ and torsion $\tau_\alpha$. Then its Frenet frame $\{T_\alpha,N_\alpha,B_\alpha\}$ satisfies the Frenet equations \eqref{eq:frenet}, which take the forms listed below according to the causal character of $T_\alpha$, $N_\alpha$ and $B_\alpha$:

\begin{equation}\label{eq:frenet}
\begin{aligned}
{\dot{T}_{\alpha}} &= \kappa_{\alpha}\cdot  N_{\alpha},\\
{\dot{N}_{\alpha}} &=\varepsilon_{3}\cdot \left(\kappa_{\alpha}\cdot T_{\alpha}+\tau_{\alpha}\cdot B_{\alpha}\right),\\
{\dot{B}_{\alpha}} &= -\varepsilon_{2}\tau_{\alpha}\cdot N_{\alpha},
\end{aligned}
\end{equation}
with $\varepsilon_{1}=\langle T_{\alpha},T_{\alpha} \rangle,\ \varepsilon_{2}=\langle N_{\alpha},N_{\alpha} \rangle,\  \varepsilon_{3}=\langle B_{\alpha},B_{\alpha} \rangle$, $N_{\alpha}\wedge B_{\alpha}=-\varepsilon_{2}\cdot T_{\alpha}$,  $B_{\alpha}\wedge T_{\alpha}=-\varepsilon_{1}\cdot N_{\alpha}$.

Conversely, given continuous functions $\kappa>0$ and $\tau$ on $I$, a point $p_0\in\mathbb{L}^3$, $s_0\in I$ and an orthonormal basis $\{T_0,N_0,B_0\}$ with $T_0\wedge N_0=B_0$, there exists a unique curve $\alpha:I\to\mathbb{L}^3$ parametrized by arc length, with curvature $\kappa$ and torsion $\tau$, such that $\alpha(s_0)=p_0$ and its Frenet frame at $s_0$ is $\{T_0,N_0,B_0\}$. If the frame is not prescribed, but only the causal character of $T_\alpha,N_\alpha,B_\alpha$, the curve is unique up to a proper isometry of $\mathbb{L}^3$.

Thus, the Frenet equations are the \textbf{compatibility equations} for the Frenet frame: an orthonormal moving frame $\{T,N,B\}$ with associated scalar functions $\kappa>0$ and $\tau$ is the Frenet frame of an arc-length parametrized curve if and only if it satisfies \eqref{eq:frenet} and its first vector is the derivative of the curve.
\end{theorem}

According to Corollary~\ref{condicao dos vaerpsilons}, we have 3 cases to consider:

\begin{enumerate}
    \item $T_{\alpha}$  timelike:

 $\varepsilon_{1}=-1,\ \varepsilon_{2}=1,\  \varepsilon_{3}=1$,\
     $T_{\alpha}\wedge N_{\alpha}=B_{\alpha}$\ ,\ $N_{\alpha}\wedge B_{\alpha}=-T_{\alpha}$\ ,\  $B_{\alpha}\wedge T_{\alpha}=N_{\alpha}$,
\begin{equation}\label{eq:frenet 1}
\begin{aligned}
{\dot{T}_{\alpha}} &= \kappa_{\alpha} N_{\alpha},\\
{\dot{N}_{\alpha}} &=\kappa_{\alpha} T_{\alpha} + \tau_{\alpha} B_{\alpha},\\
{\dot{B}_{\alpha}} &= -\tau_{\alpha}N_{\alpha}.
\end{aligned}
\end{equation}

\item $N_{\alpha}$ timelike:

$\varepsilon_{1}=1,\ \varepsilon_{2}=-1,\  \varepsilon_{3}=1$,\
     $T_{\alpha}\wedge N_{\alpha}=B_{\alpha}$,\ $N_{\alpha}\wedge B_{\alpha}=T_{\alpha}$,\  $B_{\alpha}\wedge T_{\alpha}=-N_{\alpha}$,
\begin{equation}\label{eq:frenet 2}
\begin{aligned}
{\dot{T}_{\alpha}} &= \kappa_{\alpha} N_{\alpha},\\
{\dot{N}_{\alpha}} &= \kappa_{\alpha} T_{\alpha} + \tau_{\alpha} B_{\alpha},\\
{\dot{B}_{\alpha}} &= \tau_{\alpha} N_{\alpha}.
\end{aligned}
\end{equation}

\item $B_{\alpha}$ timelike:

$\varepsilon_{1}=1,\ \varepsilon_{2}=1,\  \varepsilon_{3}=-1$,\
     $T_{\alpha}\wedge N_{\alpha}=B_{\alpha}$,\ $N_{\alpha}\wedge B_{\alpha}=-T_{\alpha}$,\  $B_{\alpha}\wedge T_{\alpha}=-N_{\alpha}$,
\begin{equation}\label{eq:frenet 3}
\begin{aligned}
  {\dot{T}_{\alpha}} &= \kappa_{\alpha} N_{\alpha},\\
{\dot{N}_{\alpha}} &= -\kappa_{\alpha} T_{\alpha} - \tau_{\alpha} B_{\alpha},\\
{\dot{B}_{\alpha}} &= -\tau_{\alpha} N_{\alpha}.
\end{aligned}
\end{equation}
\end{enumerate}

\subsection{Frenet apparatus in spherical coordinates with respect to a generic parameter}\label{sec:angular}

This subsection details the Frenet frame of a curve in $\mathbb{L}^3$ in terms of the angular coordinates of its tangent vector, starting from its arc-length parametrization and then considering a general parameter $t$, according to the causal character of $T$, $N$ and $B$.

Let $\alpha: I \subset \mathbb{R} \to \mathbb{L}^3$ be a smooth unit-speed curve with arc-length parameter $s$ and $\langle\dot{T}_\alpha,\dot{T}_\alpha\rangle\neq 0$. Since its unit tangent vector field $T_\alpha(s) = \dot{\alpha}(s)$ inherently satisfies $\langle T_\alpha(s),T_\alpha(s)\rangle = -1$ for every $s \in I$ if $T_{\alpha}$ is timelike, or $\langle T_\alpha(s),T_\alpha(s)\rangle = 1$ if $T_{\alpha}$ is spacelike, its trajectory is naturally constrained to the two-sheeted hyperboloid $\mathbb{H}^2\cup(-\mathbb{H}^2)$, where $\mathbb{H}^2=\{x\in\mathbb{L}^3:\langle x,x\rangle=-1,\ x_3>0\}$ is the hyperbolic plane (upper sheet), or to the de Sitter space $\mathbb{S}^2_1=\{x\in\mathbb{L}^3:\langle x,x\rangle=1\}$, respectively. Consequently, we can seamlessly parametrize $T_\alpha(s)$ in terms of two angular functions as follows:

\begin{equation}\label{Talpha1}
T_{\alpha}(s)=(\cos\theta(s)\tan\phi(s),\ \sin\theta(s)\tan\phi(s),\ \sec\phi(s)) \qquad (T_{\alpha} \text{ timelike}),
\end{equation}
where $\theta(s)\in[0,2\pi)$ and $\phi(s) \in \left(-\dfrac{\pi}{2},\dfrac{\pi}{2}\right)\cup\left(\dfrac{\pi}{2},\dfrac{3\pi}{2}\right)$, and
\begin{equation}\label{Talpha2}
T_{\alpha}(s)=(\cos\theta(s)\cosh\phi(s),\ \sin\theta(s)\cosh\phi(s),\ \sinh\phi(s)) \qquad (T_{\alpha} \text{ spacelike}),
\end{equation}
where $\theta(s)\in[0,2\pi)$ and $\phi(s)\in\mathbb{R}$, being $\theta$ and $\phi$ differentiable real-valued functions.

\begin{proposition}\label{thm:arbitrary_param}
Let $\beta(t) = \alpha(s(t))$ be a smooth, orientation-preserving reparametrization of $\alpha$ satisfying $\dfrac{ds}{dt}>0$. In order to express the Frenet apparatus explicitly in terms of the general parameter $t$, we define, for convenience, the composite angular functions as
\begin{equation}\label{eq:ThetaPhi}
\Theta(t) = \theta(s(t)), \qquad \Phi(t) = \phi(s(t)).
\end{equation}
Throughout the remainder of the text a prime ``\ $'$\ '' denotes differentiation with respect to $t$. The Frenet apparatus of $\beta$ is given by:

\begin{enumerate}
\item Case $T_\beta$ \textbf{timelike}:
\begin{equation}\label{eq:deriv_comp_1}
\beta'(t) = \frac{ds}{dt}\cdot \big( \cos\Theta\tan\Phi,\; \sin\Theta\tan\Phi,\; \sec\Phi \big),
\end{equation}
\begin{align}
T_\beta(t) &= \big( \cos\Theta\tan\Phi,\; \sin\Theta\tan\Phi,\; \sec\Phi \big), \label{eq:T_x_1}\\
N_\beta(t) &= \frac{\big( -\Theta'\sin\Theta\tan\Phi + \Phi'\cos\Theta\sec^2\Phi,\; \Theta'\cos\Theta\tan\Phi + \Phi'\sin\Theta\sec^2\Phi,\; \Phi'\sec\Phi\tan\Phi \big)}{\sqrt{\Theta'^2\tan^2\Phi + \Phi'^2\sec^2\Phi}}, \label{eq:N_x_1}\\
B_\beta(t) &= \frac{\big( -\sec\Phi\,(\Phi'\sin\Theta + \Theta'\tan\Phi\cos\Theta),\; \sec\Phi\,(\Phi'\cos\Theta - \Theta'\tan\Phi\sin\Theta),\; -\Theta'\tan^2\Phi \big)}{\sqrt{\Theta'^2\tan^2\Phi + \Phi'^2\sec^2\Phi}}, \label{eq:B_x_1}\\
\kappa_\beta(t) &= \frac{1}{\dfrac{ds}{dt}} \sqrt{\Theta'^2\tan^2\Phi + \Phi'^2\sec^2\Phi}, \label{eq:kappa_x_1}\\
\tau_\beta(t) &= \frac{1}{\dfrac{ds}{dt}} \left\{ \Theta'\sec\Phi- \left[ \arctan\left( \frac{\Phi^{\prime}}{\Theta^{\prime}\sin\Phi} \right) \right]^{{\prime}} \right\}. \label{eq:tau_x_1}
\end{align}

\item Case $N_\beta$ \textbf{timelike}:
\begin{equation}\label{eq:deriv_comp_2}
\beta'(t) = \frac{ds}{dt}\cdot \big( \cos\Theta\cosh\Phi,\; \sin\Theta\cosh\Phi,\; \sinh\Phi \big),
\end{equation}
\begin{align}
T_\beta(t) &= \big( \cos\Theta\cosh\Phi,\; \sin\Theta\cosh\Phi,\; \sinh\Phi \big), \label{eq:T_x_2}\\
N_\beta(t) &= \frac{\big( -\Theta'\sin\Theta\cosh\Phi + \Phi'\cos\Theta\sinh\Phi,\; \Theta'\cos\Theta\cosh\Phi + \Phi'\sin\Theta\sinh\Phi,\; \Phi'\cosh\Phi \big)}{\sqrt{\Phi'^2 - \Theta'^2\cosh^2\Phi}}, \label{eq:N_x_2}\\
B_\beta(t) &= \frac{\big( \Phi'\sin\Theta - \Theta'\cos\Theta\sinh\Phi\cosh\Phi,\; -\Phi'\cos\Theta - \Theta'\sin\Theta\sinh\Phi\cosh\Phi,\; -\Theta'\cosh^2\Phi \big)}{\sqrt{\Phi'^2 - \Theta'^2\cosh^2\Phi}}, \label{eq:B_x_2}\\
\kappa_\beta(t) &= \frac{1}{\dfrac{ds}{dt}} \sqrt{\Phi'^2 - \Theta'^2\cosh^2\Phi}, \label{eq:kappa_x_2}\\
\tau_\beta(t) &= -\frac{1}{\dfrac{ds}{dt}} \left\{ \Theta'\sinh\Phi+ \left[ \operatorname{arctanh}\left( \frac{\Theta^{\prime}\cosh\Phi}{\Phi^{\prime}} \right) \right]^{{\prime}} \right\}. \label{eq:tau_x_2}
\end{align}

\item Case $B_\beta$ \textbf{timelike}:
\begin{equation}\label{eq:deriv_comp_3}
\beta'(t) = \frac{ds}{dt}\cdot \big( \cos\Theta\cosh\Phi,\; \sin\Theta\cosh\Phi,\; \sinh\Phi \big),
\end{equation}
\begin{align}
T_\beta(t) &= \big( \cos\Theta\cosh\Phi,\; \sin\Theta\cosh\Phi,\; \sinh\Phi \big), \label{eq:T_x_3}\\
N_\beta(t) &= \frac{\big( -\Theta'\sin\Theta\cosh\Phi + \Phi'\cos\Theta\sinh\Phi,\; \Theta'\cos\Theta\cosh\Phi + \Phi'\sin\Theta\sinh\Phi,\; \Phi'\cosh\Phi \big)}{\sqrt{\Theta'^2\cosh^2\Phi - \Phi'^2}}, \label{eq:N_x_3}\\
B_\beta(t) &= \frac{\big( \Phi'\sin\Theta - \Theta'\cos\Theta\sinh\Phi\cosh\Phi,\; -\Phi'\cos\Theta - \Theta'\sin\Theta\sinh\Phi\cosh\Phi,\; -\Theta'\cosh^2\Phi \big)}{\sqrt{\Theta'^2\cosh^2\Phi - \Phi'^2}}, \label{eq:B_x_3}\\
\kappa_\beta(t) &= \frac{1}{\dfrac{ds}{dt}} \sqrt{\Theta'^2\cosh^2\Phi - \Phi'^2}, \label{eq:kappa_x_3}\\
\tau_\beta(t) &= \frac{1}{\dfrac{ds}{dt}} \left\{ \Theta'\sinh\Phi+ \left[ \operatorname{arctanh}\left( \frac{\Phi^{\prime}}{\Theta^{\prime}\cosh\Phi} \right) \right]^{{\prime}} \right\}. \label{eq:tau_x_3}
\end{align}
\end{enumerate}
\end{proposition}

\begin{proof} We prove the statements for each case separately.

\noindent\textbf{Case 1: $T_\beta$ timelike.} Let $\alpha(s)$ be a curve parametrized by arc-length. Denoting its tangent vector by \eqref{Talpha1}, we obtain
$$
\dot{T}_{\alpha}
= \bigl( -\dot{\theta}\sin\theta\tan\phi + \dot{\phi}\cos\theta\sec^2\phi,\; \dot{\theta}\cos\theta\tan\phi + \dot{\phi}\sin\theta\sec^2\phi,\; \dot{\phi}\sec\phi\tan\phi \bigr).
$$

Thus, by definition and simple calculations, we obtain the curvature, the normal vector and the binormal vector given by the equations:

\begin{equation}\label{curvaturaalpha1}
\kappa_{\alpha}(s) = \sqrt{\dot{\theta}^2\tan^2\phi + \dot{\phi}^2\sec^2\phi},
\end{equation}

\begin{equation}\label{normalalpha1}
\begin{split}
N_{\alpha}(s) ={}& \dfrac{1}{\sqrt{\dot{\theta}^2\tan^2\phi + \dot{\phi}^2\sec^2\phi}} \Big( -\dot{\theta}\sin\theta\tan\phi + \dot{\phi}\cos\theta\sec^2\phi,\\
&\ \dot{\theta}\cos\theta\tan\phi + \dot{\phi}\sin\theta\sec^2\phi,\; \dot{\phi}\sec\phi\tan\phi \Big),
\end{split}
\end{equation}

\begin{equation}\label{binormalalpha1}
\begin{split}
B_{\alpha}(s) ={}& \dfrac{1}{\sqrt{\dot{\theta}^2\tan^2\phi + \dot{\phi}^2\sec^2\phi}} \Big( -\sec\phi\,(\dot{\phi}\sin\theta + \dot{\theta}\tan\phi\cos\theta),\\
&\ \sec\phi\,(\dot{\phi}\cos\theta - \dot{\theta}\tan\phi\sin\theta),\; -\dot{\theta}\tan^2\phi \Big).
\end{split}
\end{equation}

Moreover, since $\tau_{\alpha}$ = $-\langle \dot{B}_{\alpha}\,, N_{\alpha}\rangle$ (or using the third coordinate of third Frenet Equation \eqref{eq:frenet 1}), we obtain

\begin{equation}\label{taualpha1}
\tau_{\alpha} = \dot{\theta}\sec\phi -\dot{\psi}, \ \ \mbox{with} \ \ \psi  =  \arctan\left( \frac{\dot{\phi}}{\dot{\theta}\sin\phi} \right).
\end{equation}

Since $\beta(t) = \alpha(s(t))$, by the chain rule we have

$$\beta'(t)= \dot{\alpha}(s(t))\,\,\dfrac{ds}{dt}= T_{\alpha}(s(t))\,\,\dfrac{ds}{dt}.$$

Thus, from equation \eqref{Talpha1}, we get \eqref{eq:deriv_comp_1}.

From equation \eqref{eq:ThetaPhi}, we obtain
$$\Theta^{\prime}(t)=\dot{\theta}(s(t)).\dfrac{ds}{dt}\ ,\  \Phi^{\prime}(t)=\dot{\phi}(s(t)).\dfrac{ds}{dt}.$$

Thus, from equation \eqref{curvaturaalpha1}
$$\kappa_\beta(t) =\dfrac{1}{\dfrac{ds}{dt}} \sqrt{\left[\dot{\theta}(s(t)).\dfrac{ds}{dt}\right]^2\tan^2\phi(s(t))+\left[\dot{\phi}(s(t)).\dfrac{ds}{dt}\right]^2\sec^2\phi(s(t))},$$
hence we obtain \eqref{eq:kappa_x_1}, and consequently \eqref{eq:N_x_1} and \eqref{eq:B_x_1}.

Finally,
$$\dfrac{1}{\dfrac{ds}{dt}}\cdot\Theta'(t)\sec\Phi(t)=\dfrac{1}{\dfrac{ds}{dt}}\cdot \dfrac{ds}{dt}(t)\cdot \dot{\theta}(s(t))\cdot \sec\phi(s(t))= \dot{\theta}(s(t))\cdot \sec\phi(s(t)).$$

And consider

$$f(s)=\arctan\left(\frac{\dot{\phi}(s)}{\dot{\theta}(s)\cdot\sin\phi(s)}\right).$$

Then
$$\dfrac{d}{dt}\left[ \arctan\left( \frac{\Phi'(t)}{\Theta'(t)\cdot\sin\Phi(t)} \right) \right]=\dfrac{d}{dt}\left[ \arctan\left( \frac{\dot{\phi}(s(t))\cdot \dfrac{ds}{dt}(t)}{\dot{\theta}(s(t))\cdot \dfrac{ds}{dt}(t)\cdot\sin\phi(s(t))} \right) \right]=$$
$$=\dfrac{d}{dt}\left[ \arctan\left( \frac{\dot{\phi}(s(t))}{\dot{\theta}(s(t))\cdot\sin\phi(s(t))} \right) \right]=\dfrac{d}{dt}\left[ f(s(t))\right]=\dot{f}(s(t))\cdot \dfrac{ds}{dt}(t).$$
Then
$$\dfrac{1}{\dfrac{ds}{dt}}\left\{\Theta'(t)\sec\Phi(t) -\dfrac{d}{dt}\left[ \arctan\left( \frac{\Phi'(t)}{\Theta'(t)\cdot\sin\Phi(t)} \right) \right]\right\}$$
$$=\dot{\theta}(s(t))\sec\phi(s(t)) -\dot{f}(s(t))=\tau_{\alpha}(s(t))=\tau_{\beta}(t),$$
which proves \eqref{eq:tau_x_1}.

\noindent\textbf{Case 2: $N_\beta$ timelike.} Let $\alpha(s)$ be a curve parametrized by arc-length. Denoting its tangent vector by \eqref{Talpha2}, we obtain
$$
\dot{T}_{\alpha}
= \bigl( -\dot{\theta}\sin\theta\cosh\phi + \dot{\phi}\cos\theta\sinh\phi,\; \dot{\theta}\cos\theta\cosh\phi + \dot{\phi}\sin\theta\sinh\phi,\; \dot{\phi}\cosh\phi \bigr).
$$

Thus, by definition and simple calculations, we obtain the curvature, the normal vector and the binormal vector given by the equations:

\begin{equation}\label{curvaturaalpha2}
\kappa_{\alpha}(s) = \sqrt{\dot{\phi}^2 - \dot{\theta}^2\cosh^2\phi},
\end{equation}

\begin{equation}\label{normalalpha2}
N_{\alpha}(s) = \dfrac{1}{\sqrt{\dot{\phi}^2 - \dot{\theta}^2\cosh^2\phi}} \Big( -\dot{\theta}\sin\theta\cosh\phi + \dot{\phi}\cos\theta\sinh\phi,\; \dot{\theta}\cos\theta\cosh\phi + \dot{\phi}\sin\theta\sinh\phi,\; \dot{\phi}\cosh\phi \Big),
\end{equation}

\begin{equation}\label{binormalalpha2}
B_{\alpha}(s) = \dfrac{1}{\sqrt{\dot{\phi}^2 - \dot{\theta}^2\cosh^2\phi}} \Big( \dot{\phi}\sin\theta - \dot{\theta}\cos\theta\sinh\phi\cosh\phi,\; -\dot{\phi}\cos\theta - \dot{\theta}\sin\theta\sinh\phi\cosh\phi,\; -\dot{\theta}\cosh^2\phi \Big).
\end{equation}

Moreover, since $\tau_{\alpha}$ = $-\langle \dot{B}_{\alpha}\,, N_{\alpha}\rangle$ (or using the third coordinate of third Frenet Equation \eqref{eq:frenet 2}), we obtain

\begin{equation}\label{taualpha2}
\tau_{\alpha} = -\dot{\theta}\sinh\phi -\dot{\psi}, \ \ \mbox{with} \ \ \psi  =  \operatorname{arctanh}\left( \frac{\dot{\theta}\cosh\phi}{\dot{\phi}} \right).
\end{equation}

Since $\beta(t) = \alpha(s(t))$, by the chain rule we have

$$\beta'(t)= \dot{\alpha}(s(t))\,\,\dfrac{ds}{dt}= T_{\alpha}(s(t))\,\,\dfrac{ds}{dt}.$$

Thus, from equation \eqref{Talpha2}, we get \eqref{eq:deriv_comp_2}.

From equation \eqref{eq:ThetaPhi}, we obtain
$$\Theta^{\prime}(t)=\dot{\theta}(s(t)).\dfrac{ds}{dt}\ ,\  \Phi^{\prime}(t)=\dot{\phi}(s(t)).\dfrac{ds}{dt}.$$

Thus, from equation \eqref{curvaturaalpha2}
$$\kappa_\beta(t) =\dfrac{1}{\dfrac{ds}{dt}} \sqrt{\left[\dot{\phi}(s(t)).\dfrac{ds}{dt}\right]^2-\left[\dot{\theta}(s(t)).\dfrac{ds}{dt}\right]^2\cosh^2\phi(s(t))},$$
hence we obtain \eqref{eq:kappa_x_2}, and consequently \eqref{eq:N_x_2} and \eqref{eq:B_x_2}.

Finally,
$$\dfrac{1}{\dfrac{ds}{dt}}\cdot\Theta'(t)\sinh\Phi(t)=\dfrac{1}{\dfrac{ds}{dt}}\cdot \dfrac{ds}{dt}(t)\cdot \dot{\theta}(s(t))\cdot \sinh\phi(s(t))= \dot{\theta}(s(t))\cdot \sinh\phi(s(t)).$$

And consider

$$f(s)=\operatorname{arctanh}\left(\frac{\dot{\theta}(s)\cdot\cosh\phi(s)}{\dot{\phi}(s)}\right).$$

Then
$$\dfrac{d}{dt}\left[ \operatorname{arctanh}\left( \frac{\Theta'(t)\cdot\cosh\Phi(t)}{\Phi'(t)} \right) \right]=\dfrac{d}{dt}\left[ \operatorname{arctanh}\left( \frac{\dot{\theta}(s(t))\cdot \dfrac{ds}{dt}(t)\cdot\cosh\phi(s(t))}{\dot{\phi}(s(t))\cdot \dfrac{ds}{dt}(t)} \right) \right]=$$
$$=\dfrac{d}{dt}\left[ \operatorname{arctanh}\left( \frac{\dot{\theta}(s(t))\cdot\cosh\phi(s(t))}{\dot{\phi}(s(t))} \right) \right]=\dfrac{d}{dt}\left[ f(s(t))\right]=\dot{f}(s(t))\cdot \dfrac{ds}{dt}(t).$$
Then
$$-\dfrac{1}{\dfrac{ds}{dt}}\left\{\Theta'(t)\sinh\Phi(t) +\dfrac{d}{dt}\left[ \operatorname{arctanh}\left( \frac{\Theta'(t)\cdot\cosh\Phi(t)}{\Phi'(t)} \right) \right]\right\}$$
$$=-\dot{\theta}(s(t))\sinh\phi(s(t)) -\dot{f}(s(t))=\tau_{\alpha}(s(t))=\tau_{\beta}(t),$$
which proves \eqref{eq:tau_x_2}.

\noindent\textbf{Case 3: $B_\beta$ timelike.} Let $\alpha(s)$ be a curve parametrized by arc-length. Denoting its tangent vector by \eqref{Talpha2}, we obtain
$$
\dot{T}_{\alpha}
= \bigl( -\dot{\theta}\sin\theta\cosh\phi + \dot{\phi}\cos\theta\sinh\phi,\; \dot{\theta}\cos\theta\cosh\phi + \dot{\phi}\sin\theta\sinh\phi,\; \dot{\phi}\cosh\phi \bigr).
$$

Thus, by definition and simple calculations, we obtain the curvature, the normal vector and the binormal vector given by the equations:

\begin{equation}\label{curvaturaalpha3}
\kappa_{\alpha}(s) = \sqrt{\dot{\theta}^2\cosh^2\phi - \dot{\phi}^2},
\end{equation}

\begin{equation}\label{normalalpha3}
N_{\alpha}(s) = \dfrac{1}{\sqrt{\dot{\theta}^2\cosh^2\phi - \dot{\phi}^2}} \Big( -\dot{\theta}\sin\theta\cosh\phi + \dot{\phi}\cos\theta\sinh\phi,\; \dot{\theta}\cos\theta\cosh\phi + \dot{\phi}\sin\theta\sinh\phi,\; \dot{\phi}\cosh\phi \Big),
\end{equation}

\begin{equation}\label{binormalalpha3}
B_{\alpha}(s) = \dfrac{1}{\sqrt{\dot{\theta}^2\cosh^2\phi - \dot{\phi}^2}} \Big( \dot{\phi}\sin\theta - \dot{\theta}\cos\theta\sinh\phi\cosh\phi,\; -\dot{\phi}\cos\theta - \dot{\theta}\sin\theta\sinh\phi\cosh\phi,\; -\dot{\theta}\cosh^2\phi \Big).
\end{equation}

Moreover, since $\tau_{\alpha}$ = $-\langle \dot{B}_{\alpha}\,, N_{\alpha}\rangle$ (or using the third coordinate of third Frenet Equation \eqref{eq:frenet 3}), we obtain

\begin{equation}\label{taualpha3}
\tau_{\alpha} = \dot{\theta}\sinh\phi +\dot{\psi}, \ \ \mbox{with} \ \ \psi  =  \operatorname{arctanh}\left( \frac{\dot{\phi}}{\dot{\theta}\cosh\phi} \right).
\end{equation}

Since $\beta(t) = \alpha(s(t))$, by the chain rule we have

$$\beta'(t)= \dot{\alpha}(s(t))\,\,\dfrac{ds}{dt}= T_{\alpha}(s(t))\,\,\dfrac{ds}{dt}.$$

Thus, from equation \eqref{Talpha2}, we get \eqref{eq:deriv_comp_3}.

From equation \eqref{eq:ThetaPhi}, we obtain
$$\Theta^{\prime}(t)=\dot{\theta}(s(t)).\dfrac{ds}{dt}\ ,\  \Phi^{\prime}(t)=\dot{\phi}(s(t)).\dfrac{ds}{dt}.$$

Thus, from equation \eqref{curvaturaalpha3}
$$\kappa_\beta(t) =\dfrac{1}{\dfrac{ds}{dt}} \sqrt{\left[\dot{\theta}(s(t)).\dfrac{ds}{dt}\right]^2\cosh^2\phi(s(t))-\left[\dot{\phi}(s(t)).\dfrac{ds}{dt}\right]^2},$$
hence we obtain \eqref{eq:kappa_x_3}, and consequently \eqref{eq:N_x_3} and \eqref{eq:B_x_3}.

Finally,
$$\dfrac{1}{\dfrac{ds}{dt}}\cdot\Theta'(t)\sinh\Phi(t)=\dfrac{1}{\dfrac{ds}{dt}}\cdot \dfrac{ds}{dt}(t)\cdot \dot{\theta}(s(t))\cdot \sinh\phi(s(t))= \dot{\theta}(s(t))\cdot \sinh\phi(s(t)).$$

And consider

$$f(s)=\operatorname{arctanh}\left(\frac{\dot{\phi}(s)}{\dot{\theta}(s)\cdot\cosh\phi(s)}\right).$$

Then
$$\dfrac{d}{dt}\left[ \operatorname{arctanh}\left( \frac{\Phi'(t)}{\Theta'(t)\cdot\cosh\Phi(t)} \right) \right]=\dfrac{d}{dt}\left[ \operatorname{arctanh}\left( \frac{\dot{\phi}(s(t))\cdot \dfrac{ds}{dt}(t)}{\dot{\theta}(s(t))\cdot \dfrac{ds}{dt}(t)\cdot\cosh\phi(s(t))} \right) \right]=$$
$$=\dfrac{d}{dt}\left[ \operatorname{arctanh}\left( \frac{\dot{\phi}(s(t))}{\dot{\theta}(s(t))\cdot\cosh\phi(s(t))} \right) \right]=\dfrac{d}{dt}\left[ f(s(t))\right]=\dot{f}(s(t))\cdot \dfrac{ds}{dt}(t).$$
Then
$$\dfrac{1}{\dfrac{ds}{dt}}\left\{\Theta'(t)\sinh\Phi(t) +\dfrac{d}{dt}\left[ \operatorname{arctanh}\left( \frac{\Phi'(t)}{\Theta'(t)\cdot\cosh\Phi(t)} \right) \right]\right\}$$
$$=\dot{\theta}(s(t))\sinh\phi(s(t)) +\dot{f}(s(t))=\tau_{\alpha}(s(t))=\tau_{\beta}(t),$$
which proves \eqref{eq:tau_x_3}.
\end{proof}

By a coordinate system we mean a proper orthochronous Lorentz transformation of the axes: $\kappa_\beta$, $\tau_\beta$ and the Frenet frame (up to the transformation) do not change, and $(\Theta,\Phi)$ are the angular coordinates of $T_\beta$ in that system. In Case~1 this requires $T_\beta\neq\pm e_3$ along $\beta$ (so that $\Theta$ is defined and $\tan\Phi\neq0$), and we assume $T_\beta$ future pointing, i.e.\ $\Phi\in(-\frac\pi2,\frac\pi2)$ (otherwise reverse the orientation of $\beta$, which changes $T_\beta,B_\beta$ into $-T_\beta,-B_\beta$ and keeps $N_\beta,\kappa_\beta,\tau_\beta$). We write $D=\dfrac{ds}{dt}>0$, we assume that $\Theta$ and $\Phi$ are of class $C^2$ on an open interval $J$ and that $\langle T_\beta',T_\beta'\rangle\neq0$, i.e.\ $\kappa_\beta>0$; in particular, a derivative that does not vanish has constant sign. Cases~1, 2 and~3 are those of Proposition~\ref{thm:arbitrary_param}.

\begin{remark}[Straight lines and planar curves]\label{rem:planarL3}
Recall that, if $\kappa_\beta>0$, $\beta$ is planar if and only if $\tau_\beta\equiv0$, i.e.\ if and only if $B_\beta$ is constant; the plane is orthogonal to $B_\beta$, hence spacelike in Case~3 ($B_\beta$ timelike) and timelike in Cases~1 and~2 ($B_\beta$ spacelike).
\begin{enumerate}
\item[(0)] \textbf{Straight lines.} $T_\beta'\equiv0$ if and only if $\beta$ is a straight line, if and only if $\Theta$ and $\Phi$ are both constant. In this case $\langle T_\beta',T_\beta'\rangle=0$ and $N_\beta$, $B_\beta$, $\tau_\beta$ are not defined. Items (a), (b) and Corollary~\ref{cor:helix_slantL3} assume $\kappa_\beta>0$.
\item[(a)] (Case~3) $\beta$ is planar if and only if there is a coordinate system in which $\Phi\equiv0$. In this case $\kappa_\beta=|\Theta'|/D$, $N_\beta=\operatorname{sgn}(\Theta')(-\sin\Theta,\cos\Theta,0)$, $B_\beta=-\operatorname{sgn}(\Theta')(0,0,1)$ and $\tau_\beta\equiv0$.
\item[(b)] (Cases~1 and~2) If $\Theta\equiv\Theta_0$ is constant in some coordinate system, then $\beta$ is planar (in a timelike plane), $B_\beta$ is constant, $\tau_\beta\equiv0$, and
\begin{align*}
\text{Case 1:}\quad & \kappa_\beta=\frac{\sec\Phi\,|\Phi'|}{D},\qquad B_\beta=\operatorname{sgn}(\Phi')\,(-\sin\Theta_0,\cos\Theta_0,0),\\
& N_\beta=\operatorname{sgn}(\Phi')\,(\cos\Theta_0\sec\Phi,\ \sin\Theta_0\sec\Phi,\ \tan\Phi),\\
\text{Case 2:}\quad & \kappa_\beta=\frac{|\Phi'|}{D},\qquad B_\beta=\operatorname{sgn}(\Phi')\,(\sin\Theta_0,-\cos\Theta_0,0),\\
& N_\beta=\operatorname{sgn}(\Phi')\,(\cos\Theta_0\sinh\Phi,\ \sin\Theta_0\sinh\Phi,\ \cosh\Phi).
\end{align*}
Conversely, if $\beta$ is planar with $B_\beta$ spacelike and there is a timelike unit vector $w$ parallel to its plane with $\pm w\notin T_\beta(J)$, then $\Theta$ is constant in every coordinate system with $e_3=w$.
\end{enumerate}
\end{remark}

\begin{proof}
(0) In each case $T_\beta'=\Theta' A+\Phi' C$, where $A=\partial_\Theta T_\beta$ and $C=\partial_\Phi T_\beta$ are linearly independent: $A\neq0$ has third coordinate $0$, while $C$ has third coordinate $\sec\Phi\tan\Phi\neq0$ (Case~1) or $\cosh\Phi\neq0$ (Cases~2, 3). Hence $T_\beta'=0$ if and only if $\Theta'=\Phi'=0$, i.e.\ $T_\beta$ is constant, and then $\beta(t)=\beta(t_0)+\big(\int_{t_0}^t D\big)T_\beta$ is a straight line; the converse is clear.

(a) If $\beta$ lies in a spacelike plane, its unit normal $v$ is timelike, and we may take $e_3=v$ (future pointing); then $0=\langle T_\beta,e_3\rangle=-\sinh\Phi$, so $\Phi\equiv0$. Conversely, if $\Phi\equiv0$, then $(\beta_3)'=D\sinh\Phi=0$ and $\beta$ lies in the plane $z=\text{const}$. For $\Phi\equiv0$ the formulas follow from Case~3 (here $\Theta'\neq0$ because $\kappa_\beta>0$), and $\tau_\beta=0$ from \eqref{eq:tau_x_3}, since $\Phi'=0$ and $\sinh\Phi=0$.

(b) Let $v_0=(-\sin\Theta_0,\cos\Theta_0,0)$ and let $r=\tan\Phi$ (Case~1) or $r=\cosh\Phi$ (Case~2). Then $\langle\beta',v_0\rangle=Dr(-\cos\Theta_0\sin\Theta_0+\sin\Theta_0\cos\Theta_0)=0$, so $\langle\beta,v_0\rangle$ is constant. The formulas follow from Cases~1 and~2 with $\Theta'=0$ (note $\Phi'\neq0$, since $\kappa_\beta>0$); $B_\beta$ is constant, hence $\tau_\beta=0$ (formula \eqref{eq:tau_x_1} does not apply, since $\Theta'=0$). There are no curves of Case~3 with $\Theta$ constant, since then $\langle T_\beta',T_\beta'\rangle=\Theta'^2\cosh^2\Phi-\Phi'^2=-\Phi'^2<0$. Conversely, let $e_3=w$. The unit normal $v$ of the plane is spacelike and orthogonal to $w$, so $v=(-\sin a,\cos a,0)$ and $0=\langle T_\beta,v\rangle=r\sin(\Theta-a)$. Now $r=\tan\Phi\neq0$ in Case~1 (because $T_\beta\neq\pm e_3$) and $r=\cosh\Phi\ge1$ in Case~2; hence $\Theta$ is a continuous function with values in $a+\pi\mathbb{Z}$, thus constant.
\end{proof}

\begin{corollary}[Helices and slant helices with timelike axis]\label{cor:helix_slantL3}
Let $\beta$ be as in Proposition~\ref{thm:arbitrary_param}, with $\kappa_\beta>0$. Recall that $\beta$ is a \emph{helix} (resp.\ \emph{slant helix}) if $\langle T_\beta,U\rangle$ (resp.\ $\langle N_\beta,U\rangle$) is constant for some fixed vector $U\neq0$ (the axis). We consider timelike axes: the spacelike and lightlike axes are not captured by the coordinates $(\Theta,\Phi)$, whose polar axis $e_3$ is timelike.
\begin{enumerate}
\item $\beta$ is a non-planar helix with timelike axis if and only if there is a coordinate system in which $\Phi\equiv\Phi_0$ with $\Phi_0\neq0$. This only occurs in Cases~1 and~3, $\Theta'\neq0$, and
\begin{align*}
\text{Case 1:}\quad & \kappa_\beta=\frac{|\tan\Phi_0|\,|\Theta'|}{D},\quad \tau_\beta=\frac{\sec\Phi_0\,\Theta'}{D},\quad \frac{\tau_\beta}{\kappa_\beta}=\frac{\operatorname{sgn}(\Theta')}{|\sin\Phi_0|},\\
\text{Case 3:}\quad & \kappa_\beta=\frac{\cosh\Phi_0\,|\Theta'|}{D},\quad \tau_\beta=\frac{\sinh\Phi_0\,\Theta'}{D},\quad \frac{\tau_\beta}{\kappa_\beta}=\operatorname{sgn}(\Theta')\tanh\Phi_0 .
\end{align*}
The case $\Phi_0=0$ is the planar case of Remark~\ref{rem:planarL3}(a).
\item $\beta$ is a slant helix with timelike axis and not a helix if and only if there is a coordinate system in which the third coordinate of $N_\beta$ is a nonzero constant $1/m$; equivalently, $\Phi'\neq0$ and
\begin{align*}
\text{Case 1:}\quad & \operatorname{sgn}m=\operatorname{sgn}(\Phi'\tan\Phi),\quad \Theta'^2=\big(m^2-\cot^2\Phi\big)\sec^2\Phi\,\Phi'^2,\\
\text{Case 2:}\quad & \operatorname{sgn}m=\operatorname{sgn}\Phi',\quad \Theta'^2=\big(\operatorname{sech}^2\Phi-m^2\big)\Phi'^2,\\
\text{Case 3:}\quad & \operatorname{sgn}m=\operatorname{sgn}\Phi',\quad \Theta'^2=\big(m^2+\operatorname{sech}^2\Phi\big)\Phi'^2 .
\end{align*}
In this case $\Theta'\neq0$ and, with $\varepsilon=\operatorname{sgn}(\Theta'\Phi')$,
\begin{align*}
\text{Case 1}\ (m^2\tan^2\Phi>1):\quad & \kappa_\beta=\frac{|m|\sec\Phi\,|\tan\Phi|\,|\Phi'|}{D},\\
& \tau_\beta=\frac{\varepsilon\,m^2\sec^2\Phi\,|\tan\Phi|\,\Phi'}{D\sqrt{m^2\tan^2\Phi-1}},\\
\text{Case 2}\ (m^2\cosh^2\Phi<1,\ |m|<1):\quad & \kappa_\beta=\frac{|m|\cosh\Phi\,|\Phi'|}{D},\\
& \tau_\beta=\frac{\varepsilon\,m^2\sinh\Phi\cosh\Phi\,\Phi'}{D\sqrt{1-m^2\cosh^2\Phi}},\\
\text{Case 3:}\quad & \kappa_\beta=\frac{|m|\cosh\Phi\,|\Phi'|}{D},\\
& \tau_\beta=\frac{\varepsilon\,m^2\sinh\Phi\cosh\Phi\,\Phi'}{D\sqrt{1+m^2\cosh^2\Phi}} .
\end{align*}
\end{enumerate}
\end{corollary}

\begin{proof}
(1) ($\Rightarrow$) Let $\langle T_\beta,U\rangle=c$ with $U$ unit timelike, future pointing. Then $T_\beta'\perp U$, and the orthogonal complement of a timelike vector is spacelike, so $T_\beta'$ is spacelike: $N_\beta$ is spacelike and Case~2 does not occur. Take $e_3=U$ (admissible: in Case~1, $|c|>1$, since $|c|=1$ would give $T_\beta=\pm U$, constant). Then $T_{\beta,3}=-c$: in Case~3, $\sinh\Phi=-c$, so $\Phi\equiv\Phi_0$, with $\Phi_0\neq0$ because $\Phi_0=0$ is the planar case of Remark~\ref{rem:planarL3}(a); in Case~1, $\sec\Phi=-c=|c|>1$, so $\Phi\in\{\pm\arccos(1/|c|)\}$ and, by continuity, $\Phi\equiv\Phi_0\neq0$. ($\Leftarrow$) If $\Phi\equiv\Phi_0$, then $\langle T_\beta,e_3\rangle$ is constant, so $\beta$ is a helix. Since $\Phi'=0$ is incompatible with Case~2, we are in Case~1 or~3. By \eqref{eq:kappa_x_1}, \eqref{eq:kappa_x_3}, $\kappa_\beta>0$ gives $\Theta'\neq0$; by \eqref{eq:tau_x_1}, \eqref{eq:tau_x_3} (with $\Phi'=0$), $\tau_\beta=\sec\Phi_0\,\Theta'/D$ in Case~1 and $\tau_\beta=\sinh\Phi_0\,\Theta'/D$ in Case~3, which are nonzero; so $\beta$ is non-planar.

(2) ($\Rightarrow$) Let $\langle N_\beta,U\rangle=d$ with $U$ unit timelike, future pointing. If $d=0$, then $\langle T_\beta',U\rangle=0$, so $\langle T_\beta,U\rangle$ is constant and $\beta$ is a helix; hence $d\neq0$. Take $e_3=U$; this is admissible, since $T_\beta=\pm e_3$ would give $\langle N_\beta,e_3\rangle=0$. Let $m=-1/d$, so that $N_{\beta,3}=1/m$. By \eqref{eq:N_x_1}, \eqref{eq:N_x_2}, \eqref{eq:N_x_3}, $\Phi'\neq0$, and squaring:
\begin{align*}
\text{Case 1:}\quad & m^2\Phi'^2\sec^2\Phi\tan^2\Phi=\Theta'^2\tan^2\Phi+\Phi'^2\sec^2\Phi,\\
\text{Case 2:}\quad & \Phi'^2-\Theta'^2\cosh^2\Phi=m^2\Phi'^2\cosh^2\Phi,\\
\text{Case 3:}\quad & \Theta'^2\cosh^2\Phi-\Phi'^2=m^2\Phi'^2\cosh^2\Phi,
\end{align*}
which are the stated relations for $\Theta'^2$ (in Case~1 we use $\tan\Phi\neq0$); the sign of $m$ is the sign of $N_{\beta,3}$. Conversely, if $\Phi'\neq0$, the sign condition holds and $\Theta'^2$ is as stated, the same identities give $\sqrt{\Theta'^2\tan^2\Phi+\Phi'^2\sec^2\Phi}=|m|\sec\Phi\,|\tan\Phi|\,|\Phi'|$ (Case~1) and $\sqrt{|\Phi'^2-\Theta'^2\cosh^2\Phi|}=|m|\cosh\Phi\,|\Phi'|$ (Cases~2, 3), hence $N_{\beta,3}=1/m$. This also gives the formulas for $\kappa_\beta$.

\emph{$\Theta'\neq0$ and the bounds.} Case~3: $\Theta'^2=(m^2+\operatorname{sech}^2\Phi)\Phi'^2>0$. Case~2: $g=1-m^2\cosh^2\Phi=\Theta'^2\cosh^2\Phi/\Phi'^2\ge0$; if $g(t_0)=0$, then $g'(t_0)=-2m^2\cosh\Phi\sinh\Phi\,\Phi'=0$ forces $\Phi(t_0)=0$ and $m^2=1$, but then $g=1-\cosh^2\Phi<0$ for $\Phi\neq0$ near $t_0$, which is impossible; so $g>0$, $\Theta'\neq0$ and $m^2\cosh^2\Phi<1$. Case~1: $g=m^2\tan^2\Phi-1=\Theta'^2\tan^2\Phi/(\Phi'^2\sec^2\Phi)\ge0$; if $g(t_0)=0$, then $g'(t_0)=2m^2\tan\Phi\sec^2\Phi\,\Phi'\neq0$, which is impossible for $g\ge0$; so $g>0$ and $\Theta'\neq0$.

\emph{The torsion.} Put $\varepsilon=\operatorname{sgn}(\Theta'\Phi')$.

\emph{Case 3.} Let $u=\sqrt{1+m^2\cosh^2\Phi}$. Then $\Theta'=\varepsilon\Phi'u/\cosh\Phi$, $\dfrac{\Phi'}{\Theta'\cosh\Phi}=\dfrac\varepsilon u$ and $u'=m^2\cosh\Phi\sinh\Phi\,\Phi'/u$, so
\[
\Big[\operatorname{arctanh}\frac\varepsilon u\Big]'=-\frac{\varepsilon u'}{u^2-1}=-\frac{\varepsilon\tanh\Phi\,\Phi'}{u},\qquad \Theta'\sinh\Phi=\varepsilon\Phi'\tanh\Phi\,u,
\]
and \eqref{eq:tau_x_3} gives $D\tau_\beta=\varepsilon\Phi'\tanh\Phi\big(u-\tfrac1u\big)=\dfrac{\varepsilon m^2\sinh\Phi\cosh\Phi\,\Phi'}{u}$.

\emph{Case 2.} Let $v=\sqrt{1-m^2\cosh^2\Phi}>0$. Then $\Theta'=\varepsilon\Phi'v/\cosh\Phi$, $\dfrac{\Theta'\cosh\Phi}{\Phi'}=\varepsilon v$ and $v'=-m^2\cosh\Phi\sinh\Phi\,\Phi'/v$, so
\[
\big[\operatorname{arctanh}(\varepsilon v)\big]'=\frac{\varepsilon v'}{1-v^2}=-\frac{\varepsilon\tanh\Phi\,\Phi'}{v},\qquad \Theta'\sinh\Phi=\varepsilon\Phi'\tanh\Phi\,v,
\]
and \eqref{eq:tau_x_2} gives $D\tau_\beta=-\varepsilon\Phi'\tanh\Phi\big(v-\tfrac1v\big)=\dfrac{\varepsilon m^2\sinh\Phi\cosh\Phi\,\Phi'}{v}$.

\emph{Case 1.} Let $u=\sqrt{m^2\tan^2\Phi-1}>0$ and $\delta=\operatorname{sgn}(\tan\Phi)$ (constant, since $\tan\Phi\neq0$). Then $\Theta'=\varepsilon\Phi'\sec\Phi\,u/|\tan\Phi|$, $\dfrac{\Phi'}{\Theta'\sin\Phi}=\dfrac{\varepsilon\delta}{u}$ and $u'=m^2\tan\Phi\sec^2\Phi\,\Phi'/u$, so
\[
\Big[\arctan\frac{\varepsilon\delta}u\Big]'=-\frac{\varepsilon\delta u'}{1+u^2}=-\frac{\varepsilon\sec^2\Phi\,\Phi'}{u\,|\tan\Phi|},\qquad \Theta'\sec\Phi=\frac{\varepsilon\Phi'\sec^2\Phi\,u}{|\tan\Phi|},
\]
and \eqref{eq:tau_x_1} gives $D\tau_\beta=\dfrac{\varepsilon\Phi'\sec^2\Phi}{|\tan\Phi|}\Big(u+\dfrac1u\Big)=\dfrac{\varepsilon m^2\sec^2\Phi\,|\tan\Phi|\,\Phi'}{u}$.

\emph{Not a helix.} By the above, $\dfrac{\tau_\beta}{\kappa_\beta}=\operatorname{sgn}(\Theta')\,|m|\,q$, with $q=\dfrac{\sinh\Phi}{\sqrt{1+m^2\cosh^2\Phi}}$ (Case~3), $q=\dfrac{\sinh\Phi}{\sqrt{1-m^2\cosh^2\Phi}}$ (Case~2) and $q=\dfrac{\sec\Phi}{\sqrt{m^2\tan^2\Phi-1}}$ (Case~1). A direct computation gives
\[
q'=\frac{(1+m^2)\cosh\Phi\,\Phi'}{(1+m^2\cosh^2\Phi)^{3/2}},\qquad
q'=\frac{(1-m^2)\cosh\Phi\,\Phi'}{(1-m^2\cosh^2\Phi)^{3/2}},\qquad
q'=-\frac{(1+m^2)\sec\Phi\tan\Phi\,\Phi'}{(m^2\tan^2\Phi-1)^{3/2}},
\]
respectively, which never vanishes (in Case~2 because $|m|<1$). So $\tau_\beta/\kappa_\beta$ is not constant, and $\beta$ is not a helix, since a helix, whatever its axis, has $\tau_\beta/\kappa_\beta$ constant (Lancret's theorem in $\mathbb{L}^3$, \cite[Theorem~2.12]{Lopez2014}).
\end{proof}

\subsection{Mannheim pair}

\begin{definition}\label{def:mannheim}
A curve $\alpha(s)$ is called a \emph{Mannheim curve} if there exist a differentiable curve $\alpha^*(s^*)$ and a reparametrization $s^*=s^*(s)$ such that, at corresponding points, the principal normal line of $\alpha$ coincides with the binormal line of $\alpha^*$. The curve $\alpha^*$ is called the \emph{Mannheim partner} of $\alpha$, and $(\alpha,\alpha^*)$ is called a \emph{Mannheim pair}. Equivalently,
\begin{equation}
    \overrightarrow{\alpha(s)\alpha^*(s^*)}// N_\alpha(s)// B_{\alpha^*}(s^*).
\end{equation}
\end{definition}

\begin{remark}\label{rem:mannheim-sign}
Throughout this paper, we assume without loss of generality that
\begin{equation}\label{fundamental hypotesis Mannheim Pair}
    N_\alpha(s)=-B_{\alpha^*}(s^*).
\end{equation}
\end{remark}

\begin{remark}\label{rem:normalization}
The relations of Theorem~\ref{lem:relationsL3} and Corollary~\ref{corolario par s e par s*} are stated under a normalization of the orientations: besides $N_\alpha=-B_{\alpha^*}$, we assume $\frac{ds^*}{ds}>0$ and the sign conditions of Corollary~\ref{cor:angles}. These can be obtained by applying an isometry of $\mathbb{L}^3$, possibly orientation-reversing (which changes the signs of $\tau_\alpha$ and $B_\alpha$), and by reversing the orientations of $\alpha$ and $\alpha^*$. For instance, the pairs of Types~1 and~4 of Section~\ref{sec:choice_phi mannheim condition} have $\tau_\beta>0$ as printed; the map $(x,y,z)\mapsto(-x,y,z)$ together with $t\mapsto-t$ takes them to pairs with $\tau<0$ satisfying these conditions.
\end{remark}

\begin{example}\label{Example Mannheim pair slanthelixXhelix}  Mannheim pair (Slant helix, helix).
\begin{eqnarray}
 \alpha(s)=
\left(
\operatorname{arcsinh}(e^{2s})-\dfrac{1}{\sqrt{1+e^{-4s}}},\;
\dfrac{1}{\sqrt{1+e^{4s}}}-\operatorname{arcsinh}(e^{-2s}),\;
\dfrac{\sqrt{3}}{2}\ln\cosh 2s+\dfrac{2}{\sqrt{3}}
\right),\quad s>0,
\end{eqnarray}
\begin{eqnarray}
\alpha^{*}(s^{*})=
\left(
\begin{array}{l}
\operatorname{artanh}\!\Big(\cos\!\big(\tfrac12\arcsin e^{-2s^{*}}\big)\Big)-2\cos\!\big(\tfrac12\arcsin e^{-2s^{*}}\big)\\[6pt]
\ \ 2\sin\!\big(\tfrac12\arcsin e^{-2s^{*}}\big)-\operatorname{artanh}\!\Big(\sin\!\big(\tfrac12\arcsin e^{-2s^{*}}\big)\Big)\\[6pt]
\ \ \sqrt{3}\,s^{*}
\end{array}
\right),\quad s^{*}>0,
\end{eqnarray}
is a Mannheim pair with $\lambda=-\dfrac{1}{\sqrt{3}}$.

 In fact, take $s^{*}(s)=\frac{1}{2}\ln\cosh 2s$ with $s>0$.

\begin{equation}
\alpha^{*}(s^{*}(s))-\alpha(s)
=
-\dfrac{1}{\sqrt{3}}\left(
\dfrac{\sqrt{3}}{\sqrt{1+e^{-4s}}},\;
-\dfrac{\sqrt{3}}{\sqrt{1+e^{4s}}},\;
2
\right),
\qquad s>0.
\end{equation}
 \begin{equation}
N_{\alpha}(s)=
\left(
\dfrac{\sqrt{3}}{\sqrt{1+e^{-4s}}},\;
-\dfrac{\sqrt{3}}{\sqrt{1+e^{4s}}},\;
2
\right),
 \end{equation}

\begin{equation}
 B_{\alpha^{*}}(s^{*}(s))=-\left(
\dfrac{\sqrt{3}}{\sqrt{1+e^{-4s}}},\;
-\dfrac{\sqrt{3}}{\sqrt{1+e^{4s}}},\;
2
\right).
\end{equation}

Then
$$\alpha^{*}(s^{*}(s))-\alpha(s)=-\dfrac{1}{\sqrt{3}}\,N_{\alpha}(s)=\dfrac{1}{\sqrt{3}}\,B_{\alpha^{*}}(s^{*}(s)).$$
Moreover, $\kappa_\alpha=\sqrt3\operatorname{sech}^2 2s$ and $\tau_\alpha=-\sqrt3\operatorname{sech}2s\tanh2s$, so $\dfrac{\tau_\alpha}{\kappa_\alpha}=-\sinh2s$ is non-constant and $\alpha$ is not a helix. However, $\langle N_\alpha(s),e_3\rangle=-2$ is constant, so $\alpha$ is a slant helix (with timelike axis $e_3$). On the other hand, $\langle T_{\alpha^*}(s^*),e_3\rangle=-\sqrt3$ is constant, so $\alpha^*$ is a helix. Here $T_\alpha$ is spacelike, $N_\alpha$ is timelike, $B_\alpha$ is spacelike, $T_{\alpha^*}$ and $N_{\alpha^*}$ are spacelike and $B_{\alpha^*}$ is timelike, so the pair is of Type~3 (see Corollary~\ref{cor:angles}). In accordance with Proposition~\ref{prop:slant_helix_L3}, stated below, $\sigma_3=-\dfrac{2}{\sqrt3}$ is constant and $\dfrac{\tau_{\alpha^*}}{\kappa_{\alpha^*}}=\dfrac{\eta_3}{\sigma_3}=\dfrac{\sqrt3}{2}$.
\end{example}

The next result shows that the causal character of $B_{\alpha^*}$ is determined by that of $T_\alpha \wedge T_{\alpha^*}$.

\begin{corollary}
We have
\begin{enumerate}
    \item $T_{\alpha}\wedge T_{\alpha^{*}}$ is timelike if and only if $B_{\alpha^{*}}$ (or $N_{\alpha}$) is timelike.
    \item $T_{\alpha}\wedge T_{\alpha^{*}}$ is spacelike (and nonzero) if and only if $B_{\alpha^{*}}$ (or $N_{\alpha}$) is spacelike.
\end{enumerate}
\end{corollary}

\begin{proof}
Since $\langle T_{\alpha},N_{\alpha}\rangle=0=\langle T_{\alpha},-B_{\alpha^{*}}\rangle$, we have $T_{\alpha}=x\,T_{\alpha^{*}}+y\,N_{\alpha^{*}}$, with $y\neq0$: if $y=0$, then $T_\alpha=\pm T_{\alpha^*}$ and, since $\frac{d\alpha^*}{ds}=(1+\lambda\varepsilon_3\kappa_\alpha)T_\alpha+\lambda\varepsilon_3\tau_\alpha B_\alpha$, we would get $\tau_\alpha=0$, which is impossible for a Mannheim pair (Corollary~\ref{corolario par s e par s*}). Then
$$\langle T_{\alpha}\wedge T_{\alpha^{*}},T_{\alpha}\wedge T_{\alpha^{*}} \rangle=y^2\langle N_{\alpha^{*}}\wedge T_{\alpha^{*}},N_{\alpha^{*}}\wedge T_{\alpha^{*}} \rangle=y^2\langle T_{\alpha^{*}}\wedge N_{\alpha^{*}},T_{\alpha^{*}}\wedge N_{\alpha^{*}} \rangle=y^2\langle B_{\alpha^{*}},B_{\alpha^{*}} \rangle,$$
therefore
\begin{equation}\label{cases of plane T, T*}
\langle T_{\alpha}\wedge T_{\alpha^{*}},T_{\alpha}\wedge T_{\alpha^{*}} \rangle=y^2\langle B_{\alpha^{*}},B_{\alpha^{*}} \rangle=y^2\langle N_{\alpha},N_{\alpha} \rangle,
\end{equation}
and the two items of the result follow directly from this formula.
\end{proof}

The next corollary fixes, for each of the five types of \cite{ElsharkawyElshenhab}, the angle between $T_\alpha$ and $T_{\alpha^*}$.

\begin{corollary}\label{cor:angles} For a Mannheim pair $(\alpha(s),\alpha^{*}(s^{*}))$, we have five cases to consider (the types are those of Propositions~1 and~2 of \cite{ElsharkawyElshenhab}):
\begin{enumerate}
    \item $T_{\alpha^{*}}$ is timelike, $N_{\alpha^{*}}$ is spacelike, $B_{\alpha^{*}}$ is spacelike, and
    \begin{enumerate}[label=\arabic{enumi}.\arabic*.]
        \item $T_{\alpha}$ is spacelike, $N_{\alpha}$ is spacelike, $B_{\alpha}$ is timelike, with $T_{\alpha^{*}}$ and $B_{\alpha}$ in the same time cone.

        By equation \eqref{angle type 2} (\textbf{Type 1}):
\begin{equation}\label{angle phi1 cor}
     \sinh \bm{\varphi}_{1}=\langle T_{\alpha}, T_{\alpha^{*}} \rangle
\end{equation}
        \item $T_{\alpha}$ is timelike, $N_{\alpha}$ is spacelike, $B_{\alpha}$ is spacelike, with $T_{\alpha}$ and $T_{\alpha^{*}}$ in the same time cone.

        By equation \eqref{angle type 1} (\textbf{Type 2}):
\begin{equation}\label{angle phi2 cor}
     \cosh \bm{\varphi}_{2}=-\langle T_{\alpha}, T_{\alpha^{*}} \rangle
\end{equation}
    \end{enumerate}
    \item $T_{\alpha^{*}}$ is spacelike, and
    \begin{enumerate}[label=\arabic{enumi}.\arabic*.]
        \item $T_{\alpha}$ is spacelike, $N_{\alpha}$ is timelike, $B_{\alpha}$ is spacelike, and $N_{\alpha^{*}}$ is spacelike, $B_{\alpha^{*}}$ is timelike.

        $\langle T_{\alpha}\wedge T_{\alpha^{*}},T_{\alpha}\wedge T_{\alpha^{*}} \rangle<0$ by equation \eqref{cases of plane T, T*}, then by equation \eqref{angle type 3} (\textbf{Type 3}):
\begin{equation}\label{angle phi3 cor}
     \cos \bm{\varphi}_{3}=\langle T_{\alpha}, T_{\alpha^{*}} \rangle
\end{equation}
        \item $T_{\alpha}$ is timelike, $N_{\alpha}$ is spacelike, $B_{\alpha}$ is spacelike, and $N_{\alpha^{*}}$ is timelike, $B_{\alpha^{*}}$ is spacelike, with $\langle T_{\alpha^{*}},B_{\alpha}\rangle>0$.

        By equation \eqref{angle type 2} (\textbf{Type 4}):
\begin{equation}\label{angle phi4 cor}
     \sinh \bm{\varphi}_{4}=\langle T_{\alpha}, T_{\alpha^{*}} \rangle
\end{equation}
        \item $T_{\alpha}$ is spacelike, $N_{\alpha}$ is spacelike, $B_{\alpha}$ is timelike, and $N_{\alpha^{*}}$ is timelike, $B_{\alpha^{*}}$ is spacelike, with $\langle T_{\alpha}, T_{\alpha^{*}} \rangle>0$.

        $\langle T_{\alpha}\wedge T_{\alpha^{*}},T_{\alpha}\wedge T_{\alpha^{*}} \rangle>0$ by equation \eqref{cases of plane T, T*}, then by equation \eqref{angle type 5} (\textbf{Type 5}):
\begin{equation}\label{angle phi5 cor}
     \cosh \bm{\varphi}_{5}=\langle T_{\alpha}, T_{\alpha^{*}} \rangle
\end{equation}
    \end{enumerate}
\end{enumerate}
\end{corollary}

The next theorem collects the relations between the Frenet apparatus of a Mannheim pair in $\mathbb{L}^3$, for each of the five types of \cite{ElsharkawyElshenhab}, obtained by straightforward calculations from Theorems~1, 2 and~5 of \cite{ElsharkawyElshenhab}.

\begin{theorem}[Elsharkawy and Elshenhab \cite{ElsharkawyElshenhab}]\label{lem:relationsL3}
Let $\{T_\alpha,\allowbreak N_\alpha,\allowbreak B_\alpha,\allowbreak \kappa_\alpha,\allowbreak \tau_\alpha\}$ and $\{T_{\alpha^*},\allowbreak N_{\alpha^*},\allowbreak B_{\alpha^*},\allowbreak \kappa_{\alpha^*},\allowbreak \tau_{\alpha^*}\}$ be the Frenet apparatus of the Mannheim pair $(\alpha, \alpha^*)$ of Type~$k$, $k\in\{1,2,3,4,5\}$, respectively. Define $\bm{\varphi}_k(s, s^*) = \angle(T_\alpha(s), T_{\alpha^*}(s^*))$ as in Corollary~\ref{cor:angles}, and set $\bm{a}_k(s) = \bm{\varphi}_k(s, s^*(s))$ and $\bm{A}_k(s^*) = \bm{\varphi}_k(s(s^*), s^*)$. In Types~2, 3 and~5 the angle $\bm{\varphi}_k$ is taken with sign, so that the expressions of $T_{\alpha^*}$ below hold. There exists a nonzero constant $\lambda$ such that
\begin{equation}\label{eq:alpha_star}
\alpha^*(s^*) = \alpha(s(s^{*})) + \lambda N_\alpha(s(s^{*})),
\end{equation}
or equivalently
\begin{equation}\label{eq:alpha}
\alpha(s)=\alpha^*(s^{*}(s))+\lambda B_{\alpha^{*}}(s^{*}(s)).
\end{equation}
And if $\kappa_{\alpha^*}\neq 0$ the following relations hold:
\begin{enumerate}[label=\textbf{Type \arabic*.}, leftmargin=*]
\item
\begin{align}
T_{\alpha^*}(s^{*}) &= \sinh\bm{A}_1(s^{*})\cdot T_\alpha(s(s^{*})) + \cosh\bm{A}_1(s^{*})\cdot B_\alpha(s(s^{*})), \label{eq:T_star_1}\\[6pt]
N_{\alpha^*}(s^{*}) &= \cosh\bm{A}_1(s^{*})\cdot T_\alpha(s(s^{*})) + \sinh\bm{A}_1(s^{*})\cdot B_\alpha(s(s^{*})), \label{eq:N_star_1}\\[6pt]
T_\alpha(s) &= -\sinh\bm{a}_1(s)\cdot T_{\alpha^*}(s^{*}(s)) + \cosh\bm{a}_1(s)\cdot N_{\alpha^*}(s^{*}(s)), \label{eq:T_alpha_1}\\[6pt]
B_\alpha(s) &= \cosh\bm{a}_1(s)\cdot T_{\alpha^*}(s^{*}(s)) - \sinh\bm{a}_1(s)\cdot N_{\alpha^{*}}(s^{*}(s)), \label{eq:B_alpha_1}\\[6pt]
\kappa_\alpha &= \frac{1}{\lambda}\cosh^2\bm{a}_1, \qquad \tau_\alpha = \frac{1}{\lambda}\sinh\bm{a}_1\cosh\bm{a}_1, \label{eq:kappa_tau_alpha_1}\\[6pt]
\kappa_{\alpha^*} &= \frac{d\bm{A}_1}{ds^*}, \qquad \tau_{\alpha^*} = \frac{1}{\lambda}\coth\bm{A}_1, \label{eq:kappa_tau_star_1}\\[6pt]
\sinh\bm{a}_1 &= -\frac{ds^*}{ds}, \qquad \operatorname{csch}\bm{A}_{1}=-\dfrac{ds}{ds^{*}}. \label{eq:ds_1}
\end{align}
\item
\begin{align}
T_{\alpha^*}(s^{*}) &= \cosh\bm{A}_2(s^{*})\cdot T_\alpha(s(s^{*})) + \sinh\bm{A}_2(s^{*})\cdot B_\alpha(s(s^{*})), \label{eq:T_star_2}\\[6pt]
N_{\alpha^*}(s^{*}) &= \sinh\bm{A}_2(s^{*})\cdot T_\alpha(s(s^{*})) + \cosh\bm{A}_2(s^{*})\cdot B_\alpha(s(s^{*})), \label{eq:N_star_2}\\[6pt]
T_\alpha(s) &= \cosh\bm{a}_2(s)\cdot T_{\alpha^*}(s^{*}(s)) - \sinh\bm{a}_2(s)\cdot N_{\alpha^*}(s^{*}(s)), \label{eq:T_alpha_2}\\[6pt]
B_\alpha(s) &= -\sinh\bm{a}_2(s)\cdot T_{\alpha^*}(s^{*}(s)) + \cosh\bm{a}_2(s)\cdot N_{\alpha^{*}}(s^{*}(s)), \label{eq:B_alpha_2}\\[6pt]
\kappa_\alpha &= \frac{1}{\lambda}\sinh^2\bm{a}_2, \qquad \tau_\alpha = \frac{1}{\lambda}\sinh\bm{a}_2\cosh\bm{a}_2, \label{eq:kappa_tau_alpha_2}\\[6pt]
\kappa_{\alpha^*} &= \frac{d\bm{A}_2}{ds^*}, \qquad \tau_{\alpha^*} = \frac{1}{\lambda}\tanh\bm{A}_2, \label{eq:kappa_tau_star_2}\\[6pt]
\cosh\bm{a}_2 &= \frac{ds^*}{ds}, \qquad \operatorname{sech}\bm{A}_{2}=\dfrac{ds}{ds^{*}}. \label{eq:ds_2}
\end{align}
\item
\begin{align}
T_{\alpha^*}(s^{*}) &= \cos\bm{A}_3(s^{*})\cdot T_\alpha(s(s^{*})) - \sin\bm{A}_3(s^{*})\cdot B_\alpha(s(s^{*})), \label{eq:T_star_3}\\[6pt]
N_{\alpha^*}(s^{*}) &= -\sin\bm{A}_3(s^{*})\cdot T_\alpha(s(s^{*})) - \cos\bm{A}_3(s^{*})\cdot B_\alpha(s(s^{*})), \label{eq:N_star_3}\\[6pt]
T_\alpha(s) &= \cos\bm{a}_3(s)\cdot T_{\alpha^*}(s^{*}(s)) - \sin\bm{a}_3(s)\cdot N_{\alpha^*}(s^{*}(s)), \label{eq:T_alpha_3}\\[6pt]
B_\alpha(s) &= -\sin\bm{a}_3(s)\cdot T_{\alpha^*}(s^{*}(s)) - \cos\bm{a}_3(s)\cdot N_{\alpha^{*}}(s^{*}(s)), \label{eq:B_alpha_3}\\[6pt]
\kappa_\alpha &= -\frac{1}{\lambda}\sin^2\bm{a}_3, \qquad \tau_\alpha = -\frac{1}{\lambda}\sin\bm{a}_3\cos\bm{a}_3, \label{eq:kappa_tau_alpha_3}\\[6pt]
\kappa_{\alpha^*} &= \frac{d\bm{A}_3}{ds^*}, \qquad \tau_{\alpha^*} = \frac{1}{\lambda}\tan\bm{A}_3, \label{eq:kappa_tau_star_3}\\[6pt]
\cos\bm{a}_3 &= \frac{ds^*}{ds}, \qquad \sec\bm{A}_{3}=\dfrac{ds}{ds^{*}}. \label{eq:ds_3}
\end{align}
\item
\begin{align}
T_{\alpha^*}(s^{*}) &= -\sinh\bm{A}_4(s^{*})\cdot T_\alpha(s(s^{*})) + \cosh\bm{A}_4(s^{*})\cdot B_\alpha(s(s^{*})), \label{eq:T_star_4}\\[6pt]
N_{\alpha^*}(s^{*}) &= -\cosh\bm{A}_4(s^{*})\cdot T_\alpha(s(s^{*})) + \sinh\bm{A}_4(s^{*})\cdot B_\alpha(s(s^{*})), \label{eq:N_star_4}\\[6pt]
T_\alpha(s) &= \sinh\bm{a}_4(s)\cdot T_{\alpha^*}(s^{*}(s)) - \cosh\bm{a}_4(s)\cdot N_{\alpha^*}(s^{*}(s)), \label{eq:T_alpha_4}\\[6pt]
B_\alpha(s) &= \cosh\bm{a}_4(s)\cdot T_{\alpha^*}(s^{*}(s)) - \sinh\bm{a}_4(s)\cdot N_{\alpha^{*}}(s^{*}(s)), \label{eq:B_alpha_4}\\[6pt]
\kappa_\alpha &= -\frac{1}{\lambda}\cosh^2\bm{a}_4, \qquad \tau_\alpha = \frac{1}{\lambda}\sinh\bm{a}_4\cosh\bm{a}_4, \label{eq:kappa_tau_alpha_4}\\[6pt]
\kappa_{\alpha^*} &= \frac{d\bm{A}_4}{ds^*}, \qquad \tau_{\alpha^*} = -\frac{1}{\lambda}\coth\bm{A}_4, \label{eq:kappa_tau_star_4}\\[6pt]
\sinh\bm{a}_4 &= \frac{ds^*}{ds}, \qquad \operatorname{csch}\bm{A}_{4}=\dfrac{ds}{ds^{*}}. \label{eq:ds_4}
\end{align}
\item
\begin{align}
T_{\alpha^*}(s^{*}) &= \cosh\bm{A}_5(s^{*})\cdot T_\alpha(s(s^{*})) - \sinh\bm{A}_5(s^{*})\cdot B_\alpha(s(s^{*})), \label{eq:T_star_5}\\[6pt]
N_{\alpha^*}(s^{*}) &= \sinh\bm{A}_5(s^{*})\cdot T_\alpha(s(s^{*})) - \cosh\bm{A}_5(s^{*})\cdot B_\alpha(s(s^{*})), \label{eq:N_star_5}\\[6pt]
T_\alpha(s) &= \cosh\bm{a}_5(s)\cdot T_{\alpha^*}(s^{*}(s)) - \sinh\bm{a}_5(s)\cdot N_{\alpha^*}(s^{*}(s)), \label{eq:T_alpha_5}\\[6pt]
B_\alpha(s) &= \sinh\bm{a}_5(s)\cdot T_{\alpha^*}(s^{*}(s)) - \cosh\bm{a}_5(s)\cdot N_{\alpha^{*}}(s^{*}(s)), \label{eq:B_alpha_5}\\[6pt]
\kappa_\alpha &= -\frac{1}{\lambda}\sinh^2\bm{a}_5, \qquad \tau_\alpha = \frac{1}{\lambda}\sinh\bm{a}_5\cosh\bm{a}_5, \label{eq:kappa_tau_alpha_5}\\[6pt]
\kappa_{\alpha^*} &= \frac{d\bm{A}_5}{ds^*}, \qquad \tau_{\alpha^*} = -\frac{1}{\lambda}\tanh\bm{A}_5, \label{eq:kappa_tau_star_5}\\[6pt]
\cosh\bm{a}_5 &= \frac{ds^*}{ds}, \qquad \operatorname{sech}\bm{A}_{5}=\dfrac{ds}{ds^{*}}. \label{eq:ds_5}
\end{align}
\end{enumerate}
\end{theorem}

In the next result, for each type, the angles and the curvatures and torsions of one curve of the pair are written in terms of those of the other.

\begin{corollary}\label{corolario par s e par s*} Let $(\alpha(s),\alpha^{*}(s^{*}))$ be a Mannheim pair of Type~$k$ in $\mathbb{L}^3$, with the conventions of Theorem~\ref{lem:relationsL3}, and suppose $\dfrac{ds^{*}}{ds}>0$. Then $\tau_\alpha\neq0$ and, in all five types,
\begin{equation}\label{tau* e tau e kappa L3}
\tau_{\alpha^{*}}(s^{*})=\dfrac{1}{\lambda}\,\dfrac{\kappa_{\alpha}(s(s^{*}))}{\tau_{\alpha}(s(s^{*}))}.
\end{equation}
Moreover:
\begin{enumerate}[label=\textbf{Type \arabic*.}, leftmargin=*]

\item ($\lambda>0$, and then $\dfrac{\kappa_\alpha}{\tau_\alpha}<-1$ and $\bm a_1<0$)
\begin{equation}\label{eq:L3-1-A}
\cosh\bm A_1(s^{*})=\dfrac{-\dfrac{\kappa_{\alpha}(s(s^{*}))}{\tau_{\alpha}(s(s^{*}))}}{\sqrt{\left(\dfrac{\kappa_{\alpha}(s(s^{*}))}{\tau_{\alpha}(s(s^{*}))}\right)^2-1}},\qquad
\sinh\bm A_1(s^{*})=\dfrac{-1}{\sqrt{\left(\dfrac{\kappa_{\alpha}(s(s^{*}))}{\tau_{\alpha}(s(s^{*}))}\right)^2-1}},
\end{equation}
\begin{equation}\label{eq:L3-1-ds}
\dfrac{ds}{ds^{*}}(s^{*})=\sqrt{\left(\dfrac{\kappa_{\alpha}(s(s^{*}))}{\tau_{\alpha}(s(s^{*}))}\right)^2-1},
\end{equation}
\begin{equation}\label{eq:L3-1-kappa*}
\kappa_{\alpha^{*}}(s^{*})=\dfrac{-\dfrac{d}{ds^{*}}\left(\dfrac{\kappa_{\alpha}(s(s^{*}))}{\tau_{\alpha}(s(s^{*}))}\right)}{\left(\dfrac{\kappa_{\alpha}(s(s^{*}))}{\tau_{\alpha}(s(s^{*}))}\right)^2-1},
\end{equation}
\begin{equation}\label{mannheim equation 1}
\kappa_\alpha(s)^2-\tau_\alpha(s)^2=\dfrac{1}{\lambda}\kappa_\alpha(s),
\end{equation}
\begin{equation}\label{eq:L3-1-a}
\cosh\bm a_1(s)=\dfrac{-\lambda\tau_{\alpha^{*}}(s^{*}(s))}{\sqrt{\lambda^2\tau_{\alpha^{*}}(s^{*}(s))^2-1}},\qquad
\sinh\bm a_1(s)=\dfrac{-1}{\sqrt{\lambda^2\tau_{\alpha^{*}}(s^{*}(s))^2-1}},
\end{equation}
\begin{equation}\label{eq:L3-1-ds*}
\dfrac{ds^{*}}{ds}(s)=\dfrac{1}{\sqrt{\lambda^2\tau_{\alpha^{*}}(s^{*}(s))^2-1}},
\end{equation}
\begin{equation}\label{eq:L3-1-kt}
\kappa_\alpha(s)=\dfrac{\lambda\tau_{\alpha^{*}}(s^{*}(s))^2}{\lambda^2\tau_{\alpha^{*}}(s^{*}(s))^2-1},\qquad
\tau_\alpha(s)=\dfrac{\tau_{\alpha^{*}}(s^{*}(s))}{\lambda^2\tau_{\alpha^{*}}(s^{*}(s))^2-1},
\end{equation}
\begin{equation}\label{mannheim pair equation 1}
\kappa_{\alpha^{*}}(s^{*})=\dfrac{\lambda\,\dfrac{d}{ds^{*}}\left(\tau_{\alpha^{*}}(s^{*})\right)}{1-\lambda^2\tau_{\alpha^{*}}(s^{*})^2}.
\end{equation}

\item ($\lambda>0$, and then $\left|\dfrac{\kappa_\alpha}{\tau_\alpha}\right|<1$)
\begin{equation}\label{eq:L3-2-A}
\cosh\bm A_2(s^{*})=\dfrac{1}{\sqrt{1-\left(\dfrac{\kappa_{\alpha}(s(s^{*}))}{\tau_{\alpha}(s(s^{*}))}\right)^2}},\qquad
\sinh\bm A_2(s^{*})=\dfrac{\dfrac{\kappa_{\alpha}(s(s^{*}))}{\tau_{\alpha}(s(s^{*}))}}{\sqrt{1-\left(\dfrac{\kappa_{\alpha}(s(s^{*}))}{\tau_{\alpha}(s(s^{*}))}\right)^2}},
\end{equation}
\begin{equation}\label{eq:L3-2-ds}
\dfrac{ds}{ds^{*}}(s^{*})=\sqrt{1-\left(\dfrac{\kappa_{\alpha}(s(s^{*}))}{\tau_{\alpha}(s(s^{*}))}\right)^2},
\end{equation}
\begin{equation}\label{eq:L3-2-kappa*}
\kappa_{\alpha^{*}}(s^{*})=\dfrac{\dfrac{d}{ds^{*}}\left(\dfrac{\kappa_{\alpha}(s(s^{*}))}{\tau_{\alpha}(s(s^{*}))}\right)}{1-\left(\dfrac{\kappa_{\alpha}(s(s^{*}))}{\tau_{\alpha}(s(s^{*}))}\right)^2},
\end{equation}
\begin{equation}\label{mannheim equation 2}
\kappa_\alpha(s)^2-\tau_\alpha(s)^2=-\dfrac{1}{\lambda}\kappa_\alpha(s),
\end{equation}
\begin{equation}\label{eq:L3-2-a}
\cosh\bm a_2(s)=\dfrac{1}{\sqrt{1-\lambda^2\tau_{\alpha^{*}}(s^{*}(s))^2}},\qquad
\sinh\bm a_2(s)=\dfrac{\lambda\tau_{\alpha^{*}}(s^{*}(s))}{\sqrt{1-\lambda^2\tau_{\alpha^{*}}(s^{*}(s))^2}},
\end{equation}
\begin{equation}\label{eq:L3-2-ds*}
\dfrac{ds^{*}}{ds}(s)=\dfrac{1}{\sqrt{1-\lambda^2\tau_{\alpha^{*}}(s^{*}(s))^2}},
\end{equation}
\begin{equation}\label{eq:L3-2-kt}
\kappa_\alpha(s)=\dfrac{\lambda\tau_{\alpha^{*}}(s^{*}(s))^2}{1-\lambda^2\tau_{\alpha^{*}}(s^{*}(s))^2},\qquad
\tau_\alpha(s)=\dfrac{\tau_{\alpha^{*}}(s^{*}(s))}{1-\lambda^2\tau_{\alpha^{*}}(s^{*}(s))^2},
\end{equation}
\begin{equation}\label{mannheim pair equation 2}
\kappa_{\alpha^{*}}(s^{*})=\dfrac{\lambda\,\dfrac{d}{ds^{*}}\left(\tau_{\alpha^{*}}(s^{*})\right)}{1-\lambda^2\tau_{\alpha^{*}}(s^{*})^2}.
\end{equation}

\item ($\lambda<0$)
\begin{equation}\label{eq:L3-3-A}
\cos\bm A_3(s^{*})=\dfrac{1}{\sqrt{1+\left(\dfrac{\kappa_{\alpha}(s(s^{*}))}{\tau_{\alpha}(s(s^{*}))}\right)^2}},\qquad
\sin\bm A_3(s^{*})=\dfrac{\dfrac{\kappa_{\alpha}(s(s^{*}))}{\tau_{\alpha}(s(s^{*}))}}{\sqrt{1+\left(\dfrac{\kappa_{\alpha}(s(s^{*}))}{\tau_{\alpha}(s(s^{*}))}\right)^2}},
\end{equation}
\begin{equation}\label{eq:L3-3-ds}
\dfrac{ds}{ds^{*}}(s^{*})=\sqrt{1+\left(\dfrac{\kappa_{\alpha}(s(s^{*}))}{\tau_{\alpha}(s(s^{*}))}\right)^2},
\end{equation}
\begin{equation}\label{eq:L3-3-kappa*}
\kappa_{\alpha^{*}}(s^{*})=\dfrac{\dfrac{d}{ds^{*}}\left(\dfrac{\kappa_{\alpha}(s(s^{*}))}{\tau_{\alpha}(s(s^{*}))}\right)}{1+\left(\dfrac{\kappa_{\alpha}(s(s^{*}))}{\tau_{\alpha}(s(s^{*}))}\right)^2},
\end{equation}
\begin{equation}\label{mannheim equation 3}
\kappa_\alpha(s)^2+\tau_\alpha(s)^2=-\dfrac{1}{\lambda}\kappa_\alpha(s),
\end{equation}
\begin{equation}\label{eq:L3-3-a}
\cos\bm a_3(s)=\dfrac{1}{\sqrt{1+\lambda^2\tau_{\alpha^{*}}(s^{*}(s))^2}},\qquad
\sin\bm a_3(s)=\dfrac{\lambda\tau_{\alpha^{*}}(s^{*}(s))}{\sqrt{1+\lambda^2\tau_{\alpha^{*}}(s^{*}(s))^2}},
\end{equation}
\begin{equation}\label{eq:L3-3-ds*}
\dfrac{ds^{*}}{ds}(s)=\dfrac{1}{\sqrt{1+\lambda^2\tau_{\alpha^{*}}(s^{*}(s))^2}},
\end{equation}
\begin{equation}\label{eq:L3-3-kt}
\kappa_\alpha(s)=\dfrac{-\lambda\tau_{\alpha^{*}}(s^{*}(s))^2}{1+\lambda^2\tau_{\alpha^{*}}(s^{*}(s))^2},\qquad
\tau_\alpha(s)=\dfrac{-\tau_{\alpha^{*}}(s^{*}(s))}{1+\lambda^2\tau_{\alpha^{*}}(s^{*}(s))^2},
\end{equation}
\begin{equation}\label{mannheim pair equation 3}
\kappa_{\alpha^{*}}(s^{*})=\dfrac{\lambda\,\dfrac{d}{ds^{*}}\left(\tau_{\alpha^{*}}(s^{*})\right)}{1+\lambda^2\tau_{\alpha^{*}}(s^{*})^2}.
\end{equation}

\item ($\lambda<0$, and then $\dfrac{\kappa_\alpha}{\tau_\alpha}<-1$ and $\bm a_4>0$)
\begin{equation}\label{eq:L3-4-A}
\cosh\bm A_4(s^{*})=\dfrac{-\dfrac{\kappa_{\alpha}(s(s^{*}))}{\tau_{\alpha}(s(s^{*}))}}{\sqrt{\left(\dfrac{\kappa_{\alpha}(s(s^{*}))}{\tau_{\alpha}(s(s^{*}))}\right)^2-1}},\qquad
\sinh\bm A_4(s^{*})=\dfrac{1}{\sqrt{\left(\dfrac{\kappa_{\alpha}(s(s^{*}))}{\tau_{\alpha}(s(s^{*}))}\right)^2-1}},
\end{equation}
\begin{equation}\label{eq:L3-4-ds}
\dfrac{ds}{ds^{*}}(s^{*})=\sqrt{\left(\dfrac{\kappa_{\alpha}(s(s^{*}))}{\tau_{\alpha}(s(s^{*}))}\right)^2-1},
\end{equation}
\begin{equation}\label{eq:L3-4-kappa*}
\kappa_{\alpha^{*}}(s^{*})=\dfrac{\dfrac{d}{ds^{*}}\left(\dfrac{\kappa_{\alpha}(s(s^{*}))}{\tau_{\alpha}(s(s^{*}))}\right)}{\left(\dfrac{\kappa_{\alpha}(s(s^{*}))}{\tau_{\alpha}(s(s^{*}))}\right)^2-1},
\end{equation}
\begin{equation}\label{mannheim equation 4}
\kappa_\alpha(s)^2-\tau_\alpha(s)^2=-\dfrac{1}{\lambda}\kappa_\alpha(s),
\end{equation}
\begin{equation}\label{eq:L3-4-a}
\cosh\bm a_4(s)=\dfrac{-\lambda\tau_{\alpha^{*}}(s^{*}(s))}{\sqrt{\lambda^2\tau_{\alpha^{*}}(s^{*}(s))^2-1}},\qquad
\sinh\bm a_4(s)=\dfrac{1}{\sqrt{\lambda^2\tau_{\alpha^{*}}(s^{*}(s))^2-1}},
\end{equation}
\begin{equation}\label{eq:L3-4-ds*}
\dfrac{ds^{*}}{ds}(s)=\dfrac{1}{\sqrt{\lambda^2\tau_{\alpha^{*}}(s^{*}(s))^2-1}},
\end{equation}
\begin{equation}\label{eq:L3-4-kt}
\kappa_\alpha(s)=\dfrac{-\lambda\tau_{\alpha^{*}}(s^{*}(s))^2}{\lambda^2\tau_{\alpha^{*}}(s^{*}(s))^2-1},\qquad
\tau_\alpha(s)=\dfrac{-\tau_{\alpha^{*}}(s^{*}(s))}{\lambda^2\tau_{\alpha^{*}}(s^{*}(s))^2-1},
\end{equation}
\begin{equation}\label{mannheim pair equation 4}
\kappa_{\alpha^{*}}(s^{*})=\dfrac{-\lambda\,\dfrac{d}{ds^{*}}\left(\tau_{\alpha^{*}}(s^{*})\right)}{1-\lambda^2\tau_{\alpha^{*}}(s^{*})^2}.
\end{equation}

\item ($\lambda<0$, and then $\left|\dfrac{\kappa_\alpha}{\tau_\alpha}\right|<1$)
\begin{equation}\label{eq:L3-5-A}
\cosh\bm A_5(s^{*})=\dfrac{1}{\sqrt{1-\left(\dfrac{\kappa_{\alpha}(s(s^{*}))}{\tau_{\alpha}(s(s^{*}))}\right)^2}},\qquad
\sinh\bm A_5(s^{*})=\dfrac{-\dfrac{\kappa_{\alpha}(s(s^{*}))}{\tau_{\alpha}(s(s^{*}))}}{\sqrt{1-\left(\dfrac{\kappa_{\alpha}(s(s^{*}))}{\tau_{\alpha}(s(s^{*}))}\right)^2}},
\end{equation}
\begin{equation}\label{eq:L3-5-ds}
\dfrac{ds}{ds^{*}}(s^{*})=\sqrt{1-\left(\dfrac{\kappa_{\alpha}(s(s^{*}))}{\tau_{\alpha}(s(s^{*}))}\right)^2},
\end{equation}
\begin{equation}\label{eq:L3-5-kappa*}
\kappa_{\alpha^{*}}(s^{*})=\dfrac{-\dfrac{d}{ds^{*}}\left(\dfrac{\kappa_{\alpha}(s(s^{*}))}{\tau_{\alpha}(s(s^{*}))}\right)}{1-\left(\dfrac{\kappa_{\alpha}(s(s^{*}))}{\tau_{\alpha}(s(s^{*}))}\right)^2},
\end{equation}
\begin{equation}\label{mannheim equation 5}
\kappa_\alpha(s)^2-\tau_\alpha(s)^2=\dfrac{1}{\lambda}\kappa_\alpha(s),
\end{equation}
\begin{equation}\label{eq:L3-5-a}
\cosh\bm a_5(s)=\dfrac{1}{\sqrt{1-\lambda^2\tau_{\alpha^{*}}(s^{*}(s))^2}},\qquad
\sinh\bm a_5(s)=\dfrac{-\lambda\tau_{\alpha^{*}}(s^{*}(s))}{\sqrt{1-\lambda^2\tau_{\alpha^{*}}(s^{*}(s))^2}},
\end{equation}
\begin{equation}\label{eq:L3-5-ds*}
\dfrac{ds^{*}}{ds}(s)=\dfrac{1}{\sqrt{1-\lambda^2\tau_{\alpha^{*}}(s^{*}(s))^2}},
\end{equation}
\begin{equation}\label{eq:L3-5-kt}
\kappa_\alpha(s)=\dfrac{-\lambda\tau_{\alpha^{*}}(s^{*}(s))^2}{1-\lambda^2\tau_{\alpha^{*}}(s^{*}(s))^2},\qquad
\tau_\alpha(s)=\dfrac{-\tau_{\alpha^{*}}(s^{*}(s))}{1-\lambda^2\tau_{\alpha^{*}}(s^{*}(s))^2},
\end{equation}
\begin{equation}\label{mannheim pair equation 5}
\kappa_{\alpha^{*}}(s^{*})=\dfrac{-\lambda\,\dfrac{d}{ds^{*}}\left(\tau_{\alpha^{*}}(s^{*})\right)}{1-\lambda^2\tau_{\alpha^{*}}(s^{*})^2}.
\end{equation}

\end{enumerate}
\end{corollary}

\begin{proof}
Throughout, $\bm a_k(s)=\bm A_k(s^*(s))$, $r(s)=\kappa_\alpha(s)/\tau_\alpha(s)$, $\tau_{\alpha^*}'=\frac{d\tau_{\alpha^*}}{ds^*}$, and the formulas of Theorem~\ref{lem:relationsL3} for Type~$k$ are used without further mention. Since $\frac{ds^*}{ds}>0$, the map $s\mapsto s^*(s)$ is a diffeomorphism onto an interval, so an identity at $s$ is the same as an identity at $s^*=s^*(s)$.

\emph{Sign of $\lambda$.} Since $\kappa_\alpha>0$, the expressions of $\kappa_\alpha$ give $\lambda>0$ in Types~1 and~2, and $\lambda<0$ in Types~3, 4 and~5.

\emph{$\tau_\alpha\neq 0$.} In Types~1 and~4, $\sinh\bm a_k=\mp\frac{ds^*}{ds}\neq 0$; in Types~2 and~5, $\kappa_\alpha>0$ gives $\sinh\bm a_k\neq 0$; in Type~3, $\kappa_\alpha>0$ gives $\sin\bm a_3\neq 0$, and $\cos\bm a_3=\frac{ds^*}{ds}>0$.

\emph{Differentiability of $\bm A_k$.} From the expressions of $\tau_{\alpha^*}$, $\bm A_1=\operatorname{arccoth}(\lambda\tau_{\alpha^*})$, $\bm A_2=\operatorname{arctanh}(\lambda\tau_{\alpha^*})$, $\bm A_3=\arctan(\lambda\tau_{\alpha^*})$ (since $\cos\bm A_3=\frac{ds^*}{ds}>0$), $\bm A_4=\operatorname{arccoth}(-\lambda\tau_{\alpha^*})$ and $\bm A_5=\operatorname{arctanh}(-\lambda\tau_{\alpha^*})$. In particular $\bm A_k$ is differentiable.

\begin{enumerate}
\item \textbf{Formula for $\tau_{\alpha^*}$.}
\begin{align*}
\text{Type 1:}\quad &\dfrac{1}{\lambda}\dfrac{\kappa_{\alpha}}{\tau_{\alpha}}=\dfrac{1}{\lambda}\dfrac{\frac{1}{\lambda}\cosh^2\bm a_1}{\frac{1}{\lambda}\sinh\bm a_1\cosh\bm a_1}=\dfrac{1}{\lambda}\coth\bm a_1=\dfrac{1}{\lambda}\coth\bm A_1=\tau_{\alpha^*}(s^*(s)),\\
\text{Type 2:}\quad &\dfrac{1}{\lambda}\dfrac{\frac{1}{\lambda}\sinh^2\bm a_2}{\frac{1}{\lambda}\sinh\bm a_2\cosh\bm a_2}=\dfrac{1}{\lambda}\tanh\bm a_2=\dfrac{1}{\lambda}\tanh\bm A_2=\tau_{\alpha^*}(s^*(s)),\\
\text{Type 3:}\quad &\dfrac{1}{\lambda}\dfrac{-\frac{1}{\lambda}\sin^2\bm a_3}{-\frac{1}{\lambda}\sin\bm a_3\cos\bm a_3}=\dfrac{1}{\lambda}\tan\bm a_3=\dfrac{1}{\lambda}\tan\bm A_3=\tau_{\alpha^*}(s^*(s)),\\
\text{Type 4:}\quad &\dfrac{1}{\lambda}\dfrac{-\frac{1}{\lambda}\cosh^2\bm a_4}{\frac{1}{\lambda}\sinh\bm a_4\cosh\bm a_4}=-\dfrac{1}{\lambda}\coth\bm a_4=-\dfrac{1}{\lambda}\coth\bm A_4=\tau_{\alpha^*}(s^*(s)),\\
\text{Type 5:}\quad &\dfrac{1}{\lambda}\dfrac{-\frac{1}{\lambda}\sinh^2\bm a_5}{\frac{1}{\lambda}\sinh\bm a_5\cosh\bm a_5}=-\dfrac{1}{\lambda}\tanh\bm a_5=-\dfrac{1}{\lambda}\tanh\bm A_5=\tau_{\alpha^*}(s^*(s)).
\end{align*}
In particular
\begin{equation}\label{eq:r_values}
r=\coth\bm a_1,\ \ \tanh\bm a_2,\ \ \tan\bm a_3,\ \ -\coth\bm a_4,\ \ -\tanh\bm a_5\quad\text{(Types 1--5)},\qquad r(s)=\lambda\,\tau_{\alpha^*}(s^*(s)).
\end{equation}

\item \textbf{$\bm a_k(s)$ in terms of $r$ and $\frac{ds^*}{ds}$; hence $\cosh\bm A_k,\sinh\bm A_k$ (or $\cos\bm A_3,\sin\bm A_3$).} Using \eqref{eq:r_values} and the expression of $\frac{ds^*}{ds}$ in the theorem:
\begin{enumerate}[label=\textbf{Type \arabic*.}, leftmargin=*]
\item $\frac{ds^*}{ds}=-\sinh\bm a_1>0$, so $\bm a_1<0$ and $r=\coth\bm a_1<-1$. From $\coth^2-1=1/\sinh^2$: $\sinh\bm a_1=\frac{-1}{\sqrt{r^2-1}}$, $\cosh\bm a_1=r\sinh\bm a_1=\frac{-r}{\sqrt{r^2-1}}$, $\frac{ds^*}{ds}=\frac{1}{\sqrt{r^2-1}}$.
\item $\frac{ds^*}{ds}=\cosh\bm a_2>0$ and $1-\tanh^2=1/\cosh^2$: $\cosh\bm a_2=\frac{1}{\sqrt{1-r^2}}$, $\sinh\bm a_2=r\cosh\bm a_2=\frac{r}{\sqrt{1-r^2}}$, $\frac{ds^*}{ds}=\frac{1}{\sqrt{1-r^2}}$.
\item $\frac{ds^*}{ds}=\cos\bm a_3>0$ and $1+\tan^2=1/\cos^2$: $\cos\bm a_3=\frac{1}{\sqrt{1+r^2}}$, $\sin\bm a_3=r\cos\bm a_3=\frac{r}{\sqrt{1+r^2}}$, $\frac{ds^*}{ds}=\frac{1}{\sqrt{1+r^2}}$.
\item $\frac{ds^*}{ds}=\sinh\bm a_4>0$, so $\bm a_4>0$ and $r=-\coth\bm a_4<-1$. Then $\sinh\bm a_4=\frac{1}{\sqrt{r^2-1}}$, $\cosh\bm a_4=-r\sinh\bm a_4=\frac{-r}{\sqrt{r^2-1}}$, $\frac{ds^*}{ds}=\frac{1}{\sqrt{r^2-1}}$.
\item $\frac{ds^*}{ds}=\cosh\bm a_5>0$: $\cosh\bm a_5=\frac{1}{\sqrt{1-r^2}}$, $\sinh\bm a_5=-r\cosh\bm a_5=\frac{-r}{\sqrt{1-r^2}}$, $\frac{ds^*}{ds}=\frac{1}{\sqrt{1-r^2}}$.
\end{enumerate}
Evaluating at $s=s(s^*)$ gives the formulas for $\bm A_k(s^*)=\bm a_k(s(s^*))$.

\item \textbf{Formula for $\frac{ds}{ds^*}$.} Since $\frac{ds}{ds^*}=\left(\frac{ds^*}{ds}\right)^{-1}$, item~2 gives, at $s=s(s^*)$, $\frac{ds}{ds^*}=\sqrt{r^2-1}$ in Types~1 and~4, $\sqrt{1-r^2}$ in Types~2 and~5, and $\sqrt{1+r^2}$ in Type~3.

\item \textbf{Formula for $\kappa_{\alpha^*}(s^*)$.} By \eqref{eq:r_values}, $r(s(s^*))=\coth\bm A_1,\allowbreak\ \tanh\bm A_2,\allowbreak\ \tan\bm A_3,\allowbreak\ -\coth\bm A_4,\allowbreak\ -\tanh\bm A_5$. Differentiating with respect to $s^*$ and using $\kappa_{\alpha^*}=\frac{d\bm A_k}{ds^*}$:
\begin{align*}
\text{Type 1:}\quad &\dfrac{d}{ds^*}r(s(s^*))=-\operatorname{csch}^2\bm A_1\,\kappa_{\alpha^*}=-(r^2-1)\,\kappa_{\alpha^*},\\
\text{Type 2:}\quad &\dfrac{d}{ds^*}r(s(s^*))=\operatorname{sech}^2\bm A_2\,\kappa_{\alpha^*}=(1-r^2)\,\kappa_{\alpha^*},\\
\text{Type 3:}\quad &\dfrac{d}{ds^*}r(s(s^*))=\sec^2\bm A_3\,\kappa_{\alpha^*}=(1+r^2)\,\kappa_{\alpha^*},\\
\text{Type 4:}\quad &\dfrac{d}{ds^*}r(s(s^*))=\operatorname{csch}^2\bm A_4\,\kappa_{\alpha^*}=(r^2-1)\,\kappa_{\alpha^*},\\
\text{Type 5:}\quad &\dfrac{d}{ds^*}r(s(s^*))=-\operatorname{sech}^2\bm A_5\,\kappa_{\alpha^*}=-(1-r^2)\,\kappa_{\alpha^*},
\end{align*}
where we used $\operatorname{csch}^2=\coth^2-1$, $\operatorname{sech}^2=1-\tanh^2$ and $\sec^2=1+\tan^2$. The factors $r^2-1$ (Types~1, 4), $1-r^2$ (Types~2, 5) and $1+r^2$ (Type~3) do not vanish, since $|r|>1$ in Types~1 and~4 and $|r|<1$ in Types~2 and~5 (item~2). Solving for $\kappa_{\alpha^*}$ gives the formulas.

\item \textbf{Mannheim equation.} It follows directly from the expressions of $\kappa_\alpha,\tau_\alpha$ in Theorem~\ref{lem:relationsL3}:
\begin{align*}
\text{Types 1, 4:}\quad & \kappa_\alpha^2-\tau_\alpha^2=\dfrac{1}{\lambda^2}\cosh^2\bm a_k\,(\cosh^2\bm a_k-\sinh^2\bm a_k)=\dfrac{1}{\lambda^2}\cosh^2\bm a_k\\
& \qquad=\pm\dfrac{1}{\lambda}\kappa_\alpha\ \ (+\ \text{Type 1},\ -\ \text{Type 4}),\\
\text{Types 2, 5:}\quad & \kappa_\alpha^2-\tau_\alpha^2=\dfrac{1}{\lambda^2}\sinh^2\bm a_k\,(\sinh^2\bm a_k-\cosh^2\bm a_k)=-\dfrac{1}{\lambda^2}\sinh^2\bm a_k\\
& \qquad=\mp\dfrac{1}{\lambda}\kappa_\alpha\ \ (-\ \text{Type 2},\ +\ \text{Type 5}),\\
\text{Type 3:}\quad & \kappa_\alpha^2+\tau_\alpha^2=\dfrac{1}{\lambda^2}\sin^2\bm a_3\,(\sin^2\bm a_3+\cos^2\bm a_3)=\dfrac{1}{\lambda^2}\sin^2\bm a_3=-\dfrac{1}{\lambda}\kappa_\alpha .
\end{align*}

\item \textbf{$\cosh\bm a_k,\sinh\bm a_k$ (or $\cos\bm a_3,\sin\bm a_3$).} By \eqref{eq:r_values}, $r(s)=\lambda\tau_{\alpha^*}(s^*(s))$. Hence the formulas of item~2 for $\bm a_k(s)$, with $r$ replaced by $\lambda\tau_{\alpha^*}(s^*(s))$, give the stated expressions.

\item \textbf{Formula for $\frac{ds^*}{ds}$.} By item~2 with $r=\lambda\tau_{\alpha^*}(s^*(s))$: $\frac{ds^*}{ds}=\frac{1}{\sqrt{\lambda^2\tau_{\alpha^*}^2-1}}$ in Types~1 and~4, $\frac{1}{\sqrt{1-\lambda^2\tau_{\alpha^*}^2}}$ in Types~2 and~5, and $\frac{1}{\sqrt{1+\lambda^2\tau_{\alpha^*}^2}}$ in Type~3.

\item \textbf{Formula for $\kappa_\alpha(s)$.} From the expressions of $\kappa_\alpha$ in Theorem~\ref{lem:relationsL3} and item~6, with $\tau_{\alpha^*}=\tau_{\alpha^*}(s^*(s))$:
\begin{align*}
\text{Type 1:}\quad &\kappa_\alpha=\dfrac{1}{\lambda}\cosh^2\bm a_1=\dfrac{\lambda\tau_{\alpha^*}^2}{\lambda^2\tau_{\alpha^*}^2-1},\\
\text{Type 2:}\quad &\kappa_\alpha=\dfrac{1}{\lambda}\sinh^2\bm a_2=\dfrac{\lambda\tau_{\alpha^*}^2}{1-\lambda^2\tau_{\alpha^*}^2},\\
\text{Type 3:}\quad &\kappa_\alpha=-\dfrac{1}{\lambda}\sin^2\bm a_3=\dfrac{-\lambda\tau_{\alpha^*}^2}{1+\lambda^2\tau_{\alpha^*}^2},\\
\text{Type 4:}\quad &\kappa_\alpha=-\dfrac{1}{\lambda}\cosh^2\bm a_4=\dfrac{-\lambda\tau_{\alpha^*}^2}{\lambda^2\tau_{\alpha^*}^2-1},\\
\text{Type 5:}\quad &\kappa_\alpha=-\dfrac{1}{\lambda}\sinh^2\bm a_5=\dfrac{-\lambda\tau_{\alpha^*}^2}{1-\lambda^2\tau_{\alpha^*}^2}.
\end{align*}

\item \textbf{Formula for $\tau_\alpha(s)$.} From the expressions of $\tau_\alpha$ in Theorem~\ref{lem:relationsL3} and item~6, which fixes the signs:
\begin{align*}
\text{Type 1:}\quad &\tau_\alpha=\dfrac{1}{\lambda}\sinh\bm a_1\cosh\bm a_1=\dfrac{\tau_{\alpha^*}}{\lambda^2\tau_{\alpha^*}^2-1},\\
\text{Type 2:}\quad &\tau_\alpha=\dfrac{1}{\lambda}\sinh\bm a_2\cosh\bm a_2=\dfrac{\tau_{\alpha^*}}{1-\lambda^2\tau_{\alpha^*}^2},\\
\text{Type 3:}\quad &\tau_\alpha=-\dfrac{1}{\lambda}\sin\bm a_3\cos\bm a_3=\dfrac{-\tau_{\alpha^*}}{1+\lambda^2\tau_{\alpha^*}^2},\\
\text{Type 4:}\quad &\tau_\alpha=\dfrac{1}{\lambda}\sinh\bm a_4\cosh\bm a_4=\dfrac{-\tau_{\alpha^*}}{\lambda^2\tau_{\alpha^*}^2-1},\\
\text{Type 5:}\quad &\tau_\alpha=\dfrac{1}{\lambda}\sinh\bm a_5\cosh\bm a_5=\dfrac{-\tau_{\alpha^*}}{1-\lambda^2\tau_{\alpha^*}^2}.
\end{align*}

\item \textbf{Mannheim pair equation.} By Theorem~\ref{lem:relationsL3}, $\kappa_{\alpha^*}=\frac{d\bm A_k}{ds^*}$, and $\lambda\tau_{\alpha^*}=\coth\bm A_1,\allowbreak\ \tanh\bm A_2,\allowbreak\ \tan\bm A_3,\allowbreak\ -\coth\bm A_4,\allowbreak\ -\tanh\bm A_5$. Differentiating in $s^*$:
\begin{align*}
\text{Type 1:}\quad &\lambda\tau_{\alpha^*}'=-\dfrac{\kappa_{\alpha^*}}{\sinh^2\bm A_1}=(1-\lambda^2\tau_{\alpha^*}^2)\,\kappa_{\alpha^*}, \\
\text{Type 2:}\quad &\lambda\tau_{\alpha^*}'=\dfrac{\kappa_{\alpha^*}}{\cosh^2\bm A_2}=(1-\lambda^2\tau_{\alpha^*}^2)\,\kappa_{\alpha^*}, \\
\text{Type 3:}\quad &\lambda\tau_{\alpha^*}'=\dfrac{\kappa_{\alpha^*}}{\cos^2\bm A_3}=(1+\lambda^2\tau_{\alpha^*}^2)\,\kappa_{\alpha^*}, \\
\text{Type 4:}\quad &\lambda\tau_{\alpha^*}'=\dfrac{\kappa_{\alpha^*}}{\sinh^2\bm A_4}=(\lambda^2\tau_{\alpha^*}^2-1)\,\kappa_{\alpha^*}, \\
\text{Type 5:}\quad &\lambda\tau_{\alpha^*}'=-\dfrac{\kappa_{\alpha^*}}{\cosh^2\bm A_5}=-(1-\lambda^2\tau_{\alpha^*}^2)\,\kappa_{\alpha^*}.
\end{align*}
In Types~1 and~4, $1-\lambda^2\tau_{\alpha^*}^2=-1/\sinh^2\bm A_k\neq0$; in Types~2 and~5, $1-\lambda^2\tau_{\alpha^*}^2=1/\cosh^2\bm A_k\neq0$; in Type~3, $1+\lambda^2\tau_{\alpha^*}^2>0$. Solving for $\kappa_{\alpha^*}$ gives the stated formulas. For Types~1, 2, 3, 4 and~5 these are \eqref{mannheim pair equation 1}, \eqref{mannheim pair equation 2}, \eqref{mannheim pair equation 3}, \eqref{mannheim pair equation 4} and \eqref{mannheim pair equation 5}, respectively.
\end{enumerate}
\end{proof}

The following proposition is the Lorentzian version, for the five types, of the known Euclidean fact that a Mannheim curve is a slant helix if and only if its partner is a helix \cite{Ucum}.

\begin{proposition}\label{prop:slant_helix_L3}
Let $(\alpha,\alpha^*)$ be a Mannheim pair of Type~$k$ in $\mathbb{L}^3$, with the conventions of Theorem~\ref{lem:relationsL3}, $\kappa_{\alpha^*}\neq0$ and $\frac{ds^*}{ds}>0$. Let $'=d/ds$ and define
\[
\sigma_k=\frac{\kappa_\alpha^2}{(\kappa_\alpha^2-\tau_\alpha^2)^{3/2}}\Big(\frac{\tau_\alpha}{\kappa_\alpha}\Big)'\ (k=1,4),\qquad
\sigma_k=\frac{\kappa_\alpha^2}{(\tau_\alpha^2-\kappa_\alpha^2)^{3/2}}\Big(\frac{\tau_\alpha}{\kappa_\alpha}\Big)'\ (k=2,5),
\]
\[
\sigma_3=\frac{\kappa_\alpha^2}{(\kappa_\alpha^2+\tau_\alpha^2)^{3/2}}\Big(\frac{\tau_\alpha}{\kappa_\alpha}\Big)' .
\]
Then $\alpha$ is not a helix, $\sigma_k\neq0$, and
\begin{equation}\label{eq:tau_kappa_sigma}
\frac{\tau_{\alpha^*}}{\kappa_{\alpha^*}}(s^*(s))=\frac{\eta_k}{\sigma_k(s)},
\end{equation}
where $\eta_1=\eta_4=-1$, $\eta_2=\eta_5=-\operatorname{sgn}\tau_\alpha$ and $\eta_3=\operatorname{sgn}\tau_\alpha$ ($\operatorname{sgn}\tau_\alpha$ is constant, since $\tau_\alpha\neq0$). Consequently, $\alpha$ is a slant helix if and only if $\alpha^*$ is a helix.
\end{proposition}

\begin{proof}
Recall that a non-null curve is a helix if and only if $\tau/\kappa$ is constant, and that, in the types of \cite{ElsharkawyElshenhab}, $\alpha$ is spacelike with spacelike normal in Types~1 and~5, timelike in Types~2 and~4, and spacelike with timelike normal in Type~3.

Let $r=\kappa_\alpha/\tau_\alpha$; it never vanishes, since $\kappa_\alpha>0$, and $\operatorname{sgn}r=\operatorname{sgn}\tau_\alpha$. Put
\[
q_k=r^2-1\ (k=1,4),\qquad q_k=1-r^2\ (k=2,5),\qquad q_k=1+r^2\ (k=3);
\]
by Corollary~\ref{corolario par s e par s*}, $q_k>0$. Let $\delta_k=-1$ for $k=1,5$ and $\delta_k=1$ for $k=2,3,4$, and $\theta_k=1$ for $k=1,2$ and $\theta_k=-1$ for $k=3,4,5$.

\emph{Step 1.} By \eqref{tau* e tau e kappa L3}, $\tau_{\alpha^*}=r/\lambda$. By \eqref{eq:L3-1-ds}, \eqref{eq:L3-2-ds}, \eqref{eq:L3-3-ds}, \eqref{eq:L3-4-ds}, \eqref{eq:L3-5-ds}, $\frac{ds}{ds^*}=\sqrt{q_k}$, and by \eqref{eq:L3-1-kappa*}, \eqref{eq:L3-2-kappa*}, \eqref{eq:L3-3-kappa*}, \eqref{eq:L3-4-kappa*}, \eqref{eq:L3-5-kappa*}, $\kappa_{\alpha^*}=\frac{\delta_k}{q_k}\frac{dr}{ds^*}$. Since $\frac{dr}{ds^*}=r'\frac{ds}{ds^*}$,
\[
\kappa_{\alpha^*}=\delta_k\,\frac{r'}{\sqrt{q_k}} .
\]
As $\kappa_{\alpha^*}\neq0$, we get $r'\neq0$, so $\tau_\alpha/\kappa_\alpha=1/r$ is not constant and $\alpha$ is not a helix.

\emph{Step 2.} By \eqref{eq:L3-1-kt}, \eqref{eq:L3-2-kt}, \eqref{eq:L3-3-kt}, \eqref{eq:L3-4-kt}, \eqref{eq:L3-5-kt}, with $\tau_{\alpha^*}=r/\lambda$,
\[
\kappa_\alpha=\theta_k\frac{r^2}{\lambda q_k},\qquad \tau_\alpha=\theta_k\frac{r}{\lambda q_k}.
\]
Hence the quantity $w_k$ inside the power $3/2$ in $\sigma_k$ ($\kappa_\alpha^2-\tau_\alpha^2$ for $k=1,4$; $\tau_\alpha^2-\kappa_\alpha^2$ for $k=2,5$; $\kappa_\alpha^2+\tau_\alpha^2$ for $k=3$) equals $\dfrac{r^2}{\lambda^2q_k}>0$, and $\dfrac{\kappa_\alpha^2}{w_k^{3/2}}=\dfrac{|\lambda|\,|r|}{\sqrt{q_k}}$. Since $(\tau_\alpha/\kappa_\alpha)'=(1/r)'=-r'/r^2$,
\[
\sigma_k=-\frac{|\lambda|\,r'}{|r|\sqrt{q_k}}\neq0 .
\]

\emph{Step 3.} From Steps~1 and~2, $\dfrac{r'}{\sqrt{q_k}}=-\dfrac{|r|\sigma_k}{|\lambda|}$, so
\[
\frac{\tau_{\alpha^*}}{\kappa_{\alpha^*}}=\frac{r/\lambda}{\delta_k\,r'/\sqrt{q_k}}=\frac{-\delta_k\operatorname{sgn}(\lambda)\operatorname{sgn}(r)}{\sigma_k}.
\]
Now $\operatorname{sgn}\lambda=+,+,-,-,-$ and $\delta_k=-,+,+,+,-$ for $k=1,\dots,5$, and $\operatorname{sgn}r=-1$ for $k=1,4$ (there $r<-1$) and $\operatorname{sgn}r=\operatorname{sgn}\tau_\alpha$ otherwise. This gives $\eta_k$ as stated, and \eqref{eq:tau_kappa_sigma}.

\emph{Step 4.} In Types~1, 2, 4, 5, $|r|\neq1$, so $\tau_\alpha^2-\kappa_\alpha^2\neq0$. By Ali--L\'opez \cite{AliLopez} (Theorem~1.2 for timelike curves, Theorem~1.3 for spacelike curves with spacelike or timelike normal), $\alpha$ is a slant helix if and only if $\sigma_k$ is constant. This criterion does not depend on the sign convention for $\tau$, since $\tau\mapsto-\tau$ changes $\sigma_k$ into $-\sigma_k$. If $\alpha$ is a slant helix, $\sigma_k$ is a nonzero constant, so by \eqref{eq:tau_kappa_sigma} $\tau_{\alpha^*}/\kappa_{\alpha^*}$ is constant and $\alpha^*$ is a helix. Conversely, if $\tau_{\alpha^*}/\kappa_{\alpha^*}=c$ is constant, then $c\neq0$ (because $\tau_{\alpha^*}=r/\lambda\neq0$), so $\sigma_k=\eta_k/c$ is constant and $\alpha$ is a slant helix.
\end{proof}

The next theorem unifies the characterizations of Mannheim pairs in $\mathbb{L}^3$.
\begin{theorem}[cf.\ \cite{Liu,Kahraman2011,ElsharkawyElshenhab}]\label{theorem caracterization mannheim pair}
If $(\alpha(s),\alpha^{*}(s^{*}))$ is a Mannheim pair (here a prime on a function of $s^{*}$ denotes $d/ds^{*}$), then \begin{equation}\label{Mannheim equation how hypotesis}
\kappa_{\alpha}(s)\bigl[1+\lambda\varepsilon_3\kappa_{\alpha}(s)\bigr]\varepsilon_{1}
+ \lambda(\tau_{\alpha}(s))^2=0,
\end{equation}

and
\begin{equation}\label{eq mannheim partner}
\varepsilon_{2*}\,\tau_{\alpha^{*}}' = \frac{\kappa_{\alpha^{*}}}{\lambda}\Bigl(1 + \lambda^{2}\,\varepsilon_{1*}\varepsilon_{2*}\,\tau_{\alpha^{*}}^{2}\Bigr).
  \end{equation}

Reciprocally, if $\alpha(s)$ is a curve with $\kappa_{\alpha}>0$ and $\tau_{\alpha}\neq0$ satisfying :
\begin{equation}\label{Mannheim equation how hypotesis reciproca}
\kappa_{\alpha}(s)\bigl[1+\lambda\varepsilon_3\kappa_{\alpha}(s)\bigr]\varepsilon_{1}
+ \lambda(\tau_{\alpha}(s))^2 = 0,
\end{equation}
then $\alpha(s)$ is a Mannheim curve, with its partner curve $\alpha^{*}(s^{*})$ given by
\begin{equation}\label{equacao da reciproca alpha*}
\alpha^{*}(s^{*}) = \alpha(s(s^{*})) + \lambda N_{\alpha}(s(s^{*})),
\end{equation}
where $\varepsilon_{1*}=-\operatorname{sgn}\lambda$ and $s(s^{*})$ satisfies
\begin{equation}\label{ds/ds* of statement Mannheim theorem charac}
\left(\frac{ds}{ds^{*}}\right)^{\!2}
=-\dfrac{\varepsilon_{1*}}{\lambda }\dfrac{\kappa_{\alpha}(s(s^{*}))}{\tau_{\alpha}(s(s^{*}))^{2}},
\end{equation}
and if $\alpha^{*}(s^{*})$ is a curve with $\kappa_{\alpha^{*}}>0$ and $\tau_{\alpha^{*}}\neq0$ satisfying:

\begin{equation}\label{eq:partner_condition}
\varepsilon_{2*}\,\tau_{\alpha^{*}}'(s^{*}) = \frac{\kappa_{\alpha^{*}}(s^{*})}{\lambda}\Bigl(1 + \lambda^{2}\,\varepsilon_{1*}\varepsilon_{2*}\,\tau_{\alpha^{*}}^{2}(s^{*})\Bigr),
  \end{equation}
then $\alpha^{*}(s^{*})$ is a Mannheim partner, with its Mannheim curve $\alpha(s)$ given by
\begin{equation}\label{equacao da reciproca alpha}
\alpha(s) = \alpha^{*}(s^{*}(s)) + \lambda B_{\alpha^{*}}(s^{*}(s)),
\end{equation}
with
\begin{equation}
\kappa_{\alpha}(s)=-\lambda\varepsilon_{1*}\left(\tau_{\alpha^{*}}(s^{*}(s))\right)^2\left(\dfrac{ds^{*}}{ds}(s)\right)^2,
\end{equation}
where $s^{*}(s)$ satisfies
\begin{equation}\label{ds/ds* of statement Mannheim theorem charac 2}
\frac{\dfrac{{d^2s^*}}{ds^2}\left(s\right)}{\left(\dfrac{ds^*}{ds}\left(s\right)\right)^2}=-\lambda\,\varepsilon_{1*}\,\kappa_{\alpha^{*}}(s^{*}(s))\,\tau_{\alpha^{*}}(s^{*}(s)),
\end{equation}
and
\begin{equation}\label{eq de alpha pca}
\varepsilon_{1}=\left(\dfrac{ds^{*}}{ds}(s)\right)^2\left[\varepsilon_{1*}+\lambda^2\tau_{\alpha^{*}}(s^{*}(s))^{2}\varepsilon_{2*}\right].
\end{equation}
\end{theorem}

\begin{remark}\label{rem:reciprocal_hyp}
In the reciprocal statements the hypothesis $\tau\neq0$ is needed: a planar curve of constant curvature (a circle) satisfies \eqref{Mannheim equation how hypotesis reciproca} with $\tau=0$, but then $\alpha+\lambda N_\alpha$ is a constant. Moreover, in the last statement we assume $\lambda\varepsilon_{1*}<0$ and $\varepsilon_{1*}+\lambda^2\varepsilon_{2*}\tau_{\alpha^{*}}^2\neq0$, which are the conditions for $\alpha$ in \eqref{equacao da reciproca alpha} to be a regular non-null curve with $\kappa_\alpha>0$ (see the formula for $\kappa_\alpha$ and \eqref{eq de alpha pca}).
\end{remark}

\section{Mannheim pair equation with respect to a generic parameter}\label{sec:mannheim_x}

We can state below the characterization of a Mannheim pair with respect to any other general parameter $t, x$ of the curves.
Before that, we establish a result about the reparametrization of a function.

\begin{proposition}\label{proposicao reparametrizacao nula} A function $f:I\longrightarrow J$, $f(s)$, is identically zero if and only if the function $g=f\circ h:L\subset \mathbb{R}\longrightarrow J$ is identically zero, where $h:L\subset \mathbb{R}\longrightarrow I$ is a diffeomorphism.
\end{proposition}
\begin{proof}
Indeed, let us write $g(t)=(f\circ s)(t),\ s=s(t)$; then

($\Rightarrow$)  If $f(s)=0$, $\forall s\in I$, then  $g(t)=f(s(t))=0$, $\forall t\in L$, hence $g(t)=0$, $\forall t\in L$.

($\Leftarrow$)  If $g(t)=0$, $\forall t\in L$, then  $f(s)=g(t(s))=0$, with $t(s)$ inverse of $s(t)$ ,$\forall s\in I$, hence $f(s)=0$, $\forall s\in I$.
\end{proof}

Hence, let $\beta(t) = \alpha(s(t))$ be a reparametrization of $\alpha$, and $\beta^{*}(t) = \alpha^{*}(s^{*}(s(t)))$. Or consider a reparametrization $\beta^{*}(x) = \alpha^{*}(s^{*}(x))$ of $\alpha^{*}$, and $\beta(x) = \alpha(s(s^{*}(x)))$. Here $\varepsilon_1,\varepsilon_3$ and $\varepsilon_{1*},\varepsilon_{2*}$ are the signs of Theorem~\ref{theorem caracterization mannheim pair}.

\begin{proposition}\label{prop mannheim x geral}
\begin{enumerate}
    \item
    $\alpha(s)$, with $\kappa_\alpha>0$ and $\tau_\alpha\neq0$, is a Mannheim curve if and only if, for some constant $\lambda\neq0$ and any reparametrization $s=s(t)$,
    \begin{equation}\label{Mannheim x geral}
    \kappa_{\beta}(t)=-\lambda \bigl[ \varepsilon_3\,\kappa_{\beta}(t)^2+\varepsilon_1\,\tau_{\beta}(t)^2\bigr].
    \end{equation}
    \item
    Let $\alpha^{*}$ have $\kappa_{\alpha^{*}}>0$ and $\tau_{\alpha^{*}}\neq0$, and assume $\lambda\varepsilon_{1*}<0$ and $\varepsilon_{1*}+\lambda^2\varepsilon_{2*}\tau_{\alpha^{*}}^2\neq0$ (Remark~\ref{rem:reciprocal_hyp}). Then $\alpha^{*}(s^{*})$ is a Mannheim partner if and only if

\begin{equation}\label{eq 2 mannheim pair em x}
 \varepsilon_{2*}\,\dfrac{d}{dx}\bigl(\tau_{\beta^{*}}(x)\bigr)=\dfrac{\kappa_{\beta^{*}}(x)}{\lambda}\Bigl(1+\lambda^2\varepsilon_{1*}\varepsilon_{2*}\bigl[\tau_{\beta^{*}}(x)\bigr]^2\Bigr)\dfrac{ds^{*}}{dx},
\end{equation}
for any reparametrization $s^{*}=s^{*}(x)$.
\end{enumerate}
\end{proposition}
\begin{proof}
\begin{enumerate}
    \item Multiplying equation \eqref{Mannheim equation how hypotesis} by $\varepsilon_1$, $\alpha$ is a Mannheim curve if and only if $f(s)=0$, where $f(s)=\kappa_{\alpha}(s)+\lambda \bigl( \varepsilon_3\,\kappa_{\alpha}(s)^2+\varepsilon_1\,\tau_{\alpha}(s)^2\bigr).$ Then by proposition \ref{proposicao reparametrizacao nula}, $f(s)=0$ if and only if $f(s(t))=0$ for any reparametrization $s=s(t)$, i.e.,

$$\kappa_{\alpha}(s(t))=-\lambda \bigl[ \varepsilon_3\,\kappa_{\alpha}(s(t))^2+\varepsilon_1\,\tau_{\alpha}(s(t))^2\bigr].$$

\item By equation \eqref{eq mannheim partner}, $\alpha^{*}$ is a Mannheim partner if and only if $f(s^{*})=0$, where $f(s^{*})=\varepsilon_{2*}\,\dfrac{d \tau_{\alpha^{*}}}{ds^{*}}(s^{*})-\dfrac{\kappa_{\alpha^{*}}(s^{*})}{\lambda}\Bigl(1+\lambda^2\varepsilon_{1*}\varepsilon_{2*}\,\tau_{\alpha^{*}}(s^{*})^2\Bigr).$ Then by proposition \ref{proposicao reparametrizacao nula}, $f(s^{*})=0$ if and only if $f(s^{*}(x))=0$ for any reparametrization $s^{*}=s^{*}(x)$, i.e. ,
    \begin{equation}
    \varepsilon_{2*}\,\dfrac{d \tau_{\alpha^{*}}}{ds^{*}}(s^{*}(x))=\dfrac{\kappa_{\alpha^{*}}(s^{*}(x))}{\lambda}\Bigl(1+\lambda^2\varepsilon_{1*}\varepsilon_{2*}\,\tau_{\alpha^{*}}(s^{*}(x))^2\Bigr). \end{equation}
    Multiplying both sides by \(\dfrac{ds^{*}}{dx}\) and using the chain rule, $\alpha^*(s^{*})$ is a Mannheim partner iff for any reparametrization $s^{*}=s^{*}(x)$,
\begin{equation}
 \varepsilon_{2*}\,\dfrac{d}{dx}\bigl(\tau_{\alpha^{*}}(s^{*}(x))\bigr)=\dfrac{\kappa_{\alpha^{*}}(s^{*}(x))}{\lambda}\Bigl(1+\lambda^2\varepsilon_{1*}\varepsilon_{2*}\bigl[\tau_{\alpha^{*}}(s^{*}(x))\bigr]^2\Bigr)\dfrac{ds^{*}}{dx}.
\end{equation}

\end{enumerate}
\end{proof}

From one curve of the Mannheim pair, we obtain the other curve of the pair and its respective Frenet apparatus, taking as reference the parameter and the Frenet apparatus of the other curve. This is formalized in the following proposition.

\begin{proposition}\label{proposition of pair generic parameter L3}
Let $(\alpha(s),\alpha^{*}(s^{*}))$ be a Mannheim pair of Type~$k$ in $\mathbb{L}^3$, with the conventions of Theorem~\ref{lem:relationsL3}, and suppose $\dfrac{ds^{*}}{ds}>0$. Then $\tau_\alpha\neq0$ and
\begin{enumerate}
    \item If $t$ is a generic parameter for the curve $\alpha$, with $\dfrac{ds}{dt}>0$, $\beta(t)=\alpha(s(t))$ and $\beta^{*}(t)=\alpha^{*}(s^{*}(s(t)))$, then
    \begin{enumerate}
    \item
\begin{equation}\label{alpha* to s*os(x) L3}
\beta^*(t)= \beta(t) + \lambda N_\beta(t).
\end{equation}
    \item
    \begin{equation}\label{B*=-N to s(x) L3}
     B_{\beta^{*}}(t)=-N_{\beta}(t).
    \end{equation}
    \item $T_{\beta^{*}}$, $N_{\beta^{*}}$, $\tau_{\beta^{*}}$ and $\kappa_{\beta^{*}}$ are given, according to the type, by:
    \begin{enumerate}[label=\textbf{Type \arabic*.}, leftmargin=*]
    \item ($\lambda>0$)
\begin{equation}\label{eq:gp1-1-a}
\begin{aligned}
T_{\beta^{*}}(t)&=-\dfrac{1}{\sqrt{\left(\dfrac{\kappa_{\beta}(t)}{\tau_{\beta}(t)}\right)^2-1}}\left(T_\beta(t)+\dfrac{\kappa_{\beta}(t)}{\tau_{\beta}(t)}B_\beta(t)\right),\\
N_{\beta^{*}}(t)&=-\dfrac{1}{\sqrt{\left(\dfrac{\kappa_{\beta}(t)}{\tau_{\beta}(t)}\right)^2-1}}\left(\dfrac{\kappa_{\beta}(t)}{\tau_{\beta}(t)}T_\beta(t)+B_\beta(t)\right),
\end{aligned}
\end{equation}
\begin{equation}\label{eq:gp1-1-b}
\begin{aligned}
\tau_{\beta^{*}}(t)&=\dfrac{1}{\lambda}\dfrac{\kappa_\beta(t)}{\tau_\beta(t)}, &
\kappa_{\beta^{*}}(t)&=-\dfrac{\dfrac{d}{dt}\left(\dfrac{\kappa_{\beta}(t)}{\tau_{\beta}(t)}\right)}{\dfrac{ds}{dt}\sqrt{\left(\dfrac{\kappa_{\beta}(t)}{\tau_{\beta}(t)}\right)^2-1}}.
\end{aligned}
\end{equation}
    \item ($\lambda>0$)
\begin{equation}\label{eq:gp1-2-a}
\begin{aligned}
T_{\beta^{*}}(t)&=\dfrac{1}{\sqrt{1-\left(\dfrac{\kappa_{\beta}(t)}{\tau_{\beta}(t)}\right)^2}}\left(T_\beta(t)+\dfrac{\kappa_{\beta}(t)}{\tau_{\beta}(t)}B_\beta(t)\right),\\
N_{\beta^{*}}(t)&=\dfrac{1}{\sqrt{1-\left(\dfrac{\kappa_{\beta}(t)}{\tau_{\beta}(t)}\right)^2}}\left(\dfrac{\kappa_{\beta}(t)}{\tau_{\beta}(t)}T_\beta(t)+B_\beta(t)\right),
\end{aligned}
\end{equation}
\begin{equation}\label{eq:gp1-2-b}
\begin{aligned}
\tau_{\beta^{*}}(t)&=\dfrac{1}{\lambda}\dfrac{\kappa_\beta(t)}{\tau_\beta(t)}, &
\kappa_{\beta^{*}}(t)&=\dfrac{\dfrac{d}{dt}\left(\dfrac{\kappa_{\beta}(t)}{\tau_{\beta}(t)}\right)}{\dfrac{ds}{dt}\sqrt{1-\left(\dfrac{\kappa_{\beta}(t)}{\tau_{\beta}(t)}\right)^2}}.
\end{aligned}
\end{equation}
    \item ($\lambda<0$)
\begin{equation}\label{eq:gp1-3-a}
\begin{aligned}
T_{\beta^{*}}(t)&=\dfrac{1}{\sqrt{1+\left(\dfrac{\kappa_{\beta}(t)}{\tau_{\beta}(t)}\right)^2}}\left(T_\beta(t)-\dfrac{\kappa_{\beta}(t)}{\tau_{\beta}(t)}B_\beta(t)\right),\\
N_{\beta^{*}}(t)&=-\dfrac{1}{\sqrt{1+\left(\dfrac{\kappa_{\beta}(t)}{\tau_{\beta}(t)}\right)^2}}\left(\dfrac{\kappa_{\beta}(t)}{\tau_{\beta}(t)}T_\beta(t)+B_\beta(t)\right),
\end{aligned}
\end{equation}
\begin{equation}\label{eq:gp1-3-b}
\begin{aligned}
\tau_{\beta^{*}}(t)&=\dfrac{1}{\lambda}\dfrac{\kappa_\beta(t)}{\tau_\beta(t)}, &
\kappa_{\beta^{*}}(t)&=\dfrac{\dfrac{d}{dt}\left(\dfrac{\kappa_{\beta}(t)}{\tau_{\beta}(t)}\right)}{\dfrac{ds}{dt}\sqrt{1+\left(\dfrac{\kappa_{\beta}(t)}{\tau_{\beta}(t)}\right)^2}}.
\end{aligned}
\end{equation}
    \item ($\lambda<0$)
\begin{equation}\label{eq:gp1-4-a}
\begin{aligned}
T_{\beta^{*}}(t)&=-\dfrac{1}{\sqrt{\left(\dfrac{\kappa_{\beta}(t)}{\tau_{\beta}(t)}\right)^2-1}}\left(T_\beta(t)+\dfrac{\kappa_{\beta}(t)}{\tau_{\beta}(t)}B_\beta(t)\right),\\
N_{\beta^{*}}(t)&=\dfrac{1}{\sqrt{\left(\dfrac{\kappa_{\beta}(t)}{\tau_{\beta}(t)}\right)^2-1}}\left(\dfrac{\kappa_{\beta}(t)}{\tau_{\beta}(t)}T_\beta(t)+B_\beta(t)\right),
\end{aligned}
\end{equation}
\begin{equation}\label{eq:gp1-4-b}
\begin{aligned}
\tau_{\beta^{*}}(t)&=\dfrac{1}{\lambda}\dfrac{\kappa_\beta(t)}{\tau_\beta(t)}, &
\kappa_{\beta^{*}}(t)&=\dfrac{\dfrac{d}{dt}\left(\dfrac{\kappa_{\beta}(t)}{\tau_{\beta}(t)}\right)}{\dfrac{ds}{dt}\sqrt{\left(\dfrac{\kappa_{\beta}(t)}{\tau_{\beta}(t)}\right)^2-1}}.
\end{aligned}
\end{equation}
    \item ($\lambda<0$)
\begin{equation}\label{eq:gp1-5-a}
\begin{aligned}
T_{\beta^{*}}(t)&=\dfrac{1}{\sqrt{1-\left(\dfrac{\kappa_{\beta}(t)}{\tau_{\beta}(t)}\right)^2}}\left(T_\beta(t)+\dfrac{\kappa_{\beta}(t)}{\tau_{\beta}(t)}B_\beta(t)\right),\\
N_{\beta^{*}}(t)&=-\dfrac{1}{\sqrt{1-\left(\dfrac{\kappa_{\beta}(t)}{\tau_{\beta}(t)}\right)^2}}\left(\dfrac{\kappa_{\beta}(t)}{\tau_{\beta}(t)}T_\beta(t)+B_\beta(t)\right),
\end{aligned}
\end{equation}
\begin{equation}\label{eq:gp1-5-b}
\begin{aligned}
\tau_{\beta^{*}}(t)&=\dfrac{1}{\lambda}\dfrac{\kappa_\beta(t)}{\tau_\beta(t)}, &
\kappa_{\beta^{*}}(t)&=-\dfrac{\dfrac{d}{dt}\left(\dfrac{\kappa_{\beta}(t)}{\tau_{\beta}(t)}\right)}{\dfrac{ds}{dt}\sqrt{1-\left(\dfrac{\kappa_{\beta}(t)}{\tau_{\beta}(t)}\right)^2}}.
\end{aligned}
\end{equation}
    \end{enumerate}
\end{enumerate}

 \item If $x$ is a generic parameter for the curve $\alpha^{*}$, with $\dfrac{ds^{*}}{dx}>0$, $\beta^{*}(x)=\alpha^{*}(s^{*}(x))$ and $\beta(x)=\alpha(s(s^{*}(x)))$, then
 \begin{enumerate}
 \item \begin{equation}\label{eq:alpha_x 2 L3}
\beta(x)=\beta^{*}(x)+\lambda B_{\beta^{*}}(x).
\end{equation}
     \item  \begin{equation}\label{N=-B* to s*(x) L3}
     N_{\beta}(x)=-B_{\beta^{*}}(x).
    \end{equation}
     \item $T_{\beta}$, $B_{\beta}$, $\tau_{\beta}$ and $\kappa_{\beta}$ are given, according to the type, by:
    \begin{enumerate}[label=\textbf{Type \arabic*.}, leftmargin=*]
    \item ($\lambda>0$)
\begin{equation}\label{eq:gp2-1-a}
\begin{aligned}
T_{\beta}(x)&=\dfrac{1}{\sqrt{\lambda^2\tau_{\beta^{*}}(x)^2-1}}\Big(T_{\beta^{*}}(x)-\lambda\tau_{\beta^{*}}(x)N_{\beta^{*}}(x)\Big),\\
B_{\beta}(x)&=\dfrac{1}{\sqrt{\lambda^2\tau_{\beta^{*}}(x)^2-1}}\Big(-\lambda\tau_{\beta^{*}}(x)T_{\beta^{*}}(x)+N_{\beta^{*}}(x)\Big),
\end{aligned}
\end{equation}
\begin{equation}\label{eq:gp2-1-b}
\begin{aligned}
\tau_{\beta}(x)&=\dfrac{\tau_{\beta^{*}}(x)}{\lambda^2\tau_{\beta^{*}}(x)^2-1}, &
\kappa_{\beta}(x)&=\dfrac{\lambda\tau_{\beta^{*}}(x)^2}{\lambda^2\tau_{\beta^{*}}(x)^2-1}.
\end{aligned}
\end{equation}
    \item ($\lambda>0$)
\begin{equation}\label{eq:gp2-2-a}
\begin{aligned}
T_{\beta}(x)&=\dfrac{1}{\sqrt{1-\lambda^2\tau_{\beta^{*}}(x)^2}}\Big(T_{\beta^{*}}(x)-\lambda\tau_{\beta^{*}}(x)N_{\beta^{*}}(x)\Big),\\
B_{\beta}(x)&=\dfrac{1}{\sqrt{1-\lambda^2\tau_{\beta^{*}}(x)^2}}\Big(-\lambda\tau_{\beta^{*}}(x)T_{\beta^{*}}(x)+N_{\beta^{*}}(x)\Big),
\end{aligned}
\end{equation}
\begin{equation}\label{eq:gp2-2-b}
\begin{aligned}
\tau_{\beta}(x)&=\dfrac{\tau_{\beta^{*}}(x)}{1-\lambda^2\tau_{\beta^{*}}(x)^2}, &
\kappa_{\beta}(x)&=\dfrac{\lambda\tau_{\beta^{*}}(x)^2}{1-\lambda^2\tau_{\beta^{*}}(x)^2}.
\end{aligned}
\end{equation}
    \item ($\lambda<0$)
\begin{equation}\label{eq:gp2-3-a}
\begin{aligned}
T_{\beta}(x)&=\dfrac{1}{\sqrt{1+\lambda^2\tau_{\beta^{*}}(x)^2}}\Big(T_{\beta^{*}}(x)-\lambda\tau_{\beta^{*}}(x)N_{\beta^{*}}(x)\Big),\\
B_{\beta}(x)&=-\dfrac{1}{\sqrt{1+\lambda^2\tau_{\beta^{*}}(x)^2}}\Big(\lambda\tau_{\beta^{*}}(x)T_{\beta^{*}}(x)+N_{\beta^{*}}(x)\Big),
\end{aligned}
\end{equation}
\begin{equation}\label{eq:gp2-3-b}
\begin{aligned}
\tau_{\beta}(x)&=-\dfrac{\tau_{\beta^{*}}(x)}{1+\lambda^2\tau_{\beta^{*}}(x)^2}, &
\kappa_{\beta}(x)&=-\dfrac{\lambda\tau_{\beta^{*}}(x)^2}{1+\lambda^2\tau_{\beta^{*}}(x)^2}.
\end{aligned}
\end{equation}
    \item ($\lambda<0$)
\begin{equation}\label{eq:gp2-4-a}
\begin{aligned}
T_{\beta}(x)&=\dfrac{1}{\sqrt{\lambda^2\tau_{\beta^{*}}(x)^2-1}}\Big(T_{\beta^{*}}(x)+\lambda\tau_{\beta^{*}}(x)N_{\beta^{*}}(x)\Big),\\
B_{\beta}(x)&=-\dfrac{1}{\sqrt{\lambda^2\tau_{\beta^{*}}(x)^2-1}}\Big(\lambda\tau_{\beta^{*}}(x)T_{\beta^{*}}(x)+N_{\beta^{*}}(x)\Big),
\end{aligned}
\end{equation}
\begin{equation}\label{eq:gp2-4-b}
\begin{aligned}
\tau_{\beta}(x)&=-\dfrac{\tau_{\beta^{*}}(x)}{\lambda^2\tau_{\beta^{*}}(x)^2-1}, &
\kappa_{\beta}(x)&=-\dfrac{\lambda\tau_{\beta^{*}}(x)^2}{\lambda^2\tau_{\beta^{*}}(x)^2-1}.
\end{aligned}
\end{equation}
    \item ($\lambda<0$)
\begin{equation}\label{eq:gp2-5-a}
\begin{aligned}
T_{\beta}(x)&=\dfrac{1}{\sqrt{1-\lambda^2\tau_{\beta^{*}}(x)^2}}\Big(T_{\beta^{*}}(x)+\lambda\tau_{\beta^{*}}(x)N_{\beta^{*}}(x)\Big),\\
B_{\beta}(x)&=-\dfrac{1}{\sqrt{1-\lambda^2\tau_{\beta^{*}}(x)^2}}\Big(\lambda\tau_{\beta^{*}}(x)T_{\beta^{*}}(x)+N_{\beta^{*}}(x)\Big),
\end{aligned}
\end{equation}
\begin{equation}\label{eq:gp2-5-b}
\begin{aligned}
\tau_{\beta}(x)&=-\dfrac{\tau_{\beta^{*}}(x)}{1-\lambda^2\tau_{\beta^{*}}(x)^2}, &
\kappa_{\beta}(x)&=-\dfrac{\lambda\tau_{\beta^{*}}(x)^2}{1-\lambda^2\tau_{\beta^{*}}(x)^2}.
\end{aligned}
\end{equation}
    \end{enumerate}
 \end{enumerate}
\end{enumerate}
\end{proposition}

\begin{proof}
By Corollary~\ref{corolario par s e par s*}, $\tau_\alpha\neq 0$, so the quotients $\kappa_\beta/\tau_\beta$ are defined. Since $\beta$ and $\beta^{*}$ are reparametrizations of $\alpha$ and $\alpha^{*}$ with positive derivative, their Frenet apparatus are those of $\alpha$ and $\alpha^{*}$ at the corresponding parameter, by \eqref{tangente at t}, \eqref{kappa at t}, \eqref{normal at t}, \eqref{binormal at t}, \eqref{torsion at t}.
\begin{enumerate}
    \item
    \begin{enumerate}
    \item From equation \eqref{eq:alpha_star} at $s^{*}=s^{*}(s(t))$, using $s(s^{*}(s(t)))=s(t)$,
$$ \alpha^*(s^*(s(t))) = \alpha(s(t)) + \lambda N_\alpha(s(t)),$$
and by equation \eqref{normal at t}
$$ \beta^{*}(t) = \beta(t) + \lambda N_\beta(t).$$
    \item From equation \eqref{fundamental hypotesis Mannheim Pair} at $s=s(t)$
    $$B_{\alpha^{*}}(s^{*}(s(t)))=-N_{\alpha}(s(t)),$$
    and by equation \eqref{binormal at t}
     $$B_{\beta^{*}}(t)=-N_{\beta}(t).$$
    \item \textbf{Type 1.} From equations \eqref{eq:T_star_1} and \eqref{eq:L3-1-A} at $s^{*}=s^{*}(s(t))$, where $\bm A_1(s^{*}(s(t)))=\bm a_1(s(t))$,
$$T_{\alpha^*}(s^{*}(s(t))) = -\dfrac{1}{\sqrt{\left(\dfrac{\kappa_{\alpha}(s(t))}{\tau_{\alpha}(s(t))}\right)^2-1}}\left(T_\alpha(s(t))+\dfrac{\kappa_{\alpha}(s(t))}{\tau_{\alpha}(s(t))}B_\alpha(s(t))\right),$$
and by equations \eqref{tangente at t}, \eqref{kappa at t}, \eqref{binormal at t}, \eqref{torsion at t} we obtain the formula for $T_{\beta^{*}}(t)$. The formula for $N_{\beta^{*}}(t)$ follows in the same way from \eqref{eq:N_star_1}. The formula for $\tau_{\beta^{*}}(t)$ follows from equation \eqref{tau* e tau e kappa L3} at $s^{*}=s^{*}(s(t))$. For $\kappa_{\beta^{*}}$, put $g(s^{*})=\dfrac{\kappa_{\alpha}(s(s^{*}))}{\tau_{\alpha}(s(s^{*}))}$. Then $\dfrac{\kappa_\beta(t)}{\tau_\beta(t)}=g(s^{*}(s(t)))$ and, by the chain rule and \eqref{eq:L3-1-ds},
$$\dfrac{d}{dt}\left(\dfrac{\kappa_\beta}{\tau_\beta}\right)=g'(s^{*})\,\dfrac{ds^{*}}{ds}\,\dfrac{ds}{dt}=\dfrac{g'(s^{*})}{\sqrt{g(s^{*})^2-1}}\,\dfrac{ds}{dt},$$
so $g'(s^{*}(s(t)))=\sqrt{\left(\dfrac{\kappa_\beta}{\tau_\beta}\right)^2-1}\,\dfrac{\frac{d}{dt}(\kappa_\beta/\tau_\beta)}{ds/dt}$. Then, by \eqref{eq:L3-1-kappa*} at $s^{*}=s^{*}(s(t))$,
$$\kappa_{\beta^{*}}(t)=\kappa_{\alpha^{*}}(s^{*}(s(t)))=\dfrac{-g'(s^{*}(s(t)))}{g(s^{*}(s(t)))^2-1}=-\dfrac{\dfrac{d}{dt}\left(\dfrac{\kappa_{\beta}(t)}{\tau_{\beta}(t)}\right)}{\dfrac{ds}{dt}\sqrt{\left(\dfrac{\kappa_{\beta}(t)}{\tau_{\beta}(t)}\right)^2-1}}.$$

\textbf{Types 2 to 5.} The argument is the same, using the equations of the corresponding type: \eqref{eq:T_star_2}, \eqref{eq:N_star_2}, \eqref{eq:L3-2-A}, \eqref{eq:L3-2-ds} and \eqref{eq:L3-2-kappa*} for Type~2; \eqref{eq:T_star_3}, \eqref{eq:N_star_3}, \eqref{eq:L3-3-A}, \eqref{eq:L3-3-ds} and \eqref{eq:L3-3-kappa*} for Type~3; \eqref{eq:T_star_4}, \eqref{eq:N_star_4}, \eqref{eq:L3-4-A}, \eqref{eq:L3-4-ds} and \eqref{eq:L3-4-kappa*} for Type~4; \eqref{eq:T_star_5}, \eqref{eq:N_star_5}, \eqref{eq:L3-5-A}, \eqref{eq:L3-5-ds} and \eqref{eq:L3-5-kappa*} for Type~5. The only changes are the square root in $\frac{ds^*}{ds}$, which is $\sqrt{g^2-1}$ in Types~1 and~4, $\sqrt{1-g^2}$ in Types~2 and~5 and $\sqrt{1+g^2}$ in Type~3, and the sign of $\kappa_{\alpha^*}$, which is $-$ in Types~1 and~5 and $+$ in Types~2, 3 and~4.
    \end{enumerate}
    \item
\begin{enumerate}
        \item  From equation \eqref{eq:alpha} at $s=s(s^{*}(x))$, using $s^{*}(s(s^{*}(x)))=s^{*}(x)$,
$$\alpha(s(s^{*}(x)))=\alpha^*(s^{*}(x))+\lambda B_{\alpha^{*}}(s^{*}(x)),$$
therefore
$$\beta(x)=\beta^*(x)+\lambda B_{\beta^{*}}(x).$$
\item From equation \eqref{fundamental hypotesis Mannheim Pair} at $s=s(s^{*}(x))$
    $$N_{\alpha}(s(s^{*}(x)))=-B_{\alpha^{*}}(s^{*}(x)),$$
    therefore
     $$N_{\beta}(x)=-B_{\beta^{*}}(x).$$
\item \textbf{Type 1.} From equations \eqref{eq:T_alpha_1} and \eqref{eq:L3-1-a} at $s=s(s^{*}(x))$, where $\bm a_1(s(s^{*}(x)))=\bm A_1(s^{*}(x))$,
   $$T_\alpha(s(s^{*}(x)))=\dfrac{1}{\sqrt{\lambda^2\tau_{\alpha^*}(s^{*}(x))^2-1}}\Big(T_{\alpha^*}(s^{*}(x))-\lambda\tau_{\alpha^*}(s^{*}(x))N_{\alpha^*}(s^{*}(x))\Big),$$
therefore, since $\tau_{\alpha^*}(s^{*}(x))=\tau_{\beta^*}(x)$,
$$T_\beta(x)=\dfrac{1}{\sqrt{\lambda^2\tau_{\beta^*}(x)^2-1}}\Big(T_{\beta^*}(x)-\lambda\tau_{\beta^*}(x)N_{\beta^*}(x)\Big).$$
The formula for $B_\beta(x)$ follows in the same way from \eqref{eq:B_alpha_1}. The formulas for $\tau_\beta(x)$ and $\kappa_\beta(x)$ follow from equation \eqref{eq:L3-1-kt} at $s=s(s^{*}(x))$.

\textbf{Types 2 to 5.} The argument is the same, using the equations of the corresponding type: \eqref{eq:T_alpha_2}, \eqref{eq:B_alpha_2}, \eqref{eq:L3-2-a} and \eqref{eq:L3-2-kt} for Type~2; \eqref{eq:T_alpha_3}, \eqref{eq:B_alpha_3}, \eqref{eq:L3-3-a} and \eqref{eq:L3-3-kt} for Type~3; \eqref{eq:T_alpha_4}, \eqref{eq:B_alpha_4}, \eqref{eq:L3-4-a} and \eqref{eq:L3-4-kt} for Type~4; \eqref{eq:T_alpha_5}, \eqref{eq:B_alpha_5}, \eqref{eq:L3-5-a} and \eqref{eq:L3-5-kt} for Type~5.
\end{enumerate}
\end{enumerate}
\end{proof}

\section{Auxiliary angle function to Mannheim condition}\label{sec:choice_phi mannheim condition}

Our goal is to choose a convenient parameter $t$ and find functions $\Theta(t)$ and $\Phi(t)$ such that $\beta$ satisfies the Mannheim condition and the resulting integrals are elementary. The method is the same whether $T_\alpha$, $N_\alpha$ or $B_\alpha$ is timelike, so we describe it once and then collect the five types of Mannheim pair in a single theorem.

Throughout this section, $\alpha$ is not a helix, $\dfrac{ds}{dt}>0$ and $\Phi'(t)>0$. The method has four steps:
\begin{enumerate}
\item define an auxiliary angle $\delta(t)$ in terms of $\Theta'$, $\Phi'$ and $\Phi$, namely $\cot\delta=\dfrac{\Theta'\sin\Phi}{\Phi'}$ if $T_\alpha$ is timelike, $\tanh\delta=\dfrac{\Theta'\cosh\Phi}{\Phi'}$ if $N_\alpha$ is timelike and $\coth\delta=\dfrac{\Theta'\cosh\Phi}{\Phi'}$ if $B_\alpha$ is timelike;
\item write $\kappa_{\beta}(t)=\kappa_\alpha(s(t))$ and $\tau_{\beta}(t)=\tau_\alpha(s(t))$ in terms of $\Phi$, $\delta$ and $\delta'$, and substitute them in the Mannheim condition;
\item solve the resulting equation for $\dfrac{ds}{dt}$;
\item recover $\Theta(t)$ from the relation that defines $\delta$.
\end{enumerate}

The four steps are carried out in the proof of the theorem below, which collects the five types. The causal character of $\alpha$ and the sign of $\lambda$ determine the type of the pair: Types~2 and~4 correspond to $T_\alpha$ timelike, Type~3 to $N_\alpha$ timelike and Types~1 and~5 to $B_\alpha$ timelike.

\begin{theorem}\label{thm:general_mannheim}
Let $\lambda\neq0$ and $m$ be real constants and let $\delta(t)$ and $\Phi(t)$, with $\Phi'(t)>0$, be differentiable functions. Let $\kg(t)\neq0$, $\Theta(t)$, $\dfrac{ds}{dt}$, $\mathbf{u}(t)$, $\kappa_\beta(t)$ and $\tau_\beta(t)$ be defined, according to the type, as follows (in each item $\delta$, $\Phi$ and $\lambda$ are such that $\dfrac{ds}{dt}>0$; this fixes the sign of $\lambda$ and the range of $|\kg|$ indicated).
\begin{enumerate}[label=\textbf{Type \arabic*.}, leftmargin=*]
\item ($B_\beta$ timelike, $T_{\beta^*}$ timelike; $\lambda>0$, $\delta>0$, $|\kg|<1$)
\begin{align}
\kg &= \cosh\delta\,\tanh\Phi + \dfrac{\delta'}{\Phi'}\sinh\delta, \label{F Type 1}\\
\Theta &= \int \Phi'\,\frac{\coth\delta}{\cosh\Phi}\,dt + m, \label{Theta Type 1}\\
\dfrac{ds}{dt} &= \lambda\,\Phi'\operatorname{csch}\delta\,\big(1-\kg^2\big), \label{dsdx Type 1}\\
\mathbf{u} &= \big(\cos\Theta\cosh\Phi,\ \sin\Theta\cosh\Phi,\ \sinh\Phi\big), \nonumber\\
\kappa_{\beta} &= \frac{1}{\lambda\big(1-\kg^2\big)},\qquad
\tau_{\beta} = \frac{\kg}{\lambda\big(1-\kg^2\big)}. \label{k e tau Type 1}
\end{align}

\item ($T_\beta$ timelike, $T_{\beta^*}$ timelike; $\lambda>0$, $\delta\in(0,\pi)$, $\Phi\in(-\frac\pi2,\frac\pi2)\setminus\{0\}$, $|\kg|>1$)
\begin{align}
\kg &= \cos\delta\,\csc\Phi - \dfrac{\sin\delta\,\cos\Phi\,\delta'}{\Phi'}, \label{F Type 2}\\
\Theta &= \int \Phi'\csc\Phi\cot\delta\,dt + m, \label{Theta Type 2}\\
\dfrac{ds}{dt} &= \lambda\csc\delta\,\sec\Phi\,\Phi'\,\big(\kg^2-1\big), \label{dsdx Type 2}\\
\mathbf{u} &= \big(\cos\Theta\tan\Phi,\ \sin\Theta\tan\Phi,\ \sec\Phi\big), \nonumber\\
\kappa_{\beta} &= \frac{1}{\lambda\big(\kg^2-1\big)},\qquad
\tau_{\beta} = \frac{\kg}{\lambda\big(\kg^2-1\big)}. \label{k e tau Type 2}
\end{align}

\item ($N_\beta$ timelike, $B_{\beta^*}$ timelike; $\lambda<0$)
\begin{align}
\kg &= -\sinh\delta\,\tanh\Phi - \dfrac{\delta'}{\Phi'}\cosh\delta, \label{F Type 3}\\
\Theta &= \int \Phi'\,\frac{\tanh\delta}{\cosh\Phi}\,dt + m, \label{Theta Type 3}\\
\dfrac{ds}{dt} &= (-\lambda)\,\Phi'\operatorname{sech}\delta\,\big(1+\kg^2\big), \label{dsdx Type 3}\\
\mathbf{u} &= \big(\cos\Theta\cosh\Phi,\ \sin\Theta\cosh\Phi,\ \sinh\Phi\big), \nonumber\\
\kappa_{\beta} &= \frac{1}{(-\lambda)\big(1+\kg^2\big)},\qquad
\tau_{\beta} = \frac{\kg}{(-\lambda)\big(1+\kg^2\big)}. \label{k e tau Type 3}
\end{align}

\item ($T_\beta$ timelike, $N_{\beta^*}$ timelike; $\lambda<0$, $\delta\in(0,\pi)$, $\Phi\in(-\frac\pi2,\frac\pi2)\setminus\{0\}$, $|\kg|<1$)
\begin{align}
\kg &= \cos\delta\,\csc\Phi - \dfrac{\sin\delta\,\cos\Phi\,\delta'}{\Phi'}, \label{F Type 4}\\
\Theta &= \int \Phi'\csc\Phi\cot\delta\,dt + m, \label{Theta Type 4}\\
\dfrac{ds}{dt} &= \lambda\csc\delta\,\sec\Phi\,\Phi'\,\big(\kg^2-1\big), \label{dsdx Type 4}\\
\mathbf{u} &= \big(\cos\Theta\tan\Phi,\ \sin\Theta\tan\Phi,\ \sec\Phi\big), \nonumber\\
\kappa_{\beta} &= \frac{1}{\lambda\big(\kg^2-1\big)},\qquad
\tau_{\beta} = \frac{\kg}{\lambda\big(\kg^2-1\big)}. \label{k e tau Type 4}
\end{align}

\item ($B_\beta$ timelike, $N_{\beta^*}$ timelike; $\lambda<0$, $\delta>0$, $|\kg|>1$)
\begin{align}
\kg &= \cosh\delta\,\tanh\Phi + \dfrac{\delta'}{\Phi'}\sinh\delta, \label{F Type 5}\\
\Theta &= \int \Phi'\,\frac{\coth\delta}{\cosh\Phi}\,dt + m, \label{Theta Type 5}\\
\dfrac{ds}{dt} &= \lambda\,\Phi'\operatorname{csch}\delta\,\big(1-\kg^2\big), \label{dsdx Type 5}\\
\mathbf{u} &= \big(\cos\Theta\cosh\Phi,\ \sin\Theta\cosh\Phi,\ \sinh\Phi\big), \nonumber\\
\kappa_{\beta} &= \frac{1}{\lambda\big(1-\kg^2\big)},\qquad
\tau_{\beta} = \frac{\kg}{\lambda\big(1-\kg^2\big)}. \label{k e tau Type 5}
\end{align}
\end{enumerate}
Then the curve $\beta$ given by
\begin{equation}\label{d alpha s}
\frac{d\beta}{dt}(t) = \frac{ds}{dt}\,\mathbf{u}(t)
\end{equation}
is a Mannheim curve of Type~$k$ (up to the normalization of Remark~\ref{rem:normalization}), with partner $\beta^{*}=\beta+\lambda N_\beta$, and its curvature and torsion are those listed in the corresponding item. In every type, $\tau_\beta=\kg\,\kappa_\beta$, that is, $\kg=\tau_\beta/\kappa_\beta$. Conversely, let $\beta$ be a Mannheim curve of Type~$k$, parametrized by $t$ with $\dfrac{ds}{dt}>0$, whose tangent $T_\beta=\mathbf{u}$ has angular coordinates $(\Theta,\Phi)$ with $\Phi'>0$ (and, in Types~2 and~4, $\Phi\in(-\frac\pi2,\frac\pi2)\setminus\{0\}$; in Types~1 and~5, $\Theta'>0$). Then there exist $\delta(t)$ as in the corresponding item such that the formulas of Type~$k$ and \eqref{d alpha s} hold.
\end{theorem}

\begin{proof}
Write $D=\dfrac{ds}{dt}>0$. Types~2 and~4 have $T_\beta$ timelike (case~(i)), Type~3 has $N_\beta$ timelike (case~(ii)) and Types~1 and~5 have $B_\beta$ timelike (case~(iii)). Let
\begin{align*}
\text{(i)}\quad & K=\Phi'\sec\Phi\csc\delta, &
\text{(ii)}\quad & K=\Phi'\operatorname{sech}\delta, &
\text{(iii)}\quad & K=\Phi'\operatorname{csch}\delta,
\end{align*}
so $K>0$ under the hypotheses of the statement.

\emph{Step 1: the relation defining $\delta$.} The auxiliary angle is defined, in each case, by
\begin{equation}\label{eq:def_delta}
\text{(i)}\ \cot\delta=\frac{\Theta'\sin\Phi}{\Phi'},\qquad
\text{(ii)}\ \tanh\delta=\frac{\Theta'\cosh\Phi}{\Phi'},\qquad
\text{(iii)}\ \coth\delta=\frac{\Theta'\cosh\Phi}{\Phi'}.
\end{equation}
Since $\Phi'>0$ and $\sin\Phi\neq0$ in case~(i), these relations are equivalent to the integrands in \eqref{Theta Type 2}, \eqref{Theta Type 4}, \eqref{Theta Type 3}, \eqref{Theta Type 1} and \eqref{Theta Type 5}:
\[
\Theta'=\Phi'\csc\Phi\cot\delta,\qquad
\Theta'=\Phi'\,\frac{\tanh\delta}{\cosh\Phi},\qquad
\Theta'=\Phi'\,\frac{\coth\delta}{\cosh\Phi},
\]
and $m$ is the constant of integration.

\emph{Step 2: $\kappa_\beta$ and $\tau_\beta$ in terms of $\kg$.} By \eqref{eq:def_delta},
\begin{align*}
\text{(i)}\quad & \Theta'^2\tan^2\Phi+\Phi'^2\sec^2\Phi=\Phi'^2\sec^2\Phi\csc^2\delta=K^2,\\
\text{(ii)}\quad & \Phi'^2-\Theta'^2\cosh^2\Phi=\Phi'^2\operatorname{sech}^2\delta=K^2,\\
\text{(iii)}\quad & \Theta'^2\cosh^2\Phi-\Phi'^2=\Phi'^2\operatorname{csch}^2\delta=K^2,
\end{align*}
so by \eqref{eq:kappa_x_1}, \eqref{eq:kappa_x_2} and \eqref{eq:kappa_x_3}, $\kappa_\beta=K/D$ in the three cases. The numerators of $\tau_\beta$ in the Frenet apparatus with respect to a generic parameter (equations \eqref{eq:tau_x_1}, \eqref{eq:tau_x_2} and \eqref{eq:tau_x_3}) are
\begin{align*}
\text{(i)}\quad & \Theta'\sec\Phi-\delta'=\Phi'\cot\delta\csc\Phi\sec\Phi-\delta'=K\,\kg,\\
\text{(ii)}\quad & -\big(\Theta'\sinh\Phi+\delta'\big)=-\big(\Phi'\tanh\delta\tanh\Phi+\delta'\big)=K\,\kg,\\
\text{(iii)}\quad & \Theta'\sinh\Phi+\delta'=\Phi'\coth\delta\tanh\Phi+\delta'=K\,\kg,
\end{align*}
with $\kg$ given by \eqref{F Type 2}, \eqref{F Type 4}, \eqref{F Type 3}, \eqref{F Type 1} and \eqref{F Type 5} respectively. (In case~(i), $(\Theta'\tan\Phi,\Phi'\sec\Phi)=K(\cos\delta,\sin\delta)$, so $\delta$ is the polar angle of this vector, and $\tau_\beta D=\Theta'\sec\Phi-\delta'$ holds also where $\Theta'=0$.) Hence
\begin{equation}\label{eq:proof_kappa_tau}
\kappa_\beta=\frac{K}{D},\qquad \tau_\beta=\kg\,\frac{K}{D}=\kg\,\kappa_\beta .
\end{equation}

\emph{Step 3: the Mannheim condition is equivalent to the formula for $ds/dt$.} By Theorem~\ref{theorem caracterization mannheim pair}, $\beta$ is a Mannheim curve if and only if \eqref{Mannheim equation how hypotesis} holds, that is, $\kappa_\beta(1+\lambda\varepsilon_3\kappa_\beta)\varepsilon_1+\lambda\tau_\beta^2=0$. With $(\varepsilon_1,\varepsilon_3)=(-1,1)$, $(1,1)$ and $(1,-1)$ in cases (i), (ii) and (iii), and using $\tau_\beta^2=\kg^2\kappa_\beta^2$, this is
\[
\text{(i)}\ \kappa_\beta=\lambda\kappa_\beta^2\big(\kg^2-1\big),\qquad
\text{(ii)}\ \kappa_\beta=-\lambda\kappa_\beta^2\big(1+\kg^2\big),\qquad
\text{(iii)}\ \kappa_\beta=\lambda\kappa_\beta^2\big(1-\kg^2\big).
\]
Dividing by $\kappa_\beta>0$ and substituting $\kappa_\beta=K/D$:
\begin{align*}
\text{(i)}\quad & D=\lambda K\,(\kg^2-1),\\
\text{(ii)}\quad & D=(-\lambda)K\,(1+\kg^2),\\
\text{(iii)}\quad & D=\lambda K\,(1-\kg^2).
\end{align*}
Writing $K$ explicitly, these are \eqref{dsdx Type 2} and \eqref{dsdx Type 4}, \eqref{dsdx Type 3}, and \eqref{dsdx Type 1} and \eqref{dsdx Type 5}. Since $D>0$ and $K>0$, they force $\lambda(\kg^2-1)>0$ in case~(i), $\lambda<0$ in case~(ii) and $\lambda(1-\kg^2)>0$ in case~(iii). When they hold, \eqref{eq:proof_kappa_tau} gives \eqref{k e tau Type 1}--\eqref{k e tau Type 5}.

\emph{Step 4: the type of the pair.} By Theorem~\ref{theorem caracterization mannheim pair}, $\left(\dfrac{ds}{ds^*}\right)^2=-\dfrac{\varepsilon_{1*}}{\lambda}\dfrac{\kappa_\beta}{\tau_\beta^2}>0$, and $\kappa_\beta>0$, so $\varepsilon_{1*}=-\operatorname{sgn}\lambda$: $T_{\beta^*}$ is timelike if and only if $\lambda>0$. Since $B_{\beta^*}=-N_\beta$ and a timelike vector of $\mathbb{L}^3$ is orthogonal only to spacelike vectors, exactly one of $T_{\beta^*},N_{\beta^*},B_{\beta^*}$ is timelike. Therefore:
\begin{itemize}
\item case~(i), $N_\beta$ spacelike: $\lambda>0$ gives $T_{\beta^*}$ timelike (Type~2), $\lambda<0$ gives $T_{\beta^*}$ spacelike, $B_{\beta^*}$ spacelike and $N_{\beta^*}$ timelike (Type~4);
\item case~(ii), $N_\beta$ timelike: $\lambda<0$, $T_{\beta^*}$ spacelike and $B_{\beta^*}=-N_\beta$ timelike (Type~3);
\item case~(iii), $N_\beta$ spacelike: $\lambda>0$ gives $T_{\beta^*}$ timelike (Type~1), $\lambda<0$ gives $N_{\beta^*}$ timelike (Type~5).
\end{itemize}

\emph{Direct part.} Given $\delta$ and $\Phi$ as in the statement, define $\Theta$ and $D$ by the formulas of Type~$k$ and $\beta$ by \eqref{d alpha s}. Since $\mathbf{u}$ is a unit vector, $|\beta'|=D$, so $s$ is the arc length of $\beta$ and $T_\beta=\mathbf{u}$, with angular coordinates $(\Theta,\Phi)$. By Step~1, \eqref{eq:def_delta} holds, so Step~2 applies. By Step~3, $\beta$ satisfies the Mannheim equation, and $\kg\neq0$ gives $\tau_\beta=\kg\,\kappa_\beta\neq0$; by Theorem~\ref{theorem caracterization mannheim pair}, $\beta$ is a Mannheim curve with partner $\beta^*=\beta+\lambda N_\beta$. Its causal character follows from $T_\beta'=\kappa_\beta D\,N_\beta$ and $\langle \mathbf{u}',\mathbf{u}'\rangle=\Theta'^2\tan^2\Phi+\Phi'^2\sec^2\Phi=K^2>0$ in case~(i), $\Theta'^2\cosh^2\Phi-\Phi'^2=-K^2<0$ in case~(ii) and $\Theta'^2\cosh^2\Phi-\Phi'^2=K^2>0$ in case~(iii): $T_\beta$ is timelike and $N_\beta$ spacelike in case~(i), $T_\beta$ spacelike and $N_\beta$ timelike in case~(ii), and $T_\beta,N_\beta$ spacelike (so $B_\beta$ timelike) in case~(iii). Step~4 gives the type.

\emph{Converse.} Let $\beta$ be a Mannheim curve of Type~$k$, with $D>0$ and $T_\beta=\mathbf{u}$ as in the statement. Since $\beta$ is a Mannheim curve, $\tau_\beta\neq0$ and \eqref{Mannheim equation how hypotesis} holds. The causal character of $N_\beta$ gives $\Phi'^2-\Theta'^2\cosh^2\Phi>0$ in case~(ii) and $\Theta'^2\cosh^2\Phi-\Phi'^2>0$ in case~(iii), so $\delta$ is well defined by \eqref{eq:def_delta}: by $\delta=\operatorname{arccot}\big(\Theta'\sin\Phi/\Phi'\big)\in(0,\pi)$ in case~(i), by $\delta=\operatorname{arctanh}\big(\Theta'\cosh\Phi/\Phi'\big)$ in case~(ii), and by $\delta=\operatorname{arccoth}\big(\Theta'\cosh\Phi/\Phi'\big)>0$ in case~(iii) (here $\Theta'>0$). Now Steps~1 to~3 apply: \eqref{Mannheim equation how hypotesis} forces $D$ to be given by the formula of Type~$k$, $\Theta$ is given by the corresponding integral, $\kappa_\beta$ and $\tau_\beta$ are those listed, and $\beta'=D\,\mathbf{u}$.
\end{proof}

 \subsection{Examples}

Types~2 and~4 have $T_\beta$ timelike, Type~3 has $N_\beta$ timelike and Types~1 and~5 have $B_\beta$ timelike. Each example lists the basic functions $\Phi,\Theta,\delta$, the auxiliary functions, the curve $\beta$ and its Mannheim partner $\beta^{*}=\beta+\lambda N_\beta$, together with their Frenet data. Types~4 and~5 use the same functions as Types~2 and~1, with the opposite sign of $\lambda$, so only the data that change are listed. As noted in Remark~\ref{rem:normalization}, the pairs of Types~1 and~4 below are given in the orientation class with $\tau_\beta>0$.

\begin{example}[Type 1]

\textit{Basic functions} ($\lambda=1$, $t\in(0,\operatorname{arccosh}2)$):
\begin{equation}
\begin{aligned}
\Phi(t)&=t,\\
\Theta(t)&=4\arctan(e^t),\\
\delta(t)&=\operatorname{arccoth}(2)=\dfrac12\ln 3,
\end{aligned}
\end{equation}
that is, $\sinh\delta=\dfrac{1}{\sqrt3}$, $\cosh\delta=\dfrac{2}{\sqrt3}$ and $\operatorname{csch}\delta=\sqrt3$.

\textit{Auxiliary functions:}
\begin{equation}
\kg(t)=\dfrac{2}{\sqrt3}\tanh t,\qquad
\dfrac{ds}{dt}=\dfrac{4\operatorname{sech}^2 t-1}{\sqrt3}\quad \text{or}\quad s(t)=\dfrac{1}{\sqrt3}\big(4\tanh t-t\big).
\end{equation}

\textit{Curvature and torsion:}
\begin{equation}\label{curvature and torsion type 1}
\kappa_{\beta}(t)=\dfrac{3}{4\operatorname{sech}^2 t-1},\qquad
\tau_{\beta}(t)=\dfrac{2\sqrt3\,\tanh t}{4\operatorname{sech}^2 t-1}.
\end{equation}

\textit{Tangent vector:}
\begin{equation}
\dfrac{d\beta}{dt}(t) = \dfrac{4\operatorname{sech}^2 t-1}{\sqrt3}\left((1-2\operatorname{sech}^{2}t)\cosh t,\ -2\tanh t,\ \sinh t\right)
\end{equation}

\textit{Coordinates of the curve (up to translation):}
\begin{equation}\label{eq:beta_explicit 3}
\beta(t) = \dfrac{1}{\sqrt3}\left(
 4\arctan(e^t)-\sinh t-4\operatorname{sech}t\tanh t,\
 2\ln(\cosh t)+4\operatorname{sech}^2 t,\
 -\cosh t-4\operatorname{sech}t
\right).
\end{equation}

\textit{Frenet frame:}
\begin{equation}
\begin{aligned}
T_{\beta}(t)&=
\left( \cosh t-2\operatorname{sech}t,\; -2\tanh t,\; \sinh t \right),\\[6pt]
N_{\beta}(t)&=
\dfrac{1}{\sqrt3}\left( 4\operatorname{sech}t\tanh t+(1-2\operatorname{sech}^2 t)\sinh t,\; -2\operatorname{sech}^2 t,\; \cosh t \right),\\[6pt]
B_{\beta}(t)&=
\dfrac{1}{\sqrt3}\left( -2\operatorname{sech}t\tanh t-2(1-2\operatorname{sech}^2 t)\sinh t,\; 3-2\operatorname{sech}^2 t,\; -2\cosh t \right).
\end{aligned}
\end{equation}

\textit{Mannheim partner (up to translation):}
\begin{equation}
\beta^{*}(t)=\dfrac{1}{\sqrt3}\left(
4\arctan(e^t)-2\operatorname{sech}t\tanh t,\
2\ln(\cosh t)+2\operatorname{sech}^2t,\
-4\operatorname{sech}t
\right)
\end{equation}

\textit{Frenet frame of $\beta^{*}$ ($T_{\beta^{*}}$ timelike):}
\begin{equation}
\begin{aligned}
T_{\beta^{*}}(t)&=\dfrac{1}{\sqrt{4\operatorname{sech}^2t-1}}\left( 2\operatorname{sech}t\tanh t,\; 1-2\operatorname{sech}^2t,\; 2\operatorname{sech}t \right),\\[6pt]
N_{\beta^{*}}(t)&=\dfrac{1}{\sqrt3\sqrt{4\operatorname{sech}^2t-1}}\left( \cosh t-2\operatorname{sech}t+4\operatorname{sech}^3t,\; 4\operatorname{sech}^2t\tanh t,\; \sinh t \right),\\[6pt]
B_{\beta^{*}}(t)&=-N_{\beta}(t).
\end{aligned}
\end{equation}

\textit{Curvature, torsion and angle of the pair:}
\begin{equation}
\dfrac{ds^{*}}{dt}=\dfrac{2}{\sqrt3}\tanh t\sqrt{4\operatorname{sech}^2t-1},\quad
\kappa_{\beta^{*}}=\dfrac{3\operatorname{sech}t}{\sinh t\,(4\operatorname{sech}^2t-1)^{3/2}},\quad
\tau_{\beta^{*}}=\dfrac{\sqrt3}{2}\coth t,
\end{equation}
\begin{equation}
\sinh\varphi_1=|\langle T_\beta,T_{\beta^{*}}\rangle|=\dfrac{2\tanh t}{\sqrt{4\operatorname{sech}^2t-1}},\qquad
\cosh\varphi_1=\dfrac{\sqrt3}{\sqrt{4\operatorname{sech}^2t-1}},\qquad \tanh\varphi_1=\kg(t).
\end{equation}
\textit{Non-helix and non-slant nature:}

Using the explicit expressions for $\kappa_\beta$ and $\tau_\beta$ from \eqref{curvature and torsion type 1}, we compute
\begin{equation}
\kappa^{\smash[b]{T_{\beta}}}_{g}(t)=\frac{{\tau_\beta}(t)}{{\kappa_\beta}(t)} = \frac{2}{\sqrt3}\tanh(t),
\end{equation}
which is clearly non-constant. By Lancret's theorem in $\mathbb{L}^3$ \cite[Theorem~2.12]{Lopez2014}, $\beta$ is not a helix.

There do not exist $\textbf{v}=(\bm{v}_{1},\bm{v}_{2},\bm{v}_{3})\in\mathbb{R}^3\setminus\{0\}$ and $\bm{d}\in\mathbb{R}$ such that $\langle N_{\beta}(t),\textbf{v}\rangle = \bm{d}$ for all $t$.

In fact, on the contrary,
\[
\langle N_{\beta}(t),\textbf{v}\rangle=\bm{d} \Longleftrightarrow
\bm{v}_{1}\cdot \big(4\operatorname{sech}(t)\tanh(t)+(1-2\operatorname{sech}^2(t))\sinh(t)\big) - \bm{v}_{2}\cdot2\operatorname{sech}^2(t) - \bm{v}_{3}\cdot \cosh(t) = \bm{d}\cdot\sqrt{3}.
\]
Both sides are real analytic, so this identity holds for all $t\in\mathbb{R}$. Since $4\operatorname{sech}t\tanh t+(1-2\operatorname{sech}^2t)\sinh t=\sinh t+2\operatorname{sech}t\tanh t$, letting $t\to+\infty$ the terms in $e^{t}$ force $\bm v_3=\bm v_1$; then the left-hand side tends to $0$, so $\bm d=0$; the coefficient of $e^{-t}$ is then $3\bm v_1$, so $\bm v_1=\bm v_3=0$, and finally $\bm v_2=0$. This contradicts $\textbf{v}\neq0$. Hence, by definition, $\beta$ is not a slant helix.

\end{example}

\begin{example}[Type 2]

\textit{Basic functions} ($\lambda>0$, $t\in(0,1)$):
\begin{equation}
\begin{aligned}
\Phi(t)&=\arccos\!\Big(\dfrac{18(1+t^2)}{t^6+3t^4+27t^2+18}\Big),\\
\Theta(t)&=\arctan\!\Big(\dfrac{t\,(t^4+3t^2+27)}{6\,(3-t^2)}\Big),\\
\delta(t)&=\operatorname{arccot}\!\Big(\dfrac{t\,(27+18t^2+4t^4-t^6)}{2\,(2t^6+6t^4+6t^2+9)}\Big).
\end{aligned}
\end{equation}

\textit{Auxiliary functions:}
\begin{equation}
\kg(t)=\dfrac{3-t^2}{1+t^2},\qquad
\dfrac{ds}{dt}= \dfrac{8\lambda(1-t^2)}{(1+t^2)^2}\quad \text{or}\quad s(t)=\dfrac{8\lambda t}{1+t^2}\quad \text{or}\quad t(s)
    =\tan\!\Big(\tfrac12\arcsin\dfrac{s}{4\lambda}\Big).
\end{equation}

\textit{Curvature and torsion:}
\begin{equation}\label{curvature and torsion type 2}
\kappa_{\beta}(t)=\dfrac{(1+t^2)^2}{8\lambda(1-t^2)},\qquad
\tau_{\beta}(t) = \dfrac{(3-t^2)(1+t^2)}{8\lambda(1-t^2)}.
\end{equation}

\textit{Tangent vector:}
\[
\dfrac{d\beta}{dt}(t) = \dfrac{8\lambda(1-t^2)}{(1+t^2)^2}
\begin{bmatrix}
\dfrac{t(3-t^2)}{3(1+t^2)} \\[12pt]
\dfrac{t^2(t^4+3t^2+27)}{18(1+t^2)} \\[12pt]
\dfrac{t^6+3t^4+27t^2+18}{18(1+t^2)}
\end{bmatrix}.
\]

\textit{Coordinates of the curve (up to translation):}
\begin{equation}\label{eq:beta_explicit 1}
\beta(t) = \lambda \begin{bmatrix}
 \dfrac{4}{3}\Big[\ln(1+t^2)+\dfrac{6t^2+2}{(1+t^2)^2}\Big] \\[14pt]
 \dfrac{-4t^7+4t^5+233t^3+75t}{27(1+t^2)^2}-\dfrac{25}{9}\arctan t \\[14pt]
 \dfrac{-4t^7+4t^5+287t^3+237t}{27(1+t^2)^2}-\dfrac{7}{9}\arctan t
\end{bmatrix}.
\end{equation}

\textit{Frenet frame (tangent $T_\beta$ timelike):}
\begin{equation}
\begin{aligned}
T_{\beta}(t)&=
\left( \dfrac{t(3-t^2)}{3(1+t^2)},\; \dfrac{t^2(t^4+3t^2+27)}{18(1+t^2)},\; \dfrac{t^6+3t^4+27t^2+18}{18(1+t^2)} \right),\\[6pt]
N_{\beta}(t)&=
\left( \dfrac{3-6t^2-t^4}{3(1+t^2)^2},\; \dfrac{t(2t^6+6t^4+6t^2+27)}{9(1+t^2)^2},\; \dfrac{t(2t^6+6t^4+6t^2+9)}{9(1+t^2)^2} \right),\\[6pt]
B_{\beta}(t)&=
\left( -\dfrac{t(t^4+2t^2+9)}{3(1+t^2)^2},\; \dfrac{t^8-4t^6-18t^4-45t^2+18}{18(1+t^2)^2},\; \dfrac{t^2(t^6-4t^4-18t^2-27)}{18(1+t^2)^2} \right).
\end{aligned}
\end{equation}

\textit{Mannheim pair $\beta^{*}(t)=\beta(t)+\lambda N_{\beta}(t)$ (up to translation):}
\begin{equation}\label{eq:beta_star_explicit 1}
\beta^{*}(t)=\lambda\begin{bmatrix}
\dfrac43\ln(1+t^2)+\dfrac{-t^4+18t^2+11}{3(1+t^2)^2}\\[14pt]
\dfrac{2t^7+22t^5+251t^3+156t}{27(1+t^2)^2}-\dfrac{25}{9}\arctan t\\[14pt]
\dfrac{2t^7+22t^5+305t^3+264t}{27(1+t^2)^2}-\dfrac79\arctan t
\end{bmatrix}.
\end{equation}

\textit{Frenet frame of $\beta^{*}$ (tangent $T_{\beta^{*}}$ timelike):}
\begin{equation}
\begin{aligned}
T_{\beta^{*}}(t)&=\dfrac{\big(-24t^3,\; -2t^6-18t^4+18t^2+9,\; -2t^6-18t^4+18t^2+27\big)}{9(1+t^2)\sqrt{8(1-t^2)}},\\[6pt]
N_{\beta^{*}}(t)&=\dfrac{\big(24t(t^2-3),\; 6t^8+20t^6+24t^4-63t^2+27,\; 6t^8+20t^6+24t^4-9t^2+9\big)}{9(1+t^2)^2\sqrt{8(1-t^2)}},\\[6pt]
B_{\beta^{*}}(t)&=-N_{\beta}(t).
\end{aligned}
\end{equation}

\textit{Curvature, torsion and angle of the pair:}
\begin{equation}
\dfrac{ds^{*}}{dt}=\dfrac{\lambda(3-t^2)\sqrt{8(1-t^2)}}{(1+t^2)^2},\quad
\kappa_{\beta^{*}}=\dfrac{t(1+t^2)^2}{\lambda(1-t^2)(3-t^2)\sqrt{8(1-t^2)}},\quad
\tau_{\beta^{*}}=\dfrac{1+t^2}{\lambda(3-t^2)},
\end{equation}
\begin{equation}
\cosh\varphi_2=-\langle T_\beta,T_{\beta^{*}}\rangle=\dfrac{3-t^2}{\sqrt{8(1-t^2)}},\qquad
\sinh\varphi_2=\dfrac{1+t^2}{\sqrt{8(1-t^2)}},\qquad \tanh\varphi_2=\dfrac{1}{\kg(t)}.
\end{equation}

\textit{Non-helix and non-slant nature:}

Using the explicit expressions for $\kappa_\beta$ and $\tau_\beta$ from \eqref{curvature and torsion type 2}, we compute
\begin{equation}
\kappa^{\smash[b]{T_{\beta}}}_{g}(t)=\frac{{\tau_\beta}(t)}{{\kappa_\beta}(t)} = \frac{3-t^2}{1+t^2},
\end{equation}
which is clearly non-constant. By Lancret's theorem in $\mathbb{L}^3$ \cite[Theorem~2.12]{Lopez2014}, $\beta$ is not a helix.

There do not exist $\textbf{v}=(\bm{v}_{1},\bm{v}_{2},\bm{v}_{3})\in\mathbb{R}^3\setminus\{0\}$ and $\bm{d}\in\mathbb{R}$ such that $\langle N_{\beta}(t),\textbf{v}\rangle = \bm{d}$ for all $t$.

In fact, on the contrary,
\[
\langle N_{\beta}(t),\textbf{v}\rangle=\bm{d} \Longleftrightarrow
3\bm{v}_{1}\cdot (3-6t^2-t^4) + \bm{v}_{2}\cdot t\,(2t^6+6t^4+6t^2+27) - \bm{v}_{3}\cdot t\,(2t^6+6t^4+6t^2+9) = \bm{d}\cdot9(1+t^2)^2.
\]
Both sides are polynomials in $t$ that agree on an interval, hence they coincide as polynomials. The coefficients of $t^7$ and $t$ give $\bm v_2=\bm v_3$ and $27\bm v_2=9\bm v_3$, so $\bm v_2=\bm v_3=0$; the coefficients of $t^4$ and $t^2$ give $-3\bm v_1=9\bm d$ and $-18\bm v_1=18\bm d$, so $\bm v_1=\bm d=0$. This contradicts $\textbf{v}\neq0$. Hence, by definition, $\beta$ is not a slant helix.
\end{example}

\begin{example}[Type 3]

\textit{Basic functions} ($\lambda=-1$, $t\in(0,+\infty)$):
\begin{equation}
\begin{aligned}
\Phi(t)&=t,\\
\Theta(t)&=\arctan(e^t),\\
\delta(t)&=\operatorname{arctanh}\!\Big(\dfrac12\Big)=\dfrac12\ln 3,
\end{aligned}
\end{equation}
that is, $\sinh\delta=\dfrac{1}{\sqrt3}$, $\cosh\delta=\dfrac{2}{\sqrt3}$ and $\operatorname{sech}\delta=\dfrac{\sqrt3}{2}$.

\textit{Auxiliary functions:}
\begin{equation}
\kg(t)=-\dfrac{1}{\sqrt3}\tanh t,\qquad
\dfrac{ds}{dt}=\dfrac{\sqrt3}{2}\Big(1+\dfrac13\tanh^2 t\Big)\quad \text{or}\quad s(t)=\dfrac{\sqrt3}{6}\big(4t-\tanh t\big).
\end{equation}

\textit{Curvature and torsion:}
\begin{equation}\label{curvature and torsion type 3}
\kappa_{\beta}(t)=\dfrac{1}{1+\frac13\tanh^2 t},\qquad
\tau_{\beta}(t)=-\dfrac{\frac{1}{\sqrt3}\tanh t}{1+\frac13\tanh^2 t}.
\end{equation}

\textit{Tangent vector:}
\begin{equation}
\dfrac{d\beta}{dt}(t) = \dfrac{\sqrt3}{2}\Big(1+\dfrac13\tanh^2 t\Big)\left(\dfrac{\cosh t}{\sqrt{1+e^{2t}}} ,\dfrac{e^t\cosh t}{\sqrt{1+e^{2t}}} ,\sinh t\right)
\end{equation}

\textit{Coordinates of the curve (up to translation):}
\begin{equation}\label{eq:beta_explicit 2}
\beta(t) = \dfrac{\sqrt3}{3}\begin{pmatrix}
\operatorname{arsinh}(e^t)-\dfrac{1+2e^{2t}}{e^t\sqrt{1+e^{2t}}}\\[8pt]
\sqrt{1+e^{2t}}+\dfrac{1}{\sqrt{1+e^{2t}}}-\ln\!\Big(\dfrac{e^t}{\sqrt{1+e^{2t}}-1}\Big)\\[8pt]
2\cosh t+\dfrac12\operatorname{sech}t
\end{pmatrix}.
\end{equation}

\textit{Frenet frame:}
\begin{equation}
\begin{aligned}
T_{\beta}(t)&=
\left( \dfrac{\cosh t}{\sqrt{1+e^{2t}}},\; \dfrac{e^t\cosh t}{\sqrt{1+e^{2t}}},\; \sinh t \right),\\[6pt]
N_{\beta}(t)&=
\dfrac{1}{\sqrt3}\left( -\dfrac{e^{-t}}{\sqrt{1+e^{2t}}},\; \dfrac{e^{2t}}{\sqrt{1+e^{2t}}},\; 2\cosh t \right),\\[6pt]
B_{\beta}(t)&=
\dfrac{1}{\sqrt3}\left( \dfrac{3e^t+e^{-t}}{2\sqrt{1+e^{2t}}},\; -\dfrac{e^{2t}+3}{2\sqrt{1+e^{2t}}},\; -\cosh t \right).
\end{aligned}
\end{equation}

\textit{Mannheim partner (up to translation):}
\begin{equation}
\beta^{*}(t)=\dfrac{\sqrt3}{3}\left(
\operatorname{arsinh}(e^t)-\dfrac{2e^t}{\sqrt{1+e^{2t}}},\
\dfrac{2}{\sqrt{1+e^{2t}}}-\ln\!\Big(\dfrac{e^t}{\sqrt{1+e^{2t}}-1}\Big),\
\dfrac12\operatorname{sech}t
\right)
\end{equation}

\textit{Frenet frame of $\beta^{*}$ ($B_{\beta^{*}}$ timelike):}
\begin{equation}
\begin{aligned}
T_{\beta^{*}}(t)&=\dfrac{1}{\sqrt{3+\tanh^2t}}\left( \dfrac{2e^t}{\sqrt{1+e^{2t}}},\; -\dfrac{2}{\sqrt{1+e^{2t}}},\; -\operatorname{sech}t \right),\\[6pt]
N_{\beta^{*}}(t)&=\dfrac{2}{\sqrt3\sqrt{3+\tanh^2t}}\left( \dfrac{2e^{2t}+1}{e^t(1+e^{2t})^{3/2}},\; \dfrac{e^{2t}(e^{2t}+2)}{(1+e^{2t})^{3/2}},\; 2\sinh t \right),\\[6pt]
B_{\beta^{*}}(t)&=-N_{\beta}(t).
\end{aligned}
\end{equation}

\textit{Curvature, torsion and angle of the pair:}
\begin{equation}
\dfrac{ds^{*}}{dt}=\dfrac12\tanh t\sqrt{1+\dfrac13\tanh^2t},\quad
\kappa_{\beta^{*}}=\dfrac{6\operatorname{sech}t}{\sinh t\,(3+\tanh^2t)^{3/2}},\quad
\tau_{\beta^{*}}=+\sqrt3\coth t,
\end{equation}
\begin{equation}
\cos\varphi_3=\langle T_\beta,T_{\beta^{*}}\rangle=\dfrac{\tanh t}{\sqrt{3+\tanh^2t}},\qquad
\sin\varphi_3=\dfrac{\sqrt3}{\sqrt{3+\tanh^2t}},\qquad \tan\varphi_3=-\dfrac{1}{\kg(t)}.
\end{equation}

\textit{Non-helix and non-slant nature:}

Using the explicit expressions for $\kappa_\beta$ and $\tau_\beta$ from \eqref{curvature and torsion type 3}, we compute
\begin{equation}
\kappa^{\smash[b]{T_{\beta}}}_{g}(t)=\frac{{\tau_\beta}(t)}{{\kappa_\beta}(t)} = -\frac{1}{\sqrt3}\tanh(t),
\end{equation}
which is clearly non-constant. By Lancret's theorem in $\mathbb{L}^3$ \cite[Theorem~2.12]{Lopez2014}, $\beta$ is not a helix.

There do not exist $\textbf{v}=(\bm{v}_{1},\bm{v}_{2},\bm{v}_{3})\in\mathbb{R}^3\setminus\{0\}$ and $\bm{d}\in\mathbb{R}$ such that $\langle N_{\beta}(t),\textbf{v}\rangle = \bm{d}$ for all $t$.

In fact, on the contrary,
\[
\langle N_{\beta}(t),\textbf{v}\rangle=\bm{d} \Longleftrightarrow
-\bm{v}_{1}\cdot e^{-t} + \bm{v}_{2}\cdot e^{2t} - \bm{v}_{3}\cdot2\cosh(t)\sqrt{1+e^{2t}} = \bm{d}\cdot\sqrt{3}\sqrt{1+e^{2t}}.
\]
Both sides are real analytic, so this identity holds for all $t\in\mathbb{R}$. Using $\sqrt{1+e^{2t}}=e^{t}+\tfrac12e^{-t}+O(e^{-3t})$ we get $2\cosh t\sqrt{1+e^{2t}}=e^{2t}+\tfrac32+O(e^{-2t})$. Letting $t\to+\infty$, the coefficient of $e^{2t}$ gives $\bm v_2=\bm v_3$, that of $e^{t}$ gives $\bm d=0$, and the constant term gives $-\tfrac32\bm v_3=0$; hence $\bm v_3=\bm v_2=0$ and then $\bm v_1=0$. This contradicts $\textbf{v}\neq0$. Hence, by definition, $\beta$ is not a slant helix.
\end{example}

\begin{example}[Type 4]

Same functions $\Phi$, $\Theta$, $\delta$ as in Type~2, now with $\lambda<0$ and $t\in(1,\sqrt3)$. The formulas for $\dfrac{ds}{dt}$ and $s(t)$, for $\kappa_\beta$ and $\tau_\beta$, for $\dfrac{d\beta}{dt}$, for $\beta$, for the frame $(T_\beta,N_\beta,B_\beta)$ and for $\beta^{*}=\beta+\lambda N_\beta$ are those of Type~2 (now $\dfrac{ds}{dt}>0$ because $\lambda<0$ and $1-t^2<0$). What changes is the following.

\textit{Auxiliary functions:}
\begin{equation}
\kg(t)=\dfrac{3-t^2}{1+t^2}\in(0,1),\qquad
t(s)=\cot\!\Big(\tfrac12\arcsin\dfrac{s}{4\lambda}\Big).
\end{equation}

\textit{Frenet frame of $\beta^{*}$ ($T_{\beta^{*}}$ spacelike):}
\begin{equation}
\begin{aligned}
T_{\beta^{*}}(t)&=\dfrac{\big(24t^3,\; 2t^6+18t^4-18t^2-9,\; 2t^6+18t^4-18t^2-27\big)}{9(1+t^2)\sqrt{8(t^2-1)}},\\[6pt]
N_{\beta^{*}}(t)&=\dfrac{\big(24t(t^2-3),\; 6t^8+20t^6+24t^4-63t^2+27,\; 6t^8+20t^6+24t^4-9t^2+9\big)}{9(1+t^2)^2\sqrt{8(t^2-1)}},\\[6pt]
B_{\beta^{*}}(t)&=-N_{\beta}(t).
\end{aligned}
\end{equation}

\textit{Curvature, torsion and angle of the pair:}
\begin{equation}
\dfrac{ds^{*}}{dt}=\dfrac{(-\lambda)(3-t^2)\sqrt{8(t^2-1)}}{(1+t^2)^2},\quad
\kappa_{\beta^{*}}=\dfrac{t(1+t^2)^2}{(-\lambda)(t^2-1)(3-t^2)\sqrt{8(t^2-1)}},\quad
\tau_{\beta^{*}}=\dfrac{1+t^2}{\lambda(3-t^2)},
\end{equation}
\begin{equation}
\sinh\varphi_4=\langle T_\beta,T_{\beta^{*}}\rangle=\dfrac{3-t^2}{\sqrt{8(t^2-1)}},\qquad
\cosh\varphi_4=\dfrac{1+t^2}{\sqrt{8(t^2-1)}},\qquad \tanh\varphi_4=\kg(t).
\end{equation}

\textit{Non-helix and non-slant nature:}

Using the explicit expressions for $\kappa_\beta$ and $\tau_\beta$ from \eqref{curvature and torsion type 2}, we compute
\begin{equation}
\kappa^{\smash[b]{T_{\beta}}}_{g}(t)=\frac{{\tau_\beta}(t)}{{\kappa_\beta}(t)} = \frac{3-t^2}{1+t^2},
\end{equation}
which is clearly non-constant. By Lancret's theorem in $\mathbb{L}^3$ \cite[Theorem~2.12]{Lopez2014}, $\beta$ is not a helix.

There do not exist $\textbf{v}=(\bm{v}_{1},\bm{v}_{2},\bm{v}_{3})\in\mathbb{R}^3\setminus\{0\}$ and $\bm{d}\in\mathbb{R}$ such that $\langle N_{\beta}(t),\textbf{v}\rangle = \bm{d}$ for all $t$.

In fact, on the contrary,
\[
\langle N_{\beta}(t),\textbf{v}\rangle=\bm{d} \Longleftrightarrow
3\bm{v}_{1}\cdot (3-6t^2-t^4) + \bm{v}_{2}\cdot t\,(2t^6+6t^4+6t^2+27) - \bm{v}_{3}\cdot t\,(2t^6+6t^4+6t^2+9) = \bm{d}\cdot9(1+t^2)^2.
\]
Both sides are polynomials in $t$ that agree on an interval, hence they coincide as polynomials. The coefficients of $t^7$ and $t$ give $\bm v_2=\bm v_3$ and $27\bm v_2=9\bm v_3$, so $\bm v_2=\bm v_3=0$; the coefficients of $t^4$ and $t^2$ give $-3\bm v_1=9\bm d$ and $-18\bm v_1=18\bm d$, so $\bm v_1=\bm d=0$. This contradicts $\textbf{v}\neq0$. Hence, by definition, $\beta$ is not a slant helix.
\end{example}

\begin{example}[Type 5]

Same functions $\Phi$, $\Theta$, $\delta$ and same frame $(T_\beta,N_\beta,B_\beta)$ as in Type~1, now with $\lambda=-1$ and $t\in(\operatorname{arccosh}2,+\infty)$. The coordinates of $\beta$ and of $\beta^{*}$ are those of Type~1 with an overall minus sign. What changes is the following.

\textit{Auxiliary functions:}
\begin{equation}
\kg(t)=\dfrac{2}{\sqrt3}\tanh t>1,\qquad
\dfrac{ds}{dt}=\dfrac{1-4\operatorname{sech}^2 t}{\sqrt3}\quad \text{or}\quad s(t)=\dfrac{1}{\sqrt3}\big(t-4\tanh t\big).
\end{equation}

\textit{Curvature and torsion:}
\begin{equation}\label{curvature and torsion type 5}
\kappa_{\beta}(t)=\dfrac{3}{1-4\operatorname{sech}^2 t},\qquad
\tau_{\beta}(t)=\dfrac{2\sqrt3\,\tanh t}{1-4\operatorname{sech}^2 t}.
\end{equation}

\textit{Tangent vector:}
\begin{equation}
\dfrac{d\beta}{dt}(t) = \dfrac{1-4\operatorname{sech}^2 t}{\sqrt3}\left((1-2\operatorname{sech}^{2}t)\cosh t,\ -2\tanh t,\ \sinh t\right)
\end{equation}

\textit{Frenet frame of $\beta^{*}$ ($N_{\beta^{*}}$ timelike):}
\begin{equation}
\begin{aligned}
T_{\beta^{*}}(t)&=\dfrac{1}{\sqrt{1-4\operatorname{sech}^2t}}\left( -2\operatorname{sech}t\tanh t,\; 2\operatorname{sech}^2t-1,\; -2\operatorname{sech}t \right),\\[6pt]
N_{\beta^{*}}(t)&=\dfrac{1}{\sqrt3\sqrt{1-4\operatorname{sech}^2t}}\left( \cosh t-2\operatorname{sech}t+4\operatorname{sech}^3t,\; 4\operatorname{sech}^2t\tanh t,\; \sinh t \right),\\[6pt]
B_{\beta^{*}}(t)&=-N_{\beta}(t).
\end{aligned}
\end{equation}

\textit{Curvature, torsion and angle of the pair:}
\begin{equation}
\dfrac{ds^{*}}{dt}=\dfrac{2}{\sqrt3}\tanh t\sqrt{1-4\operatorname{sech}^2t},\quad
\kappa_{\beta^{*}}=\dfrac{3\operatorname{sech}t}{\sinh t\,(1-4\operatorname{sech}^2t)^{3/2}},\quad
\tau_{\beta^{*}}=-\dfrac{\sqrt3}{2}\coth t,
\end{equation}
\begin{equation}
\cosh\varphi_5=\langle T_\beta,T_{\beta^{*}}\rangle=\dfrac{2\tanh t}{\sqrt{1-4\operatorname{sech}^2t}},\qquad
\sinh\varphi_5=\dfrac{\sqrt3}{\sqrt{1-4\operatorname{sech}^2t}},\qquad \tanh\varphi_5=\dfrac{1}{\kg(t)}.
\end{equation}

\textit{Non-helix and non-slant nature:}

Using the explicit expressions for $\kappa_\beta$ and $\tau_\beta$ from \eqref{curvature and torsion type 5}, we compute
\begin{equation}
\kappa^{\smash[b]{T_{\beta}}}_{g}(t)=\frac{{\tau_\beta}(t)}{{\kappa_\beta}(t)} = \frac{2}{\sqrt3}\tanh(t),
\end{equation}
which is clearly non-constant. By Lancret's theorem in $\mathbb{L}^3$ \cite[Theorem~2.12]{Lopez2014}, $\beta$ is not a helix.

There do not exist $\textbf{v}=(\bm{v}_{1},\bm{v}_{2},\bm{v}_{3})\in\mathbb{R}^3\setminus\{0\}$ and $\bm{d}\in\mathbb{R}$ such that $\langle N_{\beta}(t),\textbf{v}\rangle = \bm{d}$ for all $t$.

In fact, on the contrary,
\[
\langle N_{\beta}(t),\textbf{v}\rangle=\bm{d} \Longleftrightarrow
\bm{v}_{1}\cdot \big(4\operatorname{sech}(t)\tanh(t)+(1-2\operatorname{sech}^2(t))\sinh(t)\big) - \bm{v}_{2}\cdot2\operatorname{sech}^2(t) - \bm{v}_{3}\cdot \cosh(t) = \bm{d}\cdot\sqrt{3}.
\]
Both sides are real analytic, so this identity holds for all $t\in\mathbb{R}$. Since $4\operatorname{sech}t\tanh t+(1-2\operatorname{sech}^2t)\sinh t=\sinh t+2\operatorname{sech}t\tanh t$, letting $t\to+\infty$ the terms in $e^{t}$ force $\bm v_3=\bm v_1$; then the left-hand side tends to $0$, so $\bm d=0$; the coefficient of $e^{-t}$ is then $3\bm v_1$, so $\bm v_1=\bm v_3=0$, and finally $\bm v_2=0$. This contradicts $\textbf{v}\neq0$. Hence, by definition, $\beta$ is not a slant helix.
\end{example}

\subsection{Illustrations}\label{subsec:illustrations}

\textit{Illustration:}
\begin{figure}[htbp]
\centering
\includegraphics[width=1.0\textwidth]{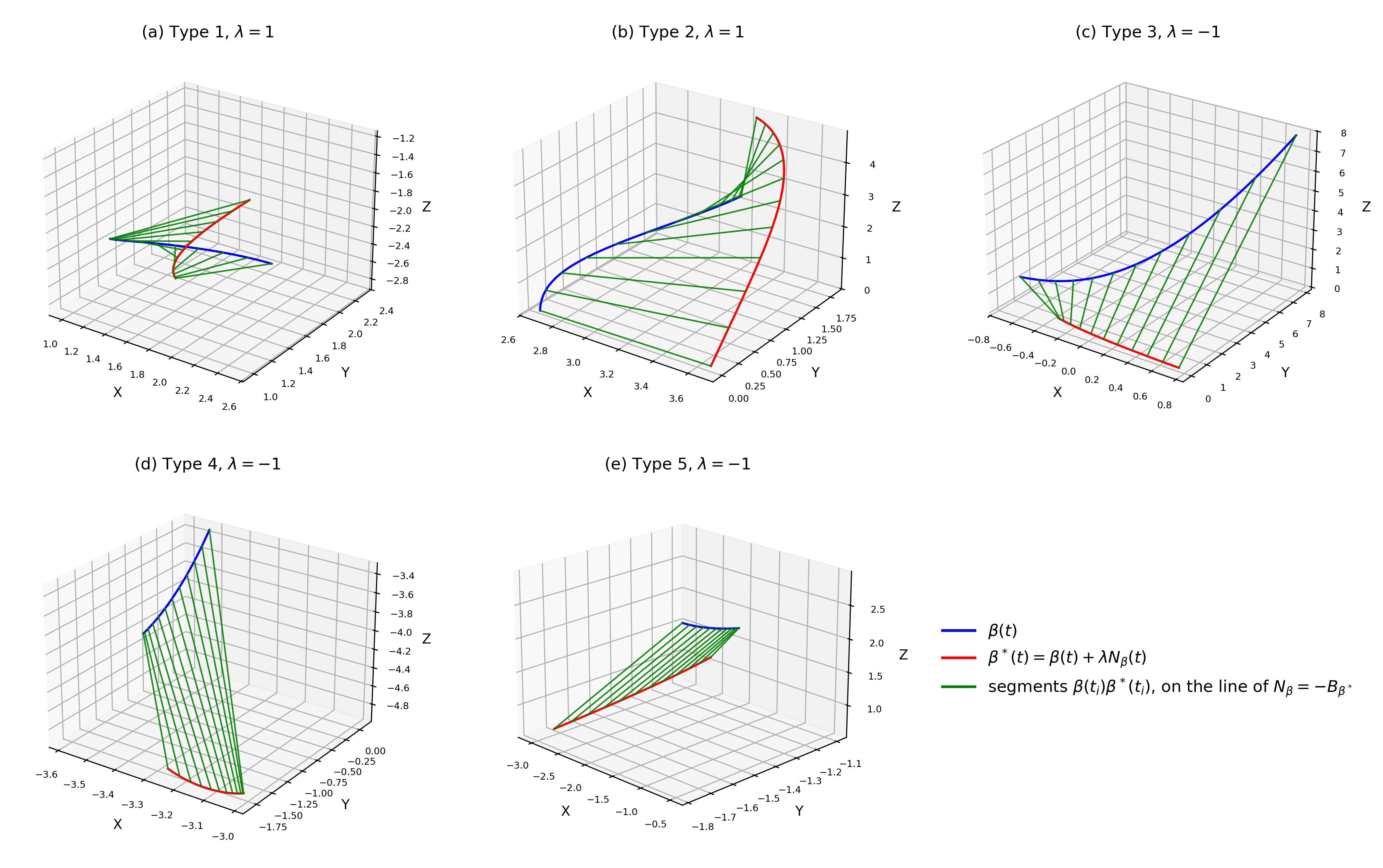}
\caption{Mannheim pairs of Types~1--5 in $\mathbb{L}^3$: the curve $\beta$ (blue), its partner $\beta^*=\beta+\lambda N_\beta$ (red) and the segments joining $\beta(t_i)$ to $\beta^*(t_i)$ (green), which lie on the lines of $N_\beta=-B_{\beta^*}$. (a)~Type~1, $\lambda=1$; (b)~Type~2, $\lambda=1$; (c)~Type~3, $\lambda=-1$; (d)~Type~4, $\lambda=-1$; (e)~Type~5, $\lambda=-1$.}
\label{fig:L3pairs}
\end{figure}

Figure~\ref{fig:L3pairs} shows the Mannheim pairs of the examples of Types~1--5. In each panel, the green segments join $\beta(t_i)$ to $\beta^*(t_i)$ for twelve equally spaced values $t_i$; since $\beta^*=\beta+\lambda N_\beta$, each of them lies on the line of $N_\beta=-B_{\beta^*}$, as in the definition of Mannheim pair. We used $\lambda=1$ in Types~1 and~2 and $\lambda=-1$ in Types~3, 4 and~5 (in Types~2 and~4 any $\lambda>0$, respectively $\lambda<0$, would do). In all types the parameter intervals stay slightly away from the endpoints of the domains, where $\beta$ or $\beta^*$ stops being a regular curve ($\frac{ds}{dt}=0$ or $\frac{ds^*}{dt}=0$) or $\kappa_{\beta^*}$ vanishes; in Types~3 and~5, whose domains are unbounded, we also cut the interval on the right (at $t=2.6$ and $t=2.1$), because the coordinates grow like $\cosh t$.

The drawings use the Euclidean structure of $\mathbb{R}^3$ only to represent the curves (axes not to scale), so the angles seen in the figure are not the Lorentzian ones. For instance, since $\langle N_\beta,N_\beta\rangle=\pm1$, each green segment has Lorentzian norm $\sqrt{|\langle\beta^*-\beta,\beta^*-\beta\rangle|}=|\lambda|$ (it is spacelike in Types~1, 2, 4, 5 and timelike in Type~3), but its Euclidean length varies along the curve.

\section{The Mannheim pair condition}\label{sec:mannheim_partner}

Our goal is to choose a convenient parameter $x$ and find functions $\Theta^{*}(x)$ and $\Phi^{*}(x)$ such that $\beta^{*}(x)=\alpha^{*}(s^{*}(x))$ satisfies the Mannheim pair condition and the resulting integrals are elementary.

As in Section~\ref{sec:angular}, $\alpha^{*}$ falls in one of three cases, according to which vector of its Frenet frame is timelike. In each of them the Mannheim pair equation \eqref{eq mannheim partner}, written with respect to the arc length $s^{*}$, takes a different form:
\begin{align}
\text{Case 1 } (\varepsilon_{1*}=-1,\ \varepsilon_{2*}=1):\quad & \frac{d\tau_{\alpha^{*}}}{ds^{*}}=\frac{\kappa_{\alpha^{*}}}{\lambda}\big(1-\lambda^2\tau_{\alpha^{*}}^2\big),\label{eq:pair_case1}\\
\text{Case 2 } (\varepsilon_{1*}=1,\ \varepsilon_{2*}=-1):\quad & \frac{d\tau_{\alpha^{*}}}{ds^{*}}=-\frac{\kappa_{\alpha^{*}}}{\lambda}\big(1-\lambda^2\tau_{\alpha^{*}}^2\big),\label{eq:pair_case2}\\
\text{Case 3 } (\varepsilon_{1*}=1,\ \varepsilon_{2*}=1):\quad & \frac{d\tau_{\alpha^{*}}}{ds^{*}}=\frac{\kappa_{\alpha^{*}}}{\lambda}\big(1+\lambda^2\tau_{\alpha^{*}}^2\big).\label{eq:pair_case3}
\end{align}
By Corollary~\ref{corolario par s e par s*}, the five types of Mannheim pairs correspond to these cases as follows:
\begin{center}
\begin{tabular}{c|c|c|c}
Case & $(T_{\alpha^*},N_{\alpha^*},B_{\alpha^*})$ & $|\lambda\tau_{\alpha^*}|<1$ & $|\lambda\tau_{\alpha^*}|>1$\\ \hline
1 & timelike, spacelike, spacelike & Type 2 & Type 1\\
2 & spacelike, timelike, spacelike & Type 5 & Type 4\\
3 & spacelike, spacelike, timelike & Type 3 & ---
\end{tabular}
\end{center}
In all cases we assume $\dfrac{ds^{*}}{dx}>0$ and ${\Phi^{*}}'(x)>0$ (the case ${\Phi^{*}}'<0$ is analogous); primes denote derivatives with respect to $x$.

\subsection{Case \texorpdfstring{$T_{\alpha^{*}}$}{T*} timelike}

Suppose $\alpha^{*}$ is not a helix with timelike axis, $T_{\alpha^{*}}$ is future pointing and $T_{\alpha^{*}}\neq\pm e_3$, so that $\Phi^{*}\in(-\frac\pi2,\frac\pi2)\setminus\{0\}$. Define an auxiliary angle $\delta^{*}(x)\in(0,\pi)$ by
\begin{equation}\label{eq:cot_delta_1}
\cot\delta^{*}(x) = \dfrac{{\Theta^{*}}'(x)\sin\Phi^{*}(x)}{{\Phi^{*}}'(x)}.
\end{equation}
Then $\sqrt{{\Theta^{*}}'^2\tan^2\Phi^{*} + {\Phi^{*}}'^2\sec^2\Phi^{*}} = {\Phi^{*}}'\sec\Phi^{*}\csc\delta^{*}$ and
\begin{equation}
\left[ \arctan\Big( \dfrac{{\Phi^{*}}'}{{\Theta^{*}}'\sin\Phi^{*}} \Big)\right]' = {\delta^{*}}'
\end{equation}
(where ${\Theta^{*}}'\neq0$; at the isolated points where ${\Theta^{*}}'=0$ it is understood by continuity). Thus, by \eqref{eq:kappa_x_1} and \eqref{eq:tau_x_1},
\begin{equation}\label{eq:kt_case1}
\kappa_{\alpha^{*}}=\dfrac{{\Phi^{*}}'\sec\Phi^{*}\csc\delta^{*}}{\dfrac{ds^{*}}{dx}},\qquad
\tau_{\alpha^{*}}=\dfrac{\Psi}{\dfrac{ds^{*}}{dx}},\qquad
\Psi:={\Phi^{*}}'\sec\Phi^{*}\csc\Phi^{*}\cot\delta^{*}-{\delta^{*}}'.
\end{equation}
Multiplying \eqref{eq:pair_case1} by $\dfrac{ds^{*}}{dx}$ and using $\kappa_{\alpha^{*}}\dfrac{ds^{*}}{dx}={\Phi^{*}}'\sec\Phi^{*}\csc\delta^{*}$, we get
\[
\frac{\lambda\,\tau_{\alpha^{*}}'}{1-\lambda^2\tau_{\alpha^{*}}^2}={\Phi^{*}}'\sec\Phi^{*}\csc\delta^{*}.
\]
Since $\operatorname{arctanh}$ and $\operatorname{arccoth}$ have the same derivative $\frac1{1-u^2}$, integrating gives
\begin{equation}\label{eq:lamtau_case1}
\lambda\tau_{\alpha^{*}}=\tanh I\ \ (|\lambda\tau_{\alpha^{*}}|<1)\qquad\text{or}\qquad \lambda\tau_{\alpha^{*}}=\coth I\ \ (|\lambda\tau_{\alpha^{*}}|>1),
\end{equation}
where $I(x)=\int{\Phi^{*}}'\sec\Phi^{*}\csc\delta^{*}\,dx+n$. Solving $\tau_{\alpha^{*}}=\Psi/\frac{ds^{*}}{dx}$ for $\frac{ds^{*}}{dx}$,
\begin{equation}\label{eq:ds_case1}
\dfrac{ds^{*}}{dx}=\lambda\coth(I)\,\Psi\quad(|\lambda\tau_{\alpha^{*}}|<1),\qquad
\dfrac{ds^{*}}{dx}=\lambda\tanh(I)\,\Psi\quad(|\lambda\tau_{\alpha^{*}}|>1).
\end{equation}
Finally, $\Theta^{*}$ is recovered from \eqref{eq:cot_delta_1}:
\[
\Theta^{*}(x) = \int {\Phi^{*}}'(x)\csc\Phi^{*}(x)\cot\delta^{*}(x) \, dx+m.
\]
Here $\lambda>0$, because the Mannheim curve $\alpha=\alpha^{*}+\lambda B_{\alpha^{*}}$ has $\kappa_\alpha>0$ (Corollary~\ref{corolario par s e par s*}).

\subsection{Cases \texorpdfstring{$N_{\alpha^{*}}$}{N*} and \texorpdfstring{$B_{\alpha^{*}}$}{B*} timelike}

The computation is the same, using Cases~2 and~3 of Proposition~\ref{thm:arbitrary_param} and the corresponding equation \eqref{eq:pair_case2} or \eqref{eq:pair_case3}.

\emph{Case $N_{\alpha^{*}}$ timelike.} Here ${\Phi^{*}}'^2>{\Theta^{*}}'^2\cosh^2\Phi^{*}$, so we may define $\delta^{*}(x)\in\mathbb{R}$ by
\[
\tanh\delta^{*}=\frac{{\Theta^{*}}'\cosh\Phi^{*}}{{\Phi^{*}}'}.
\]
Then
\begin{gather*}
\sqrt{{\Phi^{*}}'^2-{\Theta^{*}}'^2\cosh^2\Phi^{*}}={\Phi^{*}}'\operatorname{sech}\delta^{*},\qquad
{\Theta^{*}}'\sinh\Phi^{*}={\Phi^{*}}'\tanh\Phi^{*}\tanh\delta^{*},\\
\Big[\operatorname{arctanh}\Big(\frac{{\Theta^{*}}'\cosh\Phi^{*}}{{\Phi^{*}}'}\Big)\Big]'={\delta^{*}}',
\end{gather*}
so by \eqref{eq:kappa_x_2} and \eqref{eq:tau_x_2},
\[
\kappa_{\alpha^{*}}\dfrac{ds^{*}}{dx}={\Phi^{*}}'\operatorname{sech}\delta^{*},\qquad
\tau_{\alpha^{*}}\dfrac{ds^{*}}{dx}=-\Psi,\qquad \Psi:={\Phi^{*}}'\tanh\Phi^{*}\tanh\delta^{*}+{\delta^{*}}'.
\]
Multiplying \eqref{eq:pair_case2} by $\frac{ds^{*}}{dx}$ gives $\dfrac{-\lambda\tau_{\alpha^{*}}'}{1-\lambda^2\tau_{\alpha^{*}}^2}={\Phi^{*}}'\operatorname{sech}\delta^{*}$. Hence, with $I=\int{\Phi^{*}}'\operatorname{sech}\delta^{*}\,dx+n$,
\[
\lambda\tau_{\alpha^{*}}=-\tanh I\ \ \text{or}\ \ -\coth I,\qquad
\frac{ds^{*}}{dx}=\lambda\coth(I)\,\Psi\ \ \text{or}\ \ \lambda\tanh(I)\,\Psi,\ \text{respectively},
\]
$\Theta^{*}=\int{\Phi^{*}}'\operatorname{sech}\Phi^{*}\tanh\delta^{*}\,dx+m$, and $\lambda<0$.

\emph{Case $B_{\alpha^{*}}$ timelike.} Here ${\Theta^{*}}'^2\cosh^2\Phi^{*}>{\Phi^{*}}'^2$, so ${\Theta^{*}}'$ has constant sign; assume ${\Theta^{*}}'>0$ (the other case is analogous) and define $\delta^{*}(x)>0$ by
\[
\coth\delta^{*}=\frac{{\Theta^{*}}'\cosh\Phi^{*}}{{\Phi^{*}}'}.
\]
Then
\begin{gather*}
\sqrt{{\Theta^{*}}'^2\cosh^2\Phi^{*}-{\Phi^{*}}'^2}={\Phi^{*}}'\operatorname{csch}\delta^{*},\qquad
{\Theta^{*}}'\sinh\Phi^{*}={\Phi^{*}}'\tanh\Phi^{*}\coth\delta^{*},\\
\Big[\operatorname{arctanh}\Big(\frac{{\Phi^{*}}'}{{\Theta^{*}}'\cosh\Phi^{*}}\Big)\Big]'={\delta^{*}}',
\end{gather*}
so by \eqref{eq:kappa_x_3} and \eqref{eq:tau_x_3},
\[
\kappa_{\alpha^{*}}\dfrac{ds^{*}}{dx}={\Phi^{*}}'\operatorname{csch}\delta^{*},\qquad
\tau_{\alpha^{*}}\dfrac{ds^{*}}{dx}=\Psi,\qquad \Psi:={\Phi^{*}}'\tanh\Phi^{*}\coth\delta^{*}+{\delta^{*}}'.
\]
Multiplying \eqref{eq:pair_case3} by $\frac{ds^{*}}{dx}$ gives $\dfrac{\lambda\tau_{\alpha^{*}}'}{1+\lambda^2\tau_{\alpha^{*}}^2}={\Phi^{*}}'\operatorname{csch}\delta^{*}$. Hence, with $I=\int{\Phi^{*}}'\operatorname{csch}\delta^{*}\,dx+n$,
\[
\lambda\tau_{\alpha^{*}}=\tan I,\qquad \frac{ds^{*}}{dx}=\lambda\cot(I)\,\Psi,
\]
$\Theta^{*}=\int{\Phi^{*}}'\operatorname{sech}\Phi^{*}\coth\delta^{*}\,dx+m$, and $\lambda<0$.

\subsection{The general parametrization}

Thus, we can state the following theorem, which summarizes the general parametrization of a Mannheim partner curve obtained by this method.

\begin{theorem}\label{thm:mannheim_partner}
Let $m$, $n$ and $\lambda\neq0$ be real constants, and let $\Phi^{*}(x)$ and $\delta^{*}(x)$ be differentiable functions with ${\Phi^{*}}'(x)>0$. Define $\Theta^{*}$, $\Psi$ and $I$ by one of the following items:
\begin{enumerate}
\item[(1)] ($\lambda>0$, $\Phi^{*}\in(-\frac\pi2,\frac\pi2)\setminus\{0\}$, $\delta^{*}\in(0,\pi)$)
\begin{gather*}
\Theta^{*}=\int{\Phi^{*}}'\csc\Phi^{*}\cot\delta^{*}\,dx+m,\qquad
\Psi={\Phi^{*}}'\sec\Phi^{*}\csc\Phi^{*}\cot\delta^{*}-{\delta^{*}}',\\
I=\int{\Phi^{*}}'\sec\Phi^{*}\csc\delta^{*}\,dx+n,\qquad
\kappa_{\beta^{*}}\dfrac{ds^{*}}{dx}={\Phi^{*}}'\sec\Phi^{*}\csc\delta^{*},\qquad
\tau_{\beta^{*}}\dfrac{ds^{*}}{dx}=\Psi;
\end{gather*}
\item[(2)] ($\lambda<0$, $\delta^{*}\in\mathbb{R}$)
\begin{gather*}
\Theta^{*}=\int{\Phi^{*}}'\operatorname{sech}\Phi^{*}\tanh\delta^{*}\,dx+m,\qquad
\Psi={\Phi^{*}}'\tanh\Phi^{*}\tanh\delta^{*}+{\delta^{*}}',\\
I=\int{\Phi^{*}}'\operatorname{sech}\delta^{*}\,dx+n,\qquad
\kappa_{\beta^{*}}\dfrac{ds^{*}}{dx}={\Phi^{*}}'\operatorname{sech}\delta^{*},\qquad
\tau_{\beta^{*}}\dfrac{ds^{*}}{dx}=-\Psi;
\end{gather*}
\item[(3)] ($\lambda<0$, $\delta^{*}>0$)
\begin{gather*}
\Theta^{*}=\int{\Phi^{*}}'\operatorname{sech}\Phi^{*}\coth\delta^{*}\,dx+m,\qquad
\Psi={\Phi^{*}}'\tanh\Phi^{*}\coth\delta^{*}+{\delta^{*}}',\\
I=\int{\Phi^{*}}'\operatorname{csch}\delta^{*}\,dx+n,\qquad
\kappa_{\beta^{*}}\dfrac{ds^{*}}{dx}={\Phi^{*}}'\operatorname{csch}\delta^{*},\qquad
\tau_{\beta^{*}}\dfrac{ds^{*}}{dx}=\Psi.
\end{gather*}
\end{enumerate}
Let $\dfrac{ds^{*}}{dx}$ be given by
\begin{equation}\label{ds* do par}
\begin{array}{lll}
\text{(1)}\ \ \lambda\coth(I)\,\Psi\ \text{ (Type 2)}\quad\text{or}\quad \lambda\tanh(I)\,\Psi\ \text{ (Type 1)}, \\[4pt]
\text{(2)}\ \ \lambda\coth(I)\,\Psi\ \text{ (Type 5)}\quad\text{or}\quad \lambda\tanh(I)\,\Psi\ \text{ (Type 4)}, \\[4pt]
\text{(3)}\ \ \lambda\cot(I)\,\Psi\ \text{ (Type 3)},
\end{array}
\end{equation}
and assume that it is positive (which requires $I\neq0$, and $I\notin\frac\pi2\mathbb{Z}$ in item~(3)). Then the curve $\beta^{*}$ parametrized by $x$ as
\begin{equation}\label{alpha* do par}
\frac{d\beta^{*}}{dx}(x)=\frac{ds^{*}}{dx}\,T_{\beta^{*}}(x),
\end{equation}
where $T_{\beta^{*}}$ is given by \eqref{eq:T_x_1}, \eqref{eq:T_x_2} or \eqref{eq:T_x_3} (items (1), (2), (3), respectively) with $(\Theta^{*},\Phi^{*})$, is a Mannheim partner curve with the stated curvature and torsion, and $\lambda\tau_{\beta^{*}}=\tanh I$ or $\coth I$ in~(1), $-\tanh I$ or $-\coth I$ in~(2), and $\tan I$ in~(3), according to the choice in \eqref{ds* do par}. Its Mannheim curve is $\beta=\beta^{*}+\lambda B_{\beta^{*}}$.

Conversely, let $\alpha^{*}$ be a Mannheim partner curve with $\kappa_{\alpha^{*}}>0$. Then, near each point, there are a coordinate system and functions $\delta^{*}$, $\Phi^{*}$ with ${\Phi^{*}}'>0$ (the case ${\Phi^{*}}'<0$ being analogous) such that the above relations hold.
\end{theorem}

\begin{proof}
Let $\delta^{*}$, $\Phi^{*}$, $\Theta^{*}$ and $\frac{ds^{*}}{dx}$ be as in the statement. By Proposition~\ref{thm:arbitrary_param} and the definition of $\Psi$, the curvature and torsion of $\beta^{*}$ are those of \eqref{eq:kt_case1} and of the corresponding formulas in the other cases. By construction, $\tau_{\beta^{*}}$ satisfies $\lambda\tau_{\beta^{*}}=\tanh I$, $\coth I$, $-\tanh I$, $-\coth I$ or $\tan I$, so that $\kappa_{\beta^{*}}\frac{ds^{*}}{dx}$ equals $\frac{d}{dx}$ of $\operatorname{arctanh}(\lambda\tau_{\beta^{*}})$ (resp.\ $\operatorname{arccoth}$, $\operatorname{arctan}$, with the sign of the case), which is the Mannheim pair equation \eqref{eq:pair_case1}, \eqref{eq:pair_case2} or \eqref{eq:pair_case3}. By the reciprocal part of Theorem~\ref{theorem caracterization mannheim pair}, $\beta^{*}$ is a Mannheim partner curve, whose Mannheim curve is $\beta=\beta^{*}+\lambda B_{\beta^{*}}$. Conversely, if $\alpha^{*}$ is a Mannheim partner curve with $\kappa_{\alpha^{*}}>0$, then $T_{\alpha^{*}}'\neq0$, so there is a timelike axis $U$ with $\langle T_{\alpha^{*}}',U\rangle\neq0$ at a given point; taking $e_3=U$ we get ${\Phi^{*}}'\neq0$ near that point. Define $\delta^{*}$ by the corresponding relation of the cases above; the computations of this section show that \eqref{eq mannheim partner} is equivalent to the relations of the statement.
\end{proof}

\begin{remark}
We can make one choice: $\Phi^{*}(x)=x$ ($x$ as the angular parameter), and establish the corresponding forms of the equations of a Mannheim partner curve given by this theorem by the function $\delta^{*}(x)$. Some suggestions to obtain other forms of the equations (according to convenience) are to take one of the following possibilities: $\Theta^{*}(x)=x$, $s^{*}(x)=x$, $s^{*}(x)={\kappa_{\alpha^{*}}}^{-1}(x)$, $s^{*}(x)={\tau_{\alpha^{*}}}^{-1}(x)$, $s^{*}(x)={\sigma_{0}}^{-1}(x)$, $s^{*}(x)={\sigma_{1}}^{-1}(x)$, where $\sigma_{0}=\tau_{\alpha^{*}}/\kappa_{\alpha^{*}}$ and $\sigma_{1}$ is the function of Ali and L\'opez \cite{AliLopez} (see Proposition~\ref{prop:slant_helix_L3}). A choice that makes the integral $I$ elementary is to take $\delta^{*}$ such that the integrand is $1$, i.e.\ $I(x)=x+n$: $\sin\delta^{*}={\Phi^{*}}'\sec\Phi^{*}$ in item~(1), $\cosh\delta^{*}={\Phi^{*}}'$ in item~(2) and $\sinh\delta^{*}={\Phi^{*}}'$ in item~(3). This choice requires ${\Phi^{*}}'\le\cos\Phi^{*}$ in item~(1) (so it is not compatible with $\Phi^{*}(x)=x$ there) and ${\Phi^{*}}'\ge1$ in item~(2); it is the one made in the example below.
\end{remark}

\begin{example}[Type 2]\label{ex:partner_type2}
We apply item~(1) of Theorem~\ref{thm:mannheim_partner} with $\lambda=1$, $m=0$, $n=0$ and $\lambda\tau_{\beta^{*}}=\tanh I$.

\textit{Basic functions} ($x\in(0,+\infty)$):
\begin{equation}
\begin{aligned}
\Phi^{*}(x)&=\arccos\!\Big(\dfrac{8}{\cosh x+9}\Big),\\
\Theta^{*}(x)&=3\arctan\!\Big(\sinh\dfrac{x}{2}\Big)-\arctan\!\Big(\dfrac13\sinh\dfrac{x}{2}\Big),\\
\delta^{*}(x)&=\operatorname{arccot}\!\Big(3\operatorname{csch}\dfrac{x}{2}\Big).
\end{aligned}
\end{equation}

\textit{Auxiliary functions:}
\begin{equation}
I(x)=x,\qquad
\Psi(x)=\dfrac32\operatorname{sech}\dfrac x2,\qquad
\dfrac{ds^{*}}{dx}=\dfrac32\coth x\,\operatorname{sech}\dfrac x2 .
\end{equation}
The domain $x>0$ is needed because $I=x$ must not vanish: as $x\to0^{+}$, $\frac{ds^{*}}{dx}\to\infty$ and $\kappa_{\beta^{*}}\to0$.

\textit{Curvature and torsion:}
\begin{equation}
\kappa_{\beta^{*}}(x)=\dfrac23\tanh x\,\cosh\dfrac x2,\qquad
\tau_{\beta^{*}}(x)=\tanh x .
\end{equation}

\textit{Tangent vector:}
\begin{equation}
\dfrac{d\beta^{*}}{dx}(x) = \dfrac32\coth x\,\operatorname{sech}\dfrac x2
\left(
2\operatorname{sech}^2\dfrac x2-\dfrac14\cosh^2\dfrac x2-1,\
2\sinh\dfrac x2\operatorname{sech}^2\dfrac x2,\
\dfrac14\cosh^2\dfrac x2+1\right).
\end{equation}

\textit{Coordinates of the curve $\beta^{*}$ (up to translation):}
\begin{equation}\label{eq:alpha_star_explicit 1}
\beta^{*}(x) = \begin{bmatrix}
-\dfrac34\cosh\dfrac x2+\dfrac92\operatorname{sech}\dfrac x2-\operatorname{sech}^3\dfrac x2+\dfrac98\ln\!\Big(\tanh\dfrac x4\Big) \\[14pt]
3\tanh\dfrac x2+\tanh^3\dfrac x2 \\[14pt]
\dfrac34\cosh\dfrac x2-\dfrac32\operatorname{sech}\dfrac x2+\dfrac{15}{8}\ln\!\Big(\tanh\dfrac x4\Big)
\end{bmatrix}.
\end{equation}

\textit{Frenet frame (tangent $T_{\beta^{*}}$ timelike):}
\begin{equation}
\begin{aligned}
T_{\beta^{*}}(x)&=
\left( 2\operatorname{sech}^2\dfrac x2-\dfrac14\cosh^2\dfrac x2-1,\; 2\sinh\dfrac x2\operatorname{sech}^2\dfrac x2,\; \dfrac14\cosh^2\dfrac x2+1 \right),\\[6pt]
N_{\beta^{*}}(x)&=
\left( -\dfrac14\sinh\dfrac x2\cosh\dfrac x2-2\sinh\dfrac x2\operatorname{sech}^3\dfrac x2,\; 2\operatorname{sech}^3\dfrac x2-\operatorname{sech}\dfrac x2,\; \dfrac14\sinh\dfrac x2\cosh\dfrac x2 \right),\\[6pt]
B_{\beta^{*}}(x)&=
\left( \dfrac34\cosh\dfrac x2-2\operatorname{sech}^3\dfrac x2,\; -\tanh\dfrac x2-2\sinh\dfrac x2\operatorname{sech}^3\dfrac x2,\; -\dfrac34\cosh\dfrac x2 \right).
\end{aligned}
\end{equation}

\textit{Mannheim curve $\beta=\beta^{*}+\lambda B_{\beta^{*}}$ (up to translation):}
\begin{equation}
\beta(x)=\begin{bmatrix}
\dfrac92\operatorname{sech}\dfrac x2-3\operatorname{sech}^3\dfrac x2+\dfrac98\ln\!\Big(\tanh\dfrac x4\Big)\\[14pt]
3\tanh^3\dfrac x2\\[14pt]
-\dfrac32\operatorname{sech}\dfrac x2+\dfrac{15}{8}\ln\!\Big(\tanh\dfrac x4\Big)
\end{bmatrix}.
\end{equation}

\textit{Frenet frame of $\beta$ (tangent $T_\beta$ timelike):}
\begin{equation}
T_{\beta}=\cosh x\,T_{\beta^{*}}-\sinh x\,N_{\beta^{*}},\qquad
N_{\beta}=-B_{\beta^{*}},\qquad
B_{\beta}=-\sinh x\,T_{\beta^{*}}+\cosh x\,N_{\beta^{*}},
\end{equation}
that is,
\begin{equation}
\begin{aligned}
T_{\beta}(x)&=\left( 9-\dfrac94\cosh^2\dfrac x2-6\operatorname{sech}^2\dfrac x2,\; 6\sinh^3\dfrac x2\operatorname{sech}^2\dfrac x2,\; \dfrac94\cosh^2\dfrac x2-1 \right),\\[6pt]
N_{\beta}(x)&=\left( -\dfrac34\cosh\dfrac x2+2\operatorname{sech}^3\dfrac x2,\; \tanh\dfrac x2+2\sinh\dfrac x2\operatorname{sech}^3\dfrac x2,\; \dfrac34\cosh\dfrac x2 \right),\\[6pt]
B_{\beta}(x)&=\Big( \sinh\dfrac x2\Big(\dfrac94\cosh\dfrac x2-8\operatorname{sech}\dfrac x2+2\operatorname{sech}^3\dfrac x2\Big),\\
&\qquad -6\cosh\dfrac x2+9\operatorname{sech}\dfrac x2-2\operatorname{sech}^3\dfrac x2,\; -\dfrac94\sinh\dfrac x2\cosh\dfrac x2 \Big).
\end{aligned}
\end{equation}

\textit{Curvature, torsion and angle of the pair:}
\begin{equation}
\dfrac{ds}{dx}=\dfrac32\operatorname{csch}x\,\operatorname{sech}\dfrac x2,\qquad
\kappa_{\beta}(x)=\sinh^2x,\qquad
\tau_{\beta}(x)=\sinh x\cosh x,
\end{equation}
\begin{equation}
\cosh\varphi_2=-\langle T_\beta,T_{\beta^{*}}\rangle=\cosh x,\quad\text{i.e. }\ \varphi_2=x .
\end{equation}

\textit{Non-helix and non-slant nature:} Since $\kappa_\beta=\sinh^2x$ and $\tau_\beta=\sinh x\cosh x$,
\begin{equation}
\kappa^{\smash[b]{T_{\beta}}}_{g}(x)=\frac{\tau_\beta(x)}{\kappa_\beta(x)}=\coth x,
\end{equation}
which is clearly non-constant. By Lancret's theorem in $\mathbb{L}^3$ \cite[Theorem~2.12]{Lopez2014}, $\beta$ is not a helix.

There do not exist $\textbf{v}=(\bm{v}_{1},\bm{v}_{2},\bm{v}_{3})\in\mathbb{R}^3\setminus\{0\}$ and $\bm{d}\in\mathbb{R}$ such that $\langle N_{\beta}(x),\textbf{v}\rangle = \bm{d}$ for all $x$.

In fact, on the contrary, writing $u=x/2$,
\[
\bm{v}_{1}\Big(-\tfrac34\cosh u+2\operatorname{sech}^3u\Big)+\bm{v}_{2}\Big(\tanh u+2\sinh u\operatorname{sech}^3u\Big)-\bm{v}_{3}\,\tfrac34\cosh u=\bm{d}.
\]
Letting $u\to+\infty$, the terms in $\cosh u$ force $\bm{v}_3=-\bm{v}_1$, and then the left-hand side tends to $\bm{v}_2$, so $\bm{v}_2=\bm{d}$. The equation becomes
\[
2\bm{v}_1\operatorname{sech}^3u+\bm{d}\,\big(\tanh u-1+2\sinh u\operatorname{sech}^3u\big)=0,
\]
and since $\tanh u-1+2\sinh u\operatorname{sech}^3u\sim6e^{-2u}$ while $\operatorname{sech}^3u=O(e^{-3u})$, we get $\bm{d}=0$, hence $\bm{v}_1=\bm{v}_2=\bm{v}_3=0$, a contradiction. Then, by definition, $\beta$ is not a slant helix.

Moreover, $\dfrac{\tau_{\beta^*}}{\kappa_{\beta^*}}=\dfrac32\operatorname{sech}\dfrac x2$ is non-constant, so $\beta^*$ is not a helix either.

\textit{Illustration:}
\begin{figure}[htbp]
\centering
\includegraphics[width=1.0\textwidth]{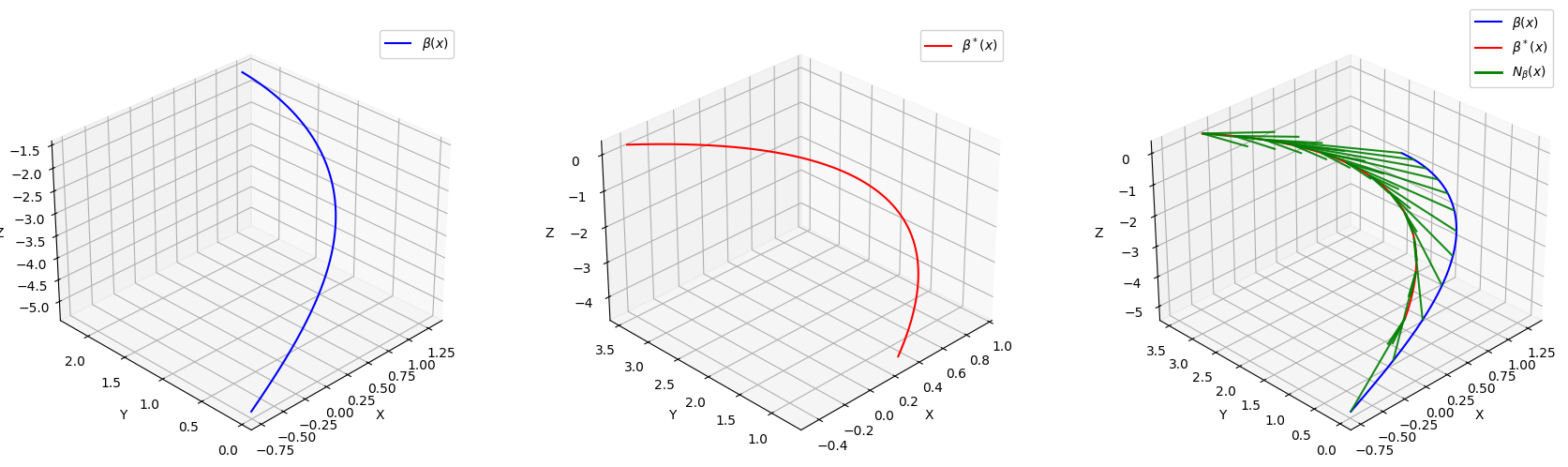}
\caption{Mannheim pair of Example~\ref{ex:partner_type2} for $\lambda=1$ and $x\in[0.5,3]$: the Mannheim curve $\beta$ (blue), its partner $\beta^*$ (red), and both curves together with the principal normal $N_\beta$ (green arrows).}
\label{fig:partner_example}
\end{figure}

In Figure~\ref{fig:partner_example}, the green arrows are drawn at twelve equally spaced values $x_i$. Since $\lambda=1$ and $\beta^*=\beta+\lambda N_\beta$, each arrow starts at $\beta(x_i)$ and ends at $\beta^*(x_i)$, and it lies on the line of $N_\beta=-B_{\beta^*}$, as in the definition of Mannheim pair. The interval $[0.5,3]$ was chosen because $\ln\tanh\frac x4\to-\infty$ as $x\to0^+$, and because $\beta^*$ and $N_\beta$ grow like $\cosh\frac x2$ while $\beta\to(0,3,0)$ as $x\to+\infty$, so for large $x$ the curve $\beta$ would be tiny compared to $\beta^*$. The drawing uses the Euclidean structure of $\mathbb{R}^3$ only to represent the curves (axes not to scale), so the angles seen in the figure are not the Lorentzian ones; in particular, $N_\beta$ does not look orthogonal to $T_\beta$.

\end{example}

\section*{Declarations}
\noindent\textbf{Funding.} There is no funding for this work.\par\smallskip
\noindent\textbf{Data availability.} Not applicable.\par\smallskip
\noindent\textbf{Competing interests.} The authors declare that they have no competing interests.\par\smallskip
\noindent\textbf{Author contributions.} All authors contributed equally to the writing of this paper. All authors read and approved the final manuscript.

\end{document}